\documentclass{siamart1116}

\usepackage{amsfonts}
\usepackage{mathtools}
\usepackage{enumerate}
\usepackage{caption}
\usepackage[normalem]{ulem}
\usepackage{dsfont}
\usepackage{accents}
\usepackage{cite}
\usepackage{algorithm}
\usepackage{algpseudocode}
\usepackage{todonotes}
\usepackage{float}
\usepackage{placeins}
\usepackage{standalone}

\ifpdf
\DeclareGraphicsExtensions{.eps,.pdf,.png,.jpg}
\else
\DeclareGraphicsExtensions{.eps}
\fi

\newcommand{\iu}{{i\mkern1mu}}

\providecommand{\Div}{\operatorname{div}}          

\providecommand{\Det}{\operatorname{det}}                    

\newcommand{\Va}{{\mathbf{a}}}

\newcommand{\Ve}{{\mathbf{e}}}
\newcommand{\Vf}{{\mathbf{f}}}
\newcommand{\Vg}{{\mathbf{g}}}
\newcommand{\Vh}{{\mathbf{h}}}

\newcommand{\Vk}{{\mathbf{k}}}

\newcommand{\Vp}{{\mathbf{p}}}

\newcommand{\Vu}{{\mathbf{u}}}
\newcommand{\Vv}{{\mathbf{v}}}

\newcommand{\Vx}{{\mathbf{x}}}
\newcommand{\Vy}{{\mathbf{y}}}
\newcommand{\Vz}{{\mathbf{z}}}

\newcommand{\VD}{{\mathbf{D}}}
\newcommand{\VE}{{\mathbf{E}}}

\newcommand{\VI}{{\mathbf{I}}}

\newcommand{\VR}{{\mathbf{R}}}
\newcommand{\VS}{{\mathbf{S}}}
\newcommand{\VT}{{\mathbf{T}}}

\newcommand{\VV}{{\mathbf{V}}}

\newcommand{\VX}{{\mathbf{X}}}

\newcommand{\nubf}{\boldsymbol{\nu}}

\newcommand{\phibf}{\boldsymbol{\phi}}
\newcommand{\varphibf}{\boldsymbol{\varphi}}

\newcommand{\Ksf}{\mathsf{K}}

\newcommand{\ksf}{\mathsf{k}}

\providecommand{\Ca}{{\cal A}}
\providecommand{\Cb}{{\cal B}}

\providecommand{\Ce}{{\cal E}}

\providecommand{\Ch}{{\cal H}}
\providecommand{\Ci}{{\cal I}}

\providecommand{\Cp}{{\cal P}}

\providecommand{\Cx}{{\cal X}}

\newcommand{\overbar}[1]{\mkern 1.5mu\overline{\mkern-1.5mu#1\mkern-1.5mu}\mkern 1.5mu}

\DeclareMathOperator{\id}{\mathbf{id} }

\newcommand{\beqn}{\begin{eqnarray*}}
\newcommand{\eeqn}{\end{eqnarray*}}

\newcommand{\ben}{\begin{equation}}
\newcommand{\een}{\end{equation}}

\newcommand{\beq}{\begin{eqnarray}}
\newcommand{\eeq}{\end{eqnarray}}

\newcommand{\benn}{\begin{equation*}}
\newcommand{\eenn}{\end{equation*}}

\newcommand{\defeq}{\vcentcolon=}

\newcommand{\weakto}{\rightharpoonup}

\newcommand{\Dr}[1]{\ensuremath{\operatorname{D}\!{#1}}}
\newcommand{\dr}[1]{\ensuremath{\operatorname{d}\mathbf{\!{#1}}}}
\newcommand{\sym}{\mathrm{sym}}
\newcommand{\asym}{\mathrm{asym}}
\DeclareMathOperator{\tr}{{tr}}

\newcommand{\CR}{\mathrm{CR}}

\newtheorem{remark}[theorem]{Remark}

\newtheorem{problem}[theorem]{Problem}

\numberwithin{theorem}{section}

\newcommand{\TheTitle}{Shape optimisation with nearly conformal transformations}
\newcommand{\TheShortTitle}{Shape optimisation with nearly conformal transformations}
\newcommand{\TheAuthors}{K.~Sturm, F.~Wechsung and J.~A.~Iglesias}
\headers{\TheShortTitle}{\TheAuthors}

\title{{\TheTitle}\thanks{Submitted to the editors \today.
\funding{This publication is based on work partially supported by the EPSRC Centre For Doctoral Training in Industrially Focused Mathematical Modelling (EP/L015803/1) as well as the EPSRC Grant EP/K030930/1.}
}}

\author{Jos\'e A. Iglesias\thanks{Radon Institute for Computational and Applied Mathematics, Austrian Academy of Sciences, Altenberger Strasse 69, 4040 Linz, Austria, \email{jose.iglesias@ricam.oeaw.ac.at}}
  \and
  Kevin Sturm
 \thanks{Radon Institute for Computational and Applied Mathematics, Austrian Academy of Sciences, Altenberger Strasse 69, 4040 Linz, Austria, \email{kevin.sturm@oeaw.ac.at}}
 \and Florian Wechsung
   \thanks{Mathematical Institute, University of Oxford, Andrew Wiles Building, Radcliffe Observatory Quarter, Woodstock Road, Oxford, OX2 6GG. \email{wechsung@maths.ac.ox.uk}}
}

\usepackage{amsopn}

\ifpdf
\hypersetup{
  pdftitle={\TheTitle},
  pdfauthor={\TheAuthors}
}
\fi

\begin{document}

\maketitle

\begin{abstract}
    In shape optimisation it is desirable to obtain deformations of a given mesh without negative impact on the mesh quality.
We propose a new algorithm using least square formulations of the Cauchy-Riemann equations. Our method allows to deform meshes in a nearly conformal way and thus approximately preserves the angles of triangles during the optimisation process.  The performance of our methodology is shown by applying our method to some unconstrained shape functions and a constrained  Stokes shape optimisation problem. 
\end{abstract}



\section{Introduction}\label{sec:1}
Shape optimisation is concerned with the minimisation of \emph{shape functions} $J$ defined on subsets $\Ca \subset \wp(\VR^d)$  of the power set of $\VR^d$. The set $\Ca$ is called the \emph{admissible set}  and its elements are referred to as \emph{shapes}. 

In many engineering applications shape optimisation techniques can be applied to enhance the behaviour of a physical system.  For instance it can be used to optimise physical quantities such as the drag of an airfoil \cite{ESSI2007,SCSCILGA2011,PhDSchmidtShape} or car \cite{KAetal2016}. 
Most shape functions originating from engineering applications are constrained by at least one partial differential equation (PDE) and therefore this type of shape optimisation problem is referred to as \emph{PDE constrained shape optimisation} problem. A generic PDE constrained shape optimisation problem is of the form
\begin{equation}
    \min_{\Omega\in \Ca} J(\Omega,u) \quad \text{ subject to }  \Ce_\Omega(u)=0,
\end{equation}
where $\Ce_\Omega:\VE(\Omega) \to (\VE(\Omega))^\ast$ represents the PDE constraint defined on a Banach space $\VE(\Omega)$ with values in its dual $(\VE(\Omega))^\ast$ \cite{SOZO92,DEZO11}. 

Shapes can be represented in different ways, e.g. by describing its boundary by level sets, polynomials, NURBS, radial basis functions, polygons and others \cite[Chapter 2]{DEZO11}.
We employ the \emph{method of mappings}.
Given an initial domain $\Omega_0$ we search for deformations $\VT(\Omega)$, i.e. we optimise over 
\begin{equation}
        \{\VT:\Omega_0 \to\VR^d : \VT = \id +\VX \text{ for } \VX\in\VV,\; \VT \text{ a homeomorphism onto its image}\}
\end{equation}
where $\VV$ is a certain vector space of functions.

Often the PDE constraint $\Ce_\Omega $ is discretised using the finite element method.
In such cases the domain $\Omega$ is replaced by an approximation $\Omega^h$ that consists of a regular triangulation.
This gives rise to a natural parametrisation of the shape by a finite element mesh/grid.
Then the optimisation process successively deforms the initial domain $\Omega^h$  by  applying transformations to the grid points.

In order to obtain an accurate solution to the PDE constraint, the geometry of the mesh is crucial; in many cases a mesh consisting of triangles / tetrahedra that are as close to unilateral as possible is desirable.
However, in general there is no guarantee that the deformed mesh $(\id + \VX)(\Omega_0^h) = \{\Vx+\VX(\Vx): \Vx\in \Omega_0^h\}$ is still regular, i.e. that cells have not become highly stretched or even overlap.
Therefore if remeshing is impractical, not possible or expensive it is important to find 'tame' deformation fields $\VX$ that deform the mesh 'as well as possible' while decreasing the cost function $J$ sufficiently well.

In \cite{schmidt2014twostage} a two stage approach is introduced: a 'good' deformation is calculated on the boundary first and then extended into the volume by solving a convection-diffusion equation.
In several papers the good properties of the linear elasticity equations have been observed for mesh deformation \cite{dwight2009robust, schulz2016computational}.
Here the shape derivative is essentially viewed as a force acting on an elastic medium.

The approach presented in this work is based on conformal mappings; this is the family of mappings that preserve angles.
This is a desirable property as this implies that mesh properties like isolateral elements remain constant under a conformal mapping.
We consider the two dimensional case as then conformality can be achieved by enforcing that the deformation satisfies the Cauchy-Riemann equations.
However, as the Cauchy-Riemann equations do not admit solutions for any given boundary data we consider various least square solutions instead. 
We briefly outline a possible extension to three dimensions at the end of the paper.

It is worth mentioning that Grid/Mesh deformations itself can be used for mesh generation; \cite{ST93}. 
One way of doing this is by solving an appropriate boundary value problem with prescribed Dirichlet conditions \cite{HE03} and \cite[pp.~117]{ST93}. 
Conformal variational grid generation is considered for instance in \cite{ZOHO89,SMER87}.

\subsection*{Structure of the paper}
In Section~\ref{sec:2} we gather basics on the Cauchy-Riemann equations and associated boundary value problems.
We recall the Riemannian mapping theorem and discuss when conformal mappings between shapes exist.
 
In Section~\ref{sec:3}, we introduce least square formulations of the Cauchy-Riemann equations and associated minimisation problems. 
We will prove uniqueness and existence of minimisers and we discuss their behaviour as conformality is enforced more strongly.

In Section~\ref{sec:4}, we discuss least square formulations in more regular spaces, namely reproducing kernel Hilbert spaces.
 
In Section~\ref{sec:5}, we describe the application of the developed methods to shape optimisation problems.
We do this by formulating shape optimisation problems on a linear space on which classical steepest descent or quasi Newton methods can be employed.
 
In Section~\ref{sec:6}, we discuss three numerical examples: a levelset example for which a conformal mapping between the initial and the final shape exists; a second example in which it is known that no such mapping exists and lastly we present the classical example of energy minimisation in Stokes flow.

\section{Conformal mappings}\label{sec:2}
\subsection{Classical Cauchy-Riemann equation}
Let $\Omega\subset \VR^2$ be a bounded domain. 
\begin{definition}
    We call a vector field $\VX = (X_1,X_2)\in [C^1(\Omega)]^2$ holomorphic if it satisfies the Cauchy-Riemann equations on $\Omega$,  
	\ben\label{eq:cauchy_riemann}
	\begin{split}
		\partial_x X_1 & =\partial_y X_2,\\
		\partial_y X_1 & =-\partial_x X_2.
	\end{split}
	\een
\end{definition}
Introducing the operator
\ben
\Cb := \begin{pmatrix}
	-\partial_x & \partial_y\\
	\partial_y & \partial_x
\end{pmatrix}, \;[C^1(\Omega)]^2 \to [C^0(\Omega)]^2,
\een
we can write \eqref{eq:cauchy_riemann} in the compact form $\Cb\VX =0$.  An injective and holomorphic mapping $F:\Omega \to \VR^2$ is called \emph{conformal}.
\begin{remark}
A vector field $\VX$ is holomorphic if and only if  $\hat\VX := X_1 + \iu X_2$ is complex differentiable on $\Omega$, where $\iu$ denotes the 
imaginary number defined by $\iu^2=-1$.
\end{remark}

\begin{definition}
We call the operator $\Cb$  the Cauchy-Riemann (CR) operator. 
\end{definition}
\begin{remark}\label{rem:holo_harmonic}
    Observe that the twofold application of the Cauchy-Riemann operator yields the vectorial Laplace operator $\Cb^2 := \Cb \circ \Cb = \VI_2\Delta$, where $\VI_2\in \VR^{2\times2}$ denotes the identity matrix in $\VR^2$.
    Therefore every holomorphic mappings is also harmonic, that is, the Laplacian of the function vanishes.  
\end{remark}
\begin{definition}[CR Dirichlet BVP]
    Given $\Vh\in [C(\partial\Omega)]^2$ the Cauchy-Riemann Dirichlet boundary value problem reads: find $\VX\in [C^1(\Omega)]^2\cap [C(\overline\Omega)]^2$ such that
	\begin{subequations}
		\label{eq:BVP_CR}
		\begin{align}
		\Cb \VX &  = 0 \quad \text{ in } \Omega, \\
		\VX& = \Vh \quad \text{ on } \partial\Omega.
		\end{align}
	\end{subequations}
\end{definition}

\begin{definition}[CR Neumann BVP]
    Given  $\Vg\in [C(\partial\Omega)]^2$ the Cauchy-Riemann Neumann boundary value problem reads: find $\VX\in [C^1(\overline\Omega)]^2$ such that
	\begin{subequations}
		\label{eq:BVP_CR_Neumann}
		\begin{align}
		\Cb \VX &  = 0 \quad \text{ in } \Omega, \\
		\partial\VX\nubf& = \Vg \quad \text{ on } \partial\Omega,
		\end{align}
		where $\nubf$ denotes the outward-pointing normal vector field along $\partial \Omega$. 
	\end{subequations}
\end{definition}

The boundary value problems \eqref{eq:BVP_CR} and \eqref{eq:BVP_CR_Neumann} do not have solutions in general. However, when the data $\Vh$ in the Dirichlet case and $\Vg$ in the Neumann case satisfy certain compatibility conditions then \eqref{eq:BVP_CR} and \eqref{eq:BVP_CR_Neumann} admit solutions, respectively.  In \cite{BE05} a closed formula for the solution of \eqref{eq:BVP_CR} and \eqref{eq:BVP_CR_Neumann} on the unit disc is given and necessary and sufficient conditions are provided.

\subsection{Bi-holomorphic transformations}
In this subsection we briefly discuss existence and non-existence of 
holomorphic mappings between two simply and doubly connected domains of $\VR^2$, respectively. 
\begin{definition}
Let $\Omega\subset \VR^2$ be open.  We call $\Vf:\Omega\to \Vf(\Omega)$ a bi-holomorphic mapping if $\Vf$ is holomorphic and injective with holomorphic inverse. 
\end{definition}

The set of bi-holomorphic functions may seem to be small, however, the Riemannian mapping theorem allows us to map every simply connected domain $\Omega\subset \VR^2$ to the unit disc
 $\VD\subset \VR^2$ via a bi-holomorphic map. The precise statement of the Riemann mapping theorem is as follows; see \cite[Thm.~4.2, p.~160]{CO78}. 

 \begin{theorem}
 Let $\Omega\subset \VR^2$ be a simply connected domain that is not the whole plane and let $\Va\in \Omega$. Then there exists a bi-holomorphic map $\Vf$ from $\Omega$ onto the unit disc $\VD\subset \VR^2$, such that $\Vf(\Va)=0$ and $\Vf'(\Va)>0$. 
\end{theorem}

%
%
%
%

\section{Nearly holomorphic shape gradients}\label{sec:3}

\subsection{Nearly holomorphic shape gradients}\label{sec:NCSG}
 Given a Lipschitz vector field $\VX\in [C^{0,1}(\overbar\Omega)]^2$, we denote by $\VT_t$ the perturbation of the identity $\VT_t :=\id + t\VX:\Omega \to \VT_t(\Omega) $. Using the Lipschitz property of $\VX$ we see that for all $\Vx,\Vy\in \Omega$, we have
 \ben
 |\Vx-\Vy| \le t\mathrm{Lip}(\VX)|\Vx-\Vy| +  |\VT_t(\Vx)-\VT_t(\Vy)|
 \een
 and therefore $T_t$ is an injective, Lipschitz continuous function with Lipschitz continuous inverse for all $|t| < \tau_\VX := 1/(\mathrm{Lip}(\VX))$.

\begin{definition}\label{def1}
    The shape derivative of $J$ at an admissible shape $\Omega$ in direction $\VX\in  [C^{0,1}(\overbar\Omega)]^d$ is defined by
    \begin{equation}
        \Dr J(\Omega)(\VX):= \lim_{t \searrow 0}\frac{J(\VT_t(\Omega))-J(\Omega)}{t}.
    \end{equation}
\end{definition}
Throughout we consider a shape function $J$ whose shape derivative exists for all $\VX\in [C^1(\Omega)]^2$ and is of the typical tensor form (cf. \cite{LaSt16})
\begin{equation}
\label{eq:shape_tensor_form}
    \Dr J(\Omega)(\VX) = \int_{\Omega} \VS_1:\partial \VX + \VS_0\cdot \VX \dr x,
\end{equation}
where $\VS_1\in [L^2(\Omega)]^{2\times2}$ and $\VS_0\in [L^2(\Omega)]^2$ are given functions. It follows from \eqref{eq:shape_tensor_form} that 
$\Dr J(\Omega)$ is a well-defined and continuous functional on $[H^1(\Omega)]^2$.

Let $(H,(\cdot,\cdot)_H)$ be a Hilbert space that is continuously embedded into $[H^1(\Omega)]^2$.
The \textit{shape gradient} of $J$ at $\Omega$ with respect to $(H,(\cdot,\cdot)_H)$ is then defined as the unique element $\nabla J (\Omega)\in H$ that satisfies 
\begin{equation}
    (\nabla J(\Omega), \Vv)_H = \Dr J(\Omega)(\Vv) \quad \text{ for all }\Vv \in H.
\end{equation}
There is a close relation between the shape gradient and steepest descent directions. Indeed it can be shown (see \cite{EigelSturm17}) that the negative normalized shape gradient $-\nabla J(\Omega)/\|\nabla J(\Omega)\|_{H}$ solves the minimisation problem
\begin{equation}
    \min_{\substack{\Vu\in H,\\ \|\Vu\|_{H}=1}} \Dr J(\Omega)(\Vu)
\end{equation}
and hence the negative shape gradient is actually the steepest descent direction. 
In addition let us notice that the shape gradient itself 
is the minimiser of 
\begin{equation}\label{eq:min_H_P}
    \min_{\Vu\in H} \frac12\|\Vu\|_H^2 - \Dr J(\Omega)(\Vu).
\end{equation}
In order to retain mesh quality, we propose to use shape gradients that satisfy the Cauchy-Riemann equations as well as possible; one way to achieve this is by 
replacing \eqref{eq:min_H_P} by the following optimisation problem 
\begin{equation} \label{eq:grad-optim-problem-with-strict-cr}
    \min_{\substack{\Vu\in H,\\ \Cb\Vu = 0}} \frac12\|\Vu\|_H^2 - \Dr J(\Omega)(\Vu), 
\end{equation}
where $\Cb$ corresponds to the Cauchy-Riemann equations using weak partial derivatives.
We will see later that the minimisation problem \eqref{eq:grad-optim-problem-with-strict-cr} is related to the Cauchy-Riemann Neumann  boundary value problem. But we already know that the Neumann boundary value problem does not always have solutions because the boundary data has to satisfy compatibility conditions.
Therefore it might happen that the solution of \eqref{eq:grad-optim-problem-with-strict-cr} does not yield a descent direction, even when $\Omega$ is not a stationary point.

In order to enlarge the shape deformation space, but still retain some conformality we propose to enforce the conformality constraint weakly by adding a penalty term. We will see later that in certain cases this approach yields strictly conformal mappings. 

Thus we consider the following regularised version of \eqref{eq:grad-optim-problem-with-strict-cr}, 
\begin{equation}\label{eq:min_abstract}
    \min_{ \Vu\in H} \frac12\left(\frac{1}{\alpha} \|\Cb \Vu\|_{\Cp}^2 + \|\Vu\|_H^2\right) - \Dr J(\Omega)(\Vu), \quad \alpha >0,
\end{equation}
where $\|\cdot\|_\Cp:\Cp\to\VR$ is a norm on a Hilbert space $\Cp$ that satisfies $\Cb(H)\subset \Cp$. Assuming that $\Cb:H\to \Cp$ is a continuous operator it is straightforward to show that \eqref{eq:min_abstract} has a unique minimiser. In addition 
 this minimiser is then the shape gradient of $J$ at $\Omega$  with respect to $H$ equipped with the inner-product 
\begin{equation}
    (\Vu, \Vv)_{\CR(\alpha)+H} \defeq \frac1\alpha(\Cb\Vu,\Cb\Vv)_{\Cp} + (\Vu, \Vv)_H.
\end{equation}
Subsequently we study the properties of the minimisers of \eqref{eq:min_abstract} for special choices of $H$ and $\Cp$. In particular we examine their behaviour for $\alpha\to0$.
\subsection{Nearly holomorphic shape gradients with $H=H^1$ and $\Cp=L^2$ }
A space of deformations that is frequently used is the space $H:=[\mathring{H}^1(\Omega)]^2$  consisting of all functions $\Vu\in [H^1(\Omega)]^2$ with mean zero, that is, $\int_\Omega \Vu\dr x = 0$.
This space becomes a Hilbert space when equipped with the inner product 
\begin{equation}
    (\Vu, \Vv)_{\mathring{H}^1} = (\partial \Vu, \partial\Vv)_{[L^2]^2}.
\end{equation}
In the setting of the previous section we let $\Cp = [L^2(\Omega)]^2$, define the norm $\|\Vv\|_{\Cp}:=\|\Vv\|_{[L^2]^2}$ and set $\Cb:[H^1(\Omega)]^2\to[L^2(\Omega)]^2$ to be the Cauchy-Riemann operator using weak derivatives.
Assuming $\VS_1$ from \eqref{eq:shape_tensor_form} is in $[H^1(\Omega)]^{2\times2}$ it can be shown (see \cite{LaSt16}) that the shape derivative \eqref{eq:shape_tensor_form} can be written as a boundary integral  
\begin{equation}
\Dr J(\Omega)(\Vv) = \int_{\partial\Omega} \Vg \cdot \Vv \dr s \quad \text{ for all } \Vv\in [\mathring{H}^1(\Omega)]^2,
\end{equation}
where $\Vg := (\VS_1\nubf\cdot \nubf)\cdot \nubf \in [H^{1/2}(\partial\Omega)]^2 \subset [L^2(\partial\Omega)]^2$.
The results below hold as well for other inner products and under the assumption that $\Dr J(\Omega)$ is merely an element of the dual space of $H$, i.e., they are compatible with using the volume formula of the shape derivative. 
\begin{problem}
For $\Vg\in [L^2(\partial\Omega)]^2$, 
    we study the following relaxation of \eqref{eq:grad-optim-problem-with-strict-cr}:
    find a minimiser $\Vu_\alpha\in [\mathring{H}^1(\Omega)]^2$ of 
    \begin{equation}\label{eq:min_L2_least_square}
        \min_{\Vu\in [\mathring{H}^1(\Omega)]^2} \frac12 \left(\|\Cb\Vu\|_{[L^2]^2}^2   + \alpha \|\partial\Vu\|_{[L^2]^2}^2\right)  - \alpha \int_{\partial \Omega} \Vg\cdot \Vu\; \dr s, \quad \alpha >0.
    \end{equation}
\end{problem}

\begin{lemma}\label{lem:existence_min_neumann}
    There exists a unique minimiser $\Vu_\alpha\in \mathring{H^1}(\Omega,\VR^2)$ of \eqref{eq:min_L2_least_square} and $\Vu_\alpha$ satisfies 
\begin{equation}
    (\Cb \Vu_\alpha,  \Cb \Vv)_{[L^2]^2} + \alpha (\partial \Vu_\alpha,\partial\Vv)_{[L^2]^{2\times2}} = \int_{\partial\Omega} \Vv\cdot\Vg\dr s
\end{equation}
for all $\Vv\in \mathring{H^1}(\Omega;\VR^2)$.  
If $\Vu_\alpha\in [C^2(\Omega)]^2\cap [C^1(\overbar \Omega)]^2$ then $\Vu_\alpha$ satisfies the corresponding Euler-Lagrange equations:
\begin{subequations}\label{eq:euler_lagrange_strong}
	\begin{align}
\Delta \Vu_\alpha & = 0 \quad \text{ in } \Omega,\\
 \partial \Vu_\alpha\nubf + \hat{\partial}\Vu_\alpha\nubf + \alpha \partial \Vu_\alpha\nubf & = \alpha \Vg \quad \text{ on } \partial \Omega, 
\end{align}
\end{subequations}
where
\ben
\hat{\partial}\Vu = \begin{pmatrix}
	-\partial_y u_2 &  \partial_x u_2 \\
				 \partial_y u_1	& - \partial_x u_1
\end{pmatrix}.
\een	
\end{lemma}
\begin{proof}
    It follows from Poincar\'e's inequality that $\Vu\mapsto \|\partial \Vu\|_{[L^2]^2}$ is a norm on $\mathring{H^1}(\Omega;\VR^2)$ that is equivalent to the norm $\Vu\mapsto \|\Vu\|_{H^1} := \sqrt{\|\Vu\|_{[L^2]^2} + \|\partial \Vu\|_{[L^2]^{2\times2}}}$. 
    This implies that the functional \eqref{eq:min_L2_least_square} is strictly convex.
    Furthermore, it is lower semi-continuous (lsc) and hence by convexity also weakly-lsc.
    Existence of a unique minimiser thus follows from the direct method of calculus of variations.

    The first order necessary and sufficient optimality condition of \eqref{eq:min_L2_least_square} reads: find $\Vu_\alpha\in [\mathring{H}^1(\Omega)]^2$ such that
 \ben\label{eq:euler_lagrange}
 \int_{\Omega} \Cb\Vu_\alpha\cdot \Cb\Vv + \alpha \partial \Vu_\alpha\colon \partial \Vv\dr x = \alpha\int_{\partial \Omega} \Vg\cdot \Vv \dr s\quad \text{ for all } \Vv \in [\mathring{H}^1(\Omega)]^2.
\een
 Partial integrating in \eqref{eq:euler_lagrange} and then using the fundamental theorem of calculus of variations leads to \eqref{eq:euler_lagrange_strong}. 
\end{proof}

In the next corollary we show that if $\Vg$ is compatible, that is, if the CR Neumann problem \eqref{eq:BVP_CR_Neumann} admits a solution, then this solution is also a minimiser of \eqref{eq:min_L2_least_square}.

\begin{corollary}\label{cor:l2-min-gives-conformal-when-exists}
    Let $\Vg\in [C(\partial \Omega)]^2$ be given. Suppose that \eqref{eq:BVP_CR_Neumann} admits a solution $\Vu \in [C^2(\Omega)]^2\cap [C^1(\overbar\Omega)]^2$. Then $\Vu$ (up to a constant) is the minimiser of \eqref{eq:min_L2_least_square}
and satisfies the Neumann problem
	\begin{subequations}
		\begin{align}
		\Delta \Vu & = 0 \quad \text{ in } \Omega,\\
		\partial \Vu\nubf & = \Vg \quad \text{ on } \partial \Omega.
		\end{align}
    \end{subequations}
\end{corollary}
\begin{proof}
    If $ \Vu$ is a solution of \eqref{eq:BVP_CR_Neumann} then $\Delta  \Vu=0$ in $\Omega$ and also $\partial \Vu = -\hat{\partial }\Vu$ on $\partial \Omega$. Therefore $\Vu$ solves 	
    \begin{align}
        \Delta  \Vu & = 0 \quad \text{ in } \Omega,\\
        (\partial\Vu - \partial \Vu + \alpha \hat\partial \Vu)\nubf & = \alpha \Vg \quad \text{ on } \partial \Omega, 
    \end{align}
    and thus owing to Lemma~\ref{lem:existence_min_neumann} $\Vu$ is the unique minimiser (up to a constant) of \eqref{eq:min_L2_least_square}.
\end{proof}
The previous lemma only gives a characterisation of the minimiser of \eqref{eq:min_L2_least_square} in the case that the Cauchy-Riemann equations admit a
solution. In general the Cauchy-Riemann equations do not have solutions and as a result the minimisers cannot be holomorphic. However, we can show that
they are nearly holomorphic by splitting the space $[\mathring{H}^1(\Omega)]^2$ into 
a holomorphic and a non-holomorphic part.
In the following proposition we think of $H=[\mathring{H}^1(\Omega)]^2$ but $H$ could be any Hilbert space continuously embedded into $[H^1(\Omega)]^2$.
\begin{proposition}\label{prop:solution-decomposition}
Consider the orthogonal decomposition $H = V \oplus V^\perp $ for $V=\{\Vu : \Cb \Vu=0\}$.
    Furthermore let $\varphibf_{\Vg} $ be the unique solution to 
    \begin{equation}\label{eq:varphi_g}
        (\varphibf_{\Vg}, \Vv)_H = \int_{\partial\Omega} \Vg\cdot\Vv\dr s.
    \end{equation}
    Denote $\varphibf_{\Vg}= \varphibf_{\Vg}^{(1)}+ \varphibf_{\Vg}^{(2)}$ for $\varphibf_{\Vg}^{(1)}\in V$ and $\varphibf_{\Vg}^{(2)}\in V^\perp$ and decompose the solution  $\Vu_{\alpha} = \Vu_{\alpha}^{(1)} + \Vu_{\alpha}^{(2)}$ of \eqref{eq:min_L2_least_square} in the same way.  
    Then 
    \begin{itemize}
        \item[(i)] $\Vu_{\alpha}^{(1)} = \varphibf_{\Vg}^{(1)}$ for all $\alpha>0$.
        \item[(ii)] There exists a constant $C>0$ independent of $\alpha$ such that 
            \begin{equation}
                \|\Vu_{\alpha}^{(2)}\|_{H} \le \| \varphibf_{\Vg}^{(2)}\|_{H} \le C \| \Vg\|_{[L^2(\partial \Omega)]^2}.
            \end{equation}
        \item[(iii)] $\Vu_{\alpha}^{(2)}$ satisfies
            \begin{equation} \label{eq:u2-characterisation}
                (\Cb \Vu_{\alpha}^{(2)}, \Cb \Vv^{(2)} )_{[L^2]^2} + \alpha ( \Vu_{\alpha}^{(2)}, \Vv^{(2)})_{H} = \alpha \int_{\partial \Omega} \Vg\cdot \Vv^{(2)} \dr s
            \end{equation}
         for all  $\Vv^{(2)} \in V^\perp$.
         \item[(iv)] We have
         \begin{align}
         \Vu_{\alpha} & \to \varphibf_{\Vg}^{(1)}  && \text{ strongly in }\quad  H \text{ as } \alpha \to 0, \\
         \frac{1}{\sqrt \alpha}\Cb(\Vu_{\alpha}) &\to 0 && \text{ strongly  in } \quad [L^2(\Omega)]^2 \text{ as } \alpha \to 0.
         \end{align}
     \item[(v)] Lastly, 
         \begin{align}
         \Vu_\alpha  \to \varphibf_{\Vg}  && \text{ strongly in }\quad H \text{ as } \alpha \to \infty.
         \end{align}
    \end{itemize}
\end{proposition}
\begin{proof}
    (i) First we observe that using the weak form of \eqref{eq:varphi_g} the optimality condition \eqref{eq:euler_lagrange} is equivalent to: find $\Vu_{\alpha}^{(1)}\in V$ and $\Vu_{\alpha}^{(2)}\in V^\perp$ such that 
    \begin{equation}\label{eq:opt_split}
        \begin{aligned}
            (\Cb \Vu_{\alpha}^{(2)}, \Cb \Vv^{(2)})_{[L^2]^2} &+ \alpha (\Vu_{\alpha}^{(1)}, \Vv^{(1)})_{H} + \alpha (\Vu_{\alpha}^{(2)},  \Vv^{(2)})_{H} \\
            & \qquad\qquad = \alpha (\varphibf_\Vg^{(1)}, \Vv^{(1)})_{H} + \alpha (\varphibf_\Vg^{(2)}, \Vv^{(2)})_{H}
        \end{aligned}
    \end{equation}
    for all $\Vv^{(1)}\in V$ and $\Vv^{(2)}\in V^\perp$. By testing \eqref{eq:opt_split} with $\Vv^{(2)}=0$ we obtain
    \begin{equation}\label{eq:opt_split2}
    \begin{aligned}
        (\Vu_{\alpha}^{(1)}-\varphibf_\Vg^{(1)}, \Vv^{(1)})_{H} 
   = 0
    \end{aligned}
    \end{equation}
    for all $\Vv^{(1)}\in V$. This means that $\Vu_{\alpha}^{(1)}-\varphibf_\Vg^{(1)}\in V^\bot$, but since by definition $\Vu_{\alpha}^{(1)}-\varphibf_\Vg^{(1)}\in V$ we conclude that $\Vu_{\alpha}^{(1)}-\varphibf_\Vg^{(1)}=0$ or $\Vu_{\alpha}^{(1)} = \varphibf_{\Vg}^{(1)}$. 
    
    (ii)-(iii) Using the weak form of \eqref{eq:varphi_g} and (i) we see that \eqref{eq:u2-characterisation} follows from \eqref{eq:opt_split}. 
    Testing \eqref{eq:u2-characterisation}  with $\Vv^{(2)}=\Vu_{\alpha}^{(2)}$ we obtain that 
    \begin{equation}
        \|\Vu_\alpha^{(2)}\|_{H}^2 \le  \|\Vg \|_{[L^2(\partial\Omega)]^2}\|\Vu_\alpha^{(2)}\|_{[L^2(\partial \Omega)]^2}
    \end{equation}
    and hence using Poincar\'e's inequality we see that there exists $C$ independent of $\alpha$, such that $\|\Vu_\alpha^{(2)}\|_{H} \le C \|\Vg\|_{[L^2(\partial \Omega)]^2}$ for all $\alpha$. 

    (iv) By (ii) we have that for every null-sequence $(\alpha_n)$, there exists $\Vu_0^{(2)}\in V^\perp$ and a subsequence $(\alpha_{n_k})$ of $(\alpha_n)$ such that $\Vu_{\alpha_{n_k}}^{(2)} \weakto \Vu_0^{(2)}$ (since $V^\perp$ is closed and convex and thus weakly closed).
    Therefore using \eqref{eq:u2-characterisation} we find
    \begin{equation}
        \begin{aligned}
            (\Cb \Vu_0^{(2)}, \Cb \Vu_0^{(2)})_{[L^2]^2} &= \lim_{n \to \infty} (\Cb \Vu_{{\alpha_{n_k}}}^{(2)}, \Cb \Vu_0^{(2)})_{[L^2]^2}\\
            & \stackrel{\eqref{eq:u2-characterisation}}{=} \lim_{n \to \infty}  - \alpha_{n_k} (\Vu_{{\alpha_{n_k}}}^{(2)}, \Vu_0^{(2)})_{H} + \alpha_{n_k} \int_{\partial \Omega} \Vg \cdot \Vu_0^{(2)}\dr s = 0.
            \end{aligned}
    \end{equation}
    Thus $\Cb \Vu_0^{(2)} = 0 $, which means $\Vu_0^{(2)}\in V$ and hence $\Vu_0^{(2)}=0$. Since the null-sequence $(\alpha_n)$ was arbitrary and since the limit is unique we obtain that $\Vu^{(2)}_{\alpha} \rightharpoonup 0$ in $H$ as $\alpha \to 0$.

    Testing \eqref{eq:u2-characterisation} with $\Vv_\alpha^{(2)} = \Vu_\alpha^{(2)}$, we obtain (using the weak-weak continuity of the trace operator)
    \begin{equation} 
        \lim_{\alpha\to 0} \|\Cb \Vu_{\alpha}^{(2)}\|_{[L^2]^2}^2 =  - \lim_{\alpha\to 0} \alpha \| \Vu_{\alpha}^{(2)}\|_{H}^2  +  \lim_{\alpha\to 0} \alpha \int_{\partial \Omega} \Vg\cdot \Vu^{(2)}_\alpha \dr s  =0.
    \end{equation}
This shows $\Cb\Vu_{\alpha}^{(2)} \to 0$ in $[L^2(\Omega)]^2$ as $\alpha \to 0$.
To show that $\Vu_{\alpha}^{(2)} \to 0$ strongly in $H$ as $\alpha \to 0$, we show that its norm converges to $0$. 
We test \eqref{eq:u2-characterisation} again with $\Vv_\alpha^{(2)} = \Vu_\alpha^{(2)}$ and estimate the left-hand side
\begin{equation}
    \label{eq:estimate_u_alpha}
    \begin{split}
        \alpha \|\Cb \Vu_{\alpha}^{(2)}\|_{[L^2]^2}^2 + \alpha \| \Vu_{\alpha}^{(2)}\|_{H}^2 & \le  \|\Cb \Vu_{\alpha}^{(2)}\|_{[L^2]^2}^2 + \alpha \| \Vu_{\alpha}^{(2)}\|_{H}^2 \\
                                                                                                                    & = \alpha \int_{\partial \Omega} \Vg\cdot \Vu_{\alpha}^{(2)} \dr s
    \end{split}
\end{equation}
  and thus it follows 
  \begin{align}
      \lim_{\alpha \to 0}  \|\Vu_{\alpha}^{(2)}\|_{H}^2 \le - \lim_{\alpha \to 0} \|\Cb \Vu_{\alpha}^{(2)}\|_{[L^2]^2}^2 + \lim_{\alpha \to 0} \int_{\partial \Omega} \Vg\cdot \Vu_{\alpha}^{(2)} \dr s = 0.
  \end{align}
  Using this strong convergence we obtain $\Cb(\frac{\Vu_{\alpha}^{(2)}}{\sqrt \alpha}) \to 0$ in $[L^2(\Omega)]^2$ as $\alpha \to 0$ from  \eqref{eq:estimate_u_alpha}. 

(v) Recall that $\Vu_\alpha^{(2)}$ is bounded in $V^\bot$ uniformly in $\alpha$ and hence for every sequence $(\alpha_n)$ with $\alpha_n\to \infty$,  we find $\overbar{\Vu}^{(2)}\in V^\bot$ and a subsequence $(\alpha_{n_k})$ of $(\alpha_n)$ such that $\Vu_{\alpha_{n_k}}^{(2)}\to \overbar{\Vu}$. 
Hence choosing $\alpha=\alpha_{n_k}$ in \eqref{eq:u2-characterisation} and then passing to the limit $k\to \infty$ we obtain
\begin{equation}
    \label{eq:1alpha_B2}
    ( \overbar{\Vu}, \Vv^{(2)})_{H} =  \int_{\partial \Omega} \Vg\cdot \Vv^{(2)} \dr s
\end{equation}
for all $\Vv^{(2)}\in V^\bot$. Since \eqref{eq:1alpha_B2} admits a unique solution it follows that $\overbar{\Vu} = \varphibf_{\Vg}^{(2)}$. Finally the strong convergence follows again by testing \eqref{eq:u2-characterisation} with $\Vv_\alpha^{(2)} = \Vu_\alpha^{(2)}$ and passing to the limit $\alpha\to \infty$. 
\end{proof}
\subsubsection*{Symmetrized shape gradient in $H^1$}
For any $\Vu \in [H^1(\Omega)]^2$ it holds that $\|\partial \Vu\|_{[L^2]^{2\times2}}^2 = \|\sym(\partial\Vu)\|_{[L^2]^{2\times2}}^2 + \|\asym(\partial \Vu)\|_{[L^2]^{2\times2}}^2$.
Now observe that $\|\asym(\partial \Vu)\|^2_{[L^2]^{2\times2}} = \|\partial_x u_2 - \partial_y u_1\|^2_{L^2}$ and hence minimising this term works against the satisfaction of the Cauchy-Riemann equations, as they require $ \partial_x u_2 + \partial_y u_1=0$.

In order for the symmetric part of the gradient to be a norm on $[H^1(\Omega)]^2$, we have to remove the null-space of the asymmetric part.
We define 
\begin{equation}
    H(\sym,\Omega) \defeq \{ \Vu \in [L^2(\Omega)]^2 : \sym(\partial\Vu) \in [L^2(\Omega)]^{2\times2}\}.
\end{equation}
By Korn's inequality in its various forms \cite[Thm.~2 and Thm.~4]{KonOle88} we know that $H(\sym,\Omega) = [H^1(\Omega)]^2$, and that 
\begin{equation}
    \mathring H(\sym,\Omega) \defeq \{\Vu \in H(\sym,\Omega) : \int_\Omega \Vu\dr x =0 \text{ and }  \int_{\Omega} \partial_y u_1 + \partial_x u_2 \dr x = 0\}
\end{equation}
is a Hilbert space when equipped with the inner-product given by 
\begin{equation}
    (\Vu, \Vv)_{\mathring{H}(\sym)} \defeq (\sym(\partial \Vu), \sym(\partial \Vv))_{[L^2]^{2\times2}}.
\end{equation}
We then consider the following minimisation problem: given $\Vg\in [L^2(\partial\Omega)]^2$,
we study the following least square relaxation of \eqref{eq:grad-optim-problem-with-strict-cr}: find a minimiser $\Vu\in \mathring{H}(\sym,\Omega)$ of 
\begin{equation}
    \label{eq:min_L2_least_square_sym}
    \min_{\Vu\in \mathring{H}(\sym,\Omega)} \frac12 \left(\|\Cb\Vu\|_{[L^2]^2}^2   + \alpha \|\sym(\partial\Vu)\|_{[L^2]^{2\times2}}^2\right)  - \alpha \int_{\partial \Omega} \Vg\cdot \Vu\; \dr s, \quad \alpha >0.
\end{equation}
The existence and uniqueness of a minimiser follows as before, and the decomposition statement as stated in Proposition~\ref{prop:solution-decomposition} holds.
%
The good properties of mesh deformations obtained by using the symmetric part of the gradient have been observed before e.g. by \cite{schulz2016computational}, where an inner product based on the linear elasticity equations has been used.

\subsection{Nearly holomorphic shape gradients with mixed boundary conditions}
Often a part of the boundary $\partial \Omega$ is clamped somewhere, meaning that only a part of the shape can freely move. The next lemma shows that the space of clamped holomorphic mappings is essentially zero. 

%
Let $\Gamma_D\subset \partial \Omega$ be a relatively closed, measurable part of the boundary with positive surface measure $|\Gamma_D|>0$ that consists of finitely many connected components and has at least one interior point. We define $\Gamma := \Omega \setminus \Gamma_D$ and let $[H^1_{\Gamma}(\Omega)]^2 \defeq \{\Vu\in [H^1(\Omega)]^2 : \Vu\vert_{\Gamma_D}=0\}$. 
Notice that in case $\Omega=\partial \Omega$ this definition reduces to the usual one if $0$ is identified with with the empty set. 

\begin{lemma}\label{lem:non-existence-clamped-cr}
    Let $\Omega\subset\VR^2$ be an open set with a boundary $\partial \Omega$ that can be locally represented by the graph of a Lipschitz function.
    Let $\Gamma_D$ be as above, then  $\Vu\in [H^1_{\Gamma}(\Omega)]^2$ satisfies $\Cb \Vu=0$ a.e. on $\Omega$ if and only if $\Vu=0$ a.e. on $\Omega$.
\end{lemma}
\begin{proof}
    Since $\Gamma_D$ contains an interior point, there exists $\Vz\in \Gamma_D$ and a ball $B_r(\Vz)$ centered at $\Vz$, of radius $r>0$ such that $B_r(\Vz)\cap \partial\Omega = B_r(\Vz)\cap\Gamma_D$.
    Now let $U\defeq \Omega\cup B_r(\Vz)$ and denote the extension of $\Vu$ by $0$ to $U$ by $\tilde\Vu$.

    \textit{Claim: $\tilde\Vu$ is holomorphic on $U$}. Thanks to \cite[Thm.~1]{GraMor78} the function $\tilde \Vu$ is holomorphic on $U$ if $\tilde \Vu$ belongs to $[H^1(U)]^2$ and the weak derivatives satisfy the Cauchy-Riemann equations almost everywhere in $U$.  To see this, we let $\psi \in C^\infty_0(U)$ and use Green's theorem in $H^1(\Omega)$ to get 
    \begin{equation}
        \int_U \partial_{x_i}\psi \,\tilde u_j\,\dr x = \int_\Omega \partial_{x_i}\psi \, u_j\,\dr x = \int_{\partial \Omega} \psi u_j \, \nu_i\, \dr s - \int_{\Omega} \psi \, \partial_{x_i} u_j \, \dr x,
    \end{equation}
   where $\nu_i$ denotes the $i$th component of the outward pointing unit normal field along $\partial \Omega$.  Notice that first term of the right hand side vanishes since the trace of $\Vu$ is $0$ on $\Gamma_D \supseteq U \cap \partial \Omega$ and $\psi = 0$ in $\Gamma = \partial \Omega \setminus \Gamma_D$. This shows that $\Cb\tilde \Vu =0$ a.e. on $U$ and hence $\tilde \Vu$ coincides almost everywhere with a 
   holomorphic function.  

    The claim now follows from the uniqueness principle (also called principle of permanence) \cite[p.~156]{gamelin2003complex}, as $\tilde\Vu$ is zero on $U\cap\Gamma_D$ which contains the accumulation point $\Vz$. 
\end{proof}

This means that a decomposition into a holomorphic and a non-holomorphic part as in Proposition~\ref{prop:solution-decomposition} is not possible for this case and that the solution will always be non-holomorphic.
\begin{problem}
Given $\Vg\in [L^2(\partial\Omega)]^2$, we study the following minimisation problem
\begin{equation} \label{eq:min_L2_least_square_dirichlet}
    \min_{\Vu\in [H^1_{\Gamma}(\Omega)]^2} \frac{1}{2} \bigg(\|\Cb \Vu\|_{[L^2]^2}^2 + \alpha \|\partial \Vu\|_{[L^2]^{2\times 2}}^2\bigg) - \alpha \int_{\Gamma} \Vg \cdot \Vu \dr s, \quad \alpha > 0. 
\end{equation}
\end{problem}

\begin{lemma}
Problem \eqref{eq:min_L2_least_square_dirichlet} admits a unique solution $\Vu$ that satisfies
\begin{equation}
    (\Cb \Vu, \Cb \Vv)_{[L^2]^2} + \alpha (\partial \Vu,\partial\Vv)_{[L^2]^{2\times2}} = \int_{\Gamma} \Vg \cdot \Vv\dr s
\end{equation}
for all $\Vv\in [H^1_{\Gamma}(\Omega)]^2$.
\end{lemma}
\begin{proof}
The proof is identical with the one of Lemma~\ref{lem:existence_min_neumann}; note that we do not need to remove the kernel of constants due to the Dirichlet boundary condition. 
\end{proof}

\subsection{Weighted holomorphicity}
\label{sec:weighted-ls-cr-fem}
In order to control where the mapping is close to holomorphic and where it is allowed to be non-holomorphic, we introduce a weighting function $\mu\in L_\infty(\Omega)$. Then the modified version of \eqref{eq:min_L2_least_square_dirichlet}  reads:
\begin{equation} \label{eq:min_L2_least_square_dirichlet_weighted}
    \min_{\Vu\in [H^1_{\Gamma}(\Omega)]^2} \frac{1}{2} \bigg(\|\mu\Cb \Vu\|_{[L^2]^2}^2 + \alpha \|\partial \Vu\|_{[L^2]^{2\times 2}}^2\bigg) - \alpha \int_{\Gamma} \Vg \cdot \Vu \dr s, \quad \alpha > 0. 
\end{equation}
Although the weight function $\mu$ is arbitrary, a choice 
yielding good numerical results on which we report later, is given by
\begin{equation}\label{eqn:weighting-function-distance}
    \mu(\Vx) := \left(\frac{\epsilon}{d_{\partial \Omega}(\Vx)+\epsilon}\right)^{1/2},
\end{equation}
where $d_{\partial \Omega}(\Vx) := \inf_{\Vy\in \partial \Omega} |\Vx-\Vy|$ denotes
the distance function associated with $\partial \Omega$ and $\epsilon>0$ is a small 
parameter. Intuitively that choice guarantees more holomorphicity near the boundary of $\Omega$ which is justified 
by the fact that grid points near the boundary are subject to larger 
deformations.
The corresponding Euler-Lagrange equation to \eqref{eq:min_L2_least_square_dirichlet_weighted} reads: find $\Vu\in [H^1_{\Gamma}(\Omega)]^2$ such that 
\begin{equation}
    (\mu \Cb \Vu, \mu \Cb \Vv)_{[L^2]^2} + \alpha (\partial \Vu,\partial\Vv)_{[L^2]^{2\times2}} = \int_{\Gamma} \Vg \cdot \Vv\dr s
\end{equation}
for all $\Vv\in [H^1_{\Gamma}(\Omega)]^2$.
\section{Nearly holomorphic shape gradients in RKHS}\label{sec:4}
In some instances it can be advantageous to work with 
smoother deformations than provided by the Hilbert space $[H^1(\Omega)]^2$. 
However, this is difficult to realise when using finite elements to discretise the deformations, as most finite element software does not support $C^k$, $k\ge 1$, conforming finite elements.
For this purpose we propose to use
reproducing kernel Hilbert spaces (RKHS) $\Ch(\Omega)$ which by construction contain continuous functions and can be used to discretise deformations of arbitrarily high regularity.

\subsection{Reproducing kernel Hilbert spaces}
We begin with the definition of RKHS. 
\begin{definition}
We call a Hilbert space $\Ch(\Omega)$ of functions $\Omega\to \VR$ a
reproducing kernel Hilbert space if for every $\Vx\in \Omega$ the mapping
\ben
\varphibf \mapsto \varphibf(\Vx),\; \Ch(\Omega) \to \VR
\een
is a continuous functional.
\end{definition}

It can be shown (see \cite[Thm.~10.2]{Wendlandbook}) that if $\Ch(\Omega)$ is a RKHS then 
there exists a kernel function $\ksf:\Omega\times\Omega\to \VR$ (called reproducing kernel) such that  
\begin{enumerate}[$(a)$]
		\item 
		$\ksf(\Vx,\cdot)\in \Ch(\Omega)$ for every $\Vx\in \Omega$,
		\item
		$(\ksf(\Vx,\cdot), f)_{\Ch} = f(\Vx)	$ for every $\Vx\in \Omega$ and every $f\in \Ch(\Omega)$.
\end{enumerate}

The kernel $\ksf$ is always symmetric, i.e.,  $\ksf(\Vx, \Vy) = \ksf(\Vy, \Vx)$ for all $\Vx,\Vy\in \Omega$ and it is also positive semi-definite in the following sense.

\begin{definition}
	A kernel $\ksf:\Omega\times \Omega\to \VR$ is said to be positive (semi-)definite on $\Omega$ if,
	for every finite subset of pairwise distinct points $\{\Vx_i\}_{i=1}^N\subset\Omega$,
	the matrix 	$(\ksf(\Vx_i,\Vx_j))_{i,j=1}^{N}$ is positive (semi-)definite.
\end{definition}

The notion of reproducing kernels readily carries over to the vector valued case.  For instance the Cartesian product $[\Ch(\Omega)]^2\coloneqq \Ch(\Omega)\times\Ch(\Omega)$ of a RKHS
$\Ch(\Omega)$ (with reproducing kernel $\ksf$) is itself a RKHS. Its (matrix valued) reproducing kernel is $\Ksf(\cdot,\cdot)=\ksf(\cdot,\cdot) \VI_2$, where $\VI_2\in\VR^{2\times2}$ denotes the identity matrix. The reproducing kernel property (b) reads: for all $\Va\in \VR^2$, $\Vx\in \Omega$,  and all $\Vf\in [\Ch(\Omega)]^2$ we have
\ben
(\Ksf(\Vx,\cdot)\Va,\Vf)_{\Ch} = \Va\cdot \Vf(\Vx).
\een 
In view of Sobolev's embedding theorem in 2D \cite[Thm.~4.57]{DemDem12} we have $H^s(\Omega)\subset C(\overbar\Omega)$, $s>1$, and thus Sobolev spaces are important examples of a RKHS. 
The space $H^{2.5}(\VR^2)$ for instance can be obtained by the positive-definite kernels with compact support \cite[Thm.~10.35 and Table 9.1]{Wendlandbook}, $\ksf^\sigma_4(\Vx,\Vy) = \varphi(\|\Vx-\Vy\|/\sigma)$, where 
\begin{equation}\label{eqn:wendland-kernel-definition}
	\varphi(r)\coloneqq \left(1-r\right)_+^4\left(4r +1\right),
	\end{equation}
	and $(x)_+\coloneqq\max(0, x)$ denotes the positive part function. It is important to notice that the norm induced by the kernel is equivalent (but not equal) to the standard
    norm on $H^{2.5}(\Omega)$; \cite[Cor.~10.13]{Wendlandbook}. This kernel is the lowest order Wendland
	kernel with compact support that is
	positive-definite on $\VR^2$ and is of class $C^2$. Therefore RKHS are a natural extension of $[H^1(\Omega)]^2$ towards smoother deformations.

    We henceforth assume that $\ksf\in C^2(\VR^2\times \VR^2)$ is a positive definite symmetric kernel. The restriction of the kernel $\ksf$ to $\Omega\subset \VR^2$ generates a RKHS which we will denote by $\Ch(\Omega)$. Similarly we denote by $[\Ch(\Omega)]^2$ the RKHS generated by 
    $\Ksf(\Vx,\Vy) := \ksf(\Vx,\Vy) \VI_2$. By \cite[Thm.~10.44, p.~167]{Wendlandbook} it follows that $\partial_{x_i}\ksf(\cdot, \Vx)\in \Ch(\Omega)$ for all $\Vx\in \Omega$.   Notice that there is a continuous trace operator $T:\Ch(\Omega) \to \Ch(\partial\Omega)$; we refer to the appendix for details. 

Finally we will need the following result:
\begin{lemma}\label{lem:weak_point_wise}
Let $(\Vv_n)$ be a sequence in $[\Ch(\Omega)]^2$ that converges weakly to $\Vv\in [\Ch(\Omega)]^2$. Then for all $\beta = (\beta_1,\beta_2)$, $\beta_1+\beta_2\le 1$, $\beta_1,\beta_2\ge 0$,  we have  $\partial^\beta \Vv_n(\Vx) \to \partial^\beta \Vv(\Vx)$ for all $\Vx\in \Omega$ as $n\to \infty$. 
\end{lemma}
\begin{proof}
	Let $(\Vv_n)$ be a sequence in $[\Ch(\Omega)]^2$ that converges weakly in $[\Ch(\Omega)]^2$ to $\Vv\in [\Ch(\Omega)]^2$. By \cite[Thm.~10.44, p.~167]{Wendlandbook} we have for all $\Vv\in [\Ch(\Omega)]^2$ and all $\Vx\in \Omega$ and $\Va\in \VR^2$, 
	\ben
	\partial^\beta \Vv(\Vx)\cdot \Va = (\partial^\beta \Ksf(\Vx,\cdot)\Va,\Vv)_{[\Ch]^2} 
	\een
	for all multi-indices $\beta =(\beta_1,\beta_2)$ with $|\beta_1|+|\beta_2| \le 1$. As a result for all $\Va\in \VR^2$ and all multi-indices $\beta =(\beta_1,\beta_2)$ with $|\beta_1|+|\beta_2| \le 1$, we have
	\ben
	\begin{split}
		0=\lim_{n\to \infty}\partial^\beta \Vv_n(\Vp_\ell)\cdot \Va &  = \lim_{n\to \infty} (\partial^\beta \Ksf(\Vp_\ell,\cdot)\Va,\Vv_n)_{[\Ch]^2}\\
		& = (\partial^\beta \Ksf(\Vp_\ell,\cdot)\Va,\Vv)_{[\Ch]^2} \\
		& = \partial^\beta \Vv(\Vp_\ell)\cdot \Va
	\end{split}
	\een
	for $\ell = 1,\ldots, n$, which shows $\Vv\in [\Ch(\Omega)]^2$ and thus finishes the proof. 
\end{proof}

\subsection{Nearly holomorphic shape gradients in $\Ch$}
We now consider the minimisation problem \eqref{eq:min_abstract} in the 
Hilbert space $H=[\Ch(\Omega)]^2$.  We require $\Vg\in [\Ch(\partial\Omega)]^2$ and 
for a finite set of points
$P_n := \{\Vp_1,\ldots, \Vp_n\}\subset \Omega$, 
define the space 
\begin{equation}
    \Cp \defeq \VR^{n \times 2}, \quad \|\Vv\|_{\Cp} \defeq \frac{1}{n} \sum_{\ell=1}^n (v_{i,1}^2+v_{i,2}^2)
\end{equation}
and let $\Cb:H\to\Cp$ be the point evaluation of the Cauchy-Riemann operator in the points $\Vp_1,\ldots, \Vp_n$.
With this choice \eqref{eq:min_abstract} reads:
\begin{equation}
    \label{eq:min_CR_relaxed_pw}
    \min_{\Vu\in [\Ch(\Omega)]^2}  \frac12\left(\frac1n \sum_{\ell=1}^n|\Cb\Vu(\Vp_\ell)|^2 + \alpha \|\Vu\|_{[\Ch]^2}^2\right) - \int_{\partial \Omega}\alpha \Vg\cdot \Vu\dr s, \quad \alpha >0.
\end{equation}
The canonical choice for the points $\Vp_\ell$ are the nodes of the mesh that is being deformed, as this is where conformality has to be enforced.

By Lemma~\ref{lem:weak_point_wise} weak convergence in $[\Ch(\Omega)]^2$ implies pointwise convergence of the partial derivatives. Therefore it can be readily checked by the direct method of calculus of variations that the minimisation problem \eqref{eq:min_CR_relaxed_pw} admits a unique solution; see \cite{Wendlandbook}.

We can now prove an analogous result to Proposition~\ref{prop:solution-decomposition}. The main difference to the $L^2(\Omega)$-case is that now we are able to show uniform convergence  of the Cauchy-Riemann operator at the points $\Vp_\ell$ instead of $L^2(\Omega)$-convergence. In this way  more localised holomorphicity can be achieved.

\begin{lemma}
    The linear space $V=\{\Vu : \Cb \Vu(\Vp_\ell)=0, \; \ell=1,\ldots, n\}$ is closed in $[\Ch(\Omega)]^2$. 
\end{lemma}
\begin{proof}
This is a direct consequence of Lemma~\ref{lem:weak_point_wise}. 
\end{proof}
As a result of the previous lemma we may consider the 
orthogonal decomposition $[\Ch(\Omega)]^2 = V \oplus V^\perp $ for $V=\{\Vu : \Cb \Vu(\Vp_\ell)=0, \; \ell=1,\ldots, n\}.$
\begin{proposition}\label{prop:solution-decomposition-RKHS}
	Let $\varphibf_{\Vg} \in [\Ch(\partial\Omega)]^2 $ be the solution to 
	\begin{equation}\label{eq:grad_RKHS}
        (\varphibf_{\Vg},\varphibf)_{[\Ch]^2}  =  \int_{\partial\Omega} \Vg\cdot\varphibf\; \dr s \quad \text{ for all } \varphibf\in [\Ch(\Omega)]^2.
	\end{equation}
	Denote $\varphibf_{\Vg}=   \varphibf_{\Vg}^{(1)}+ \varphibf_{\Vg}^{(2)}$ for $\varphibf_{\Vg}^{(1)}\in V$ and $\varphibf_{\Vg}^{(2)}\in V^\perp$ and decompose the solution  $\Vu_{\alpha} = \Vu_{\alpha}^{(1)} + \Vu_{\alpha}^{(2)}$ of \eqref{eq:min_CR_relaxed_pw} in the same way.
	
	Then 
	\begin{itemize}
		\item[(i)] $\Vu_{\alpha}^{(1)} = \varphibf_{\Vg}^{(1)}$ for all $\alpha>0$,
		\item[(ii)] there exists a constant $C>0$ independent of $\alpha$ and $n$ such that 
            \begin{equation}\label{eq:estimate_kernel}
            \| \Vu_{\alpha}^{(2)}\|_{[\Ch]^2} \le C \| \Vg\|_{[\Ch(\partial\Omega)]^2},
		\end{equation}
		\item[(iii)] $\Vu_{\alpha}^{(2)}$ satisfies
		\begin{equation} \label{eq:u2-characterisation_H_RKHS}
            \frac1n\sum_{\ell=1}^n \Cb\Vu_\alpha^{(2)}(\Vp_\ell)\cdot \Cb\Vv^{(2)}(\Vp_\ell) + \alpha ( \Vu_{\alpha}^{(2)}, \Vv^{(2)})_{[\Ch]^2} = \alpha \int_{\partial \Omega} \Vg\cdot \Vv^{(2)} \dr s
		\end{equation}
		for all  $\Vv^{(2)} \in V^\perp$, 
		\item[(iv)] We have
		\begin{align}
        \Vu_{\alpha} \to \varphibf_{\Vg}^{(1)}  &  \quad \text{ strongly in  } \quad [\Ch(\Omega)]^2 \quad \alpha \to 0, \\
             \frac{1}{\sqrt \alpha} \max_{\Vp\in \Cp_n} |\Cb(\Vu_{\alpha})(\Vp)| &\to 0 \quad  \text{ as } \alpha \to 0.
		\end{align}
		\item[(v)] Lastly, 
		\begin{equation}
            \Vu_\alpha \to \varphibf_{\Vg}    \quad \text{ strongly in  } \quad [\Ch(\Omega)]^2\quad \alpha \to \infty.
		\end{equation}
	\end{itemize}
\end{proposition}
\begin{proof}
    The proof is very similar to the one of Proposition~\ref{prop:solution-decomposition} and deferred to the appendix. 
\end{proof}

\subsection{Approximation}
Our aim is now to introduce an approximation of \eqref{eq:min_CR_relaxed_pw}.  For this purpose let $ \Cx_h := \{\Vx_1, \ldots, \Vx_{N}\}\subset \overbar\Omega$  be a given set of points and define
\ben
V(\Cx_h) := \text{span}\{\ksf(\Vx,\cdot):\; \Vx \in \Cx_h   \},
\een
which is a finite dimensional subspace of $\Ch(\Omega)$.
  Subsequently we denote by $\underline \varphibf \in \VR^{2N}$ the vector corresponding to the expansion of the 
function $\varphibf$ in the basis $\{\ksf(\cdot,\Vx_i)\Ve_k:\; i=1,\ldots, N, \; k=1,2\}$ of $[V(\Cx_h)]^2$. 
The approximation of \eqref{eq:min_CR_relaxed_pw} on $[V(\Cx_h)]^2$ is then given by the minimiser of 
\begin{equation}
    \label{eq:min_CR_relaxed_pw_discrete}
    \min_{ \Vu\in [V(\Cx_h)]^2}   \frac12 \left( \frac1n\sum_{\ell=1}^n|\Cb\Vu(\Vp_\ell)|^2 + \alpha \|\Vu\|_{[\Ch]^2}^2\right) - \int_{\partial\Omega} \Vg \cdot \Vu\dr s.
\end{equation}

\begin{lemma}\label{lem:discrete_pw}
    There exists a unique solution $\Vu_h \in [V(\Cx_h)]^2$ of \eqref{eq:min_CR_relaxed_pw_discrete} that satisfies
	\ben\label{eq:CR_saddle_point_pw}
    \frac{1}{n} \sum_{\ell=1}^n  \Cb\Vu_h(\Vp_\ell)\cdot \Cb\varphibf(\Vp_\ell) +    \alpha ( \Vu_h,\varphibf)_{[\Ch]^2}    = \int_{\partial \Omega} \Vg\cdot \varphibf\; \dr s \quad  \text{ for all } \varphibf\in [V(\Cx_h)]^2,
     \een
    \end{lemma}
\begin{proof}
	Testing  \eqref{eq:CR_saddle_point_pw} with  $\varphibf(\cdot) = \ksf(\Vx,\cdot)\Ve_k$, $k=1,2$  and using the reproducing kernel property of $\ksf$ we see that \eqref{eq:CR_saddle_point_pw} is equivalent to: find $\underline \Vu_h \in \VR^{2N}$ such that
	\begin{subequations}\label{eq:CR_saddle_point_matrix_pw}
		\begin{align}
            (B^\top B + \alpha K)\underline\Vu_h = F,
		\end{align}
	\end{subequations}
	where the matrices  $K$, $B$ and the vector $F$ are given by 
	\ben\label{eq:matrices_pw}
	B= \begin{pmatrix}
		-B_1 & B_2\\
		B_2 & B_1
	\end{pmatrix}, \qquad K= \begin{pmatrix}
	\tilde K & 0\\
	0 & \tilde K
\end{pmatrix}, \qquad F = \begin{pmatrix} F_1\\ F_2 \end{pmatrix}
\een
with
\begin{align}
B_1 &:= \partial_{x_1}^1\ksf(\Cx_h,\Cp_n),       &  B_2    & := \partial_{x_2}^1\ksf(\Cx_h,\Cp_n),        &  \tilde K & := \ksf(\Cx_h,\Cx_h) \\
F_1 & :=  (\Vg, \ksf(\Cx_h,\cdot)\Ve_1)_{[L^2]^2} ,  &  F_2    & := (\Vg, \ksf(\Cx_h,\cdot)\Ve_2)_{[L^2]^2}.      &           & 
\end{align}
It is readily checked that $B^\top B + \alpha K$ is invertible which finishes the proof. 
\end{proof}

\subsection{Weighted holomorphicity}
Similarly as in the previous section we introduce a weighting function, now it is a discrete function, $\mu:\Cp_n \to \VR$. Then a modified version of \eqref{eq:min_CR_relaxed_pw_discrete} reads:
\ben\label{eq:min_CR_relaxed_pw_discrete_weighted}
    \min_{ \Vu\in [\Ch(\Omega)]^2}   \frac12 \left( \frac1n\sum_{\ell=1}^n|\mu(\Vp_\ell)\Cb\Vu(\Vp_\ell)|^2 + \alpha \|\Vu\|_{[\Ch]^2}^2\right) - \int_{\partial\Omega} \Vg \cdot \Vu\dr s.
	\een 
The weight function $\mu$ determines how holomorphic the solution $\Vu$ will be at the points $\Vp_\ell$. It is readily checked that the corresponding Euler-Lagrange equations read:
\ben\label{eq:CR_saddle_point_pw_weighted}
    \frac{1}{n} \sum_{\ell=1}^n  \mu(\Vp_\ell)\Cb\Vu(\Vp_\ell)\cdot \mu(\Vp_\ell)\Cb\varphibf(\Vp_\ell) +    \alpha ( \Vu,\varphibf)_{[\Ch]^2}    = \int_{\partial \Omega} \Vg\cdot \varphibf\; \dr s
\een
for all $\varphibf\in [\Ch(\Omega)]^2$. 

\section{Application to shape optimisation}\label{sec:5}

\subsection{Optimisation procedure}
Let $\Omega_0\subset \VR^2$ be a given initial shape.
A common approach in shape optimisation is to view the set of admissible shapes as a manifold and to perform the following optimisation:
on $\Omega_0$ one calculates a descent direction $\Delta \Vf_0$ and then defines $\Omega_1 = (\id+\Delta\Vf_0)(\Omega)$.
At the next iteration, one can looks for a new descent direction $\Delta\Vf_1$, this time defined on $\Omega_1$ and sets $\Omega_2 = (\id + \Delta\Vf_1)(\Omega_1) = (\id+\Delta\Vf_1)\circ(\id+\Delta\Vf_0)(\Omega)$.
However, in order to formulate Quasi-Newton methods, gradients and step directions defined on previous domains need to be compared and hence a transport between tangent spaces has to be computed.

To avoid the mathematical and computational complexity of such a transport, one can reformulate the problem over a linear space, by defining it on the initial domain only:
let $\Omega=\Omega_0$ and let $H(\Omega)$ be a Hilbert space of vector fields defined on $\Omega$.
We consider
\begin{equation}
    \min_{\substack{\Vf\in H\\\|\Vf\|_{H}\le c}} J_\Omega(\Vf), \qquad J_\Omega(\Vf) = J((\id+\Vf)(\Omega)),
\end{equation}
where $c>0$ is some constant and  $J$ is a shape function.
This means, that the steps $\Delta \Vf_i$ taken by an optimisation algorithm are always defined on $\Omega$ and the deformed shapes are given by $\Omega_i = (\id + \Delta \Vf_0 + \ldots + \Delta\Vf_i) (\Omega)$.

We obtain descent directions by using the L-BFGS implementation of \cite{ridzal2014rapid}; this requires the calculation of gradients.
Denoting by $\partial J_\Omega(\Vf)$ the directional derivative of $J_\Omega$ at $\Vf$, we can calculate the gradient at $\Vf$ by solving
\begin{equation}
    (\nabla J_\Omega(\Vf), \Vv)_{H} = \partial J_\Omega(\Vf)(\Vv).
\end{equation}
The directional derivative of $J_\Omega$ is related to the classical shape derivative of $J$.
Let us set $\Omega_{\Vf} := (\id+\Vf)(\Omega)$. 
It follows from the definition that the directional derivative of $J_\Omega$  at $\Vf\in H$ in direction $\Vv\in H$ coincides with the shape derivative of $J$ at $\Omega_{\Vf}$ in direction $\Vv \circ (\id + \Vf)^{-1}$, that is, 
\begin{equation}
    \partial J_\Omega(\Vf)(\Vv) = \Dr J(\Omega_{\Vf})(\Vv \circ (\id + \Vf)^{-1}).
\end{equation}

\subsection{Regular mesh and finite elements space} 
We henceforth assume that $\Omega\subset \VR^2$ is a domain with a polygonal boundary $\partial \Omega$. 
Let  $\{\mathcal T_h\}_{h>0}$ denote a family of simplicial triangulations $\mathcal T_h=\{K\}$ consisting 
of triangles $K$ such that 
\ben
\overbar{\Omega} = \bigcup \limits_{K\in \mathcal T_h}  K, \quad \forall h >0.
\een
For every element $K\in \mathcal T_h$, $h(K)$ denotes the diameter of the smallest ball containing $K$ 
and $\rho(K)$ is the diameter of the largest ball contained in $K$. 
The maximal diameter of all elements is denoted by $h$, i.e., 
$h:=\textnormal{max} \{h(K) \  | \ K\in \mathcal T_h\}.$ 
In many finite element calculations a mesh consisting of isotropic elements is desirable.
In that case the ratio
\begin{equation}\label{eqn:mesh-quality}
    \eta(K) \defeq \frac{h(k)}{2 \rho(K)} \in [1, \infty)
\end{equation}
is a measure for the quality of the element $K$: $\eta(K)$ is $1$ for equilateral triangles and becomes large for long-and-thin triangles.

Lastly, we define Lagrange finite element functions of order $k\ge 1$ by 
\ben\label{eq:V_h_k}
V_h^k(\Omega) := \{v\in C(\overbar{\Omega}):\; v_{|K}\in \mathcal P^k(K), \;\forall  K\in \mathcal T_h \}.	
\een
For all our examples we use the \texttt{Firedrake} finite element library \cite{Rathgeber2016}.

\subsection{Efficient calculation of shape derivatives}

Consider now a space of finite element vector fields given by the basis $\{\phibf_{i}\}_i$ defined on $\Omega$.
Denoting the basis functions on the deformed domain $\Omega_{\Vf}$ by $\{\phibf_{\Omega_\Vf, i}\}_i$ we observe that $\phibf_{\Omega_\Vf, i} = \phibf_{i} \circ (\id+\Vf)^{-1}$ and hence
\begin{equation}
    \partial J_\Omega(\Vf)(\phibf_i) = \Dr J(\Omega_\Vf)(\phibf_{\Omega_\Vf, i}) \quad \text{ for all } i.
\end{equation}
This means that the derivative can be assembled efficiently on the deformed domain $\Omega_\Vf$ and then be used to calculate the gradient which lives on $\Omega$.

In order to calculate the shape derivative in the direction of kernels we follow the same procedure as in \cite{paganini2017higher}:
denote the basis vector fields of the RKHS by $\Vk_j$ and let 
$\Ci$ be the interpolation matrix between kernels and finite elements, i.e. 
$\sum_i I_{i,j} \phibf_i \approx \Vk_j$ in some sense and as a result
\begin{equation}
    [\partial J_\Omega(\Vf)(\Vk_j)]_j \approx \Ci^\top [\partial J_\Omega(\Vf)(\phibf_i)]_i = \Ci^\top [\Dr J(\Omega_{\Vf})(\phibf_{\Omega_\Vf, i})]_i.
\end{equation}

\section{Numerical experiments}\label{sec:6}

\subsection{Levelset example in finite elements}
We begin by considering a simple levelset example: the goal is to shrink a circle to a clover-like shape, as seen in Figure \ref{fig:circle-clover-domain}. Mathematically this is achieved by minimising the shape function $J(\Omega) = \int_{\Omega} f\dr\Vx$ where $f$ is given by 
\begin{equation}
    \begin{aligned}
        f(x,y) =& (\sqrt{(x - a)^2 + b y^2} - 1) (\sqrt{(x + a)^2 + b y^2} - 1) \\
               &(\sqrt{b  x^2 + (y - a)^2} - 1)  (\sqrt{b  x^2 + (y + a)^2} - 1) - \varepsilon,
    \end{aligned}
\end{equation}
where $a=4/5$, $b=2$ and $\varepsilon=0.001$ 

This example illustrates two challenges that mesh deformation methods often face.
Firstly, the initial mesh has to be significantly compressed; if this compression is not performed evenly one can quickly obtain overlapping mesh elements.
Secondly, the final mesh contains areas of high curvature; this often leads  to highly stretched finite elements, which can be disadvantageous when a PDE is solved on the mesh.
\begin{figure}[H]
	\centering
	\includegraphics[width=0.8\textwidth]{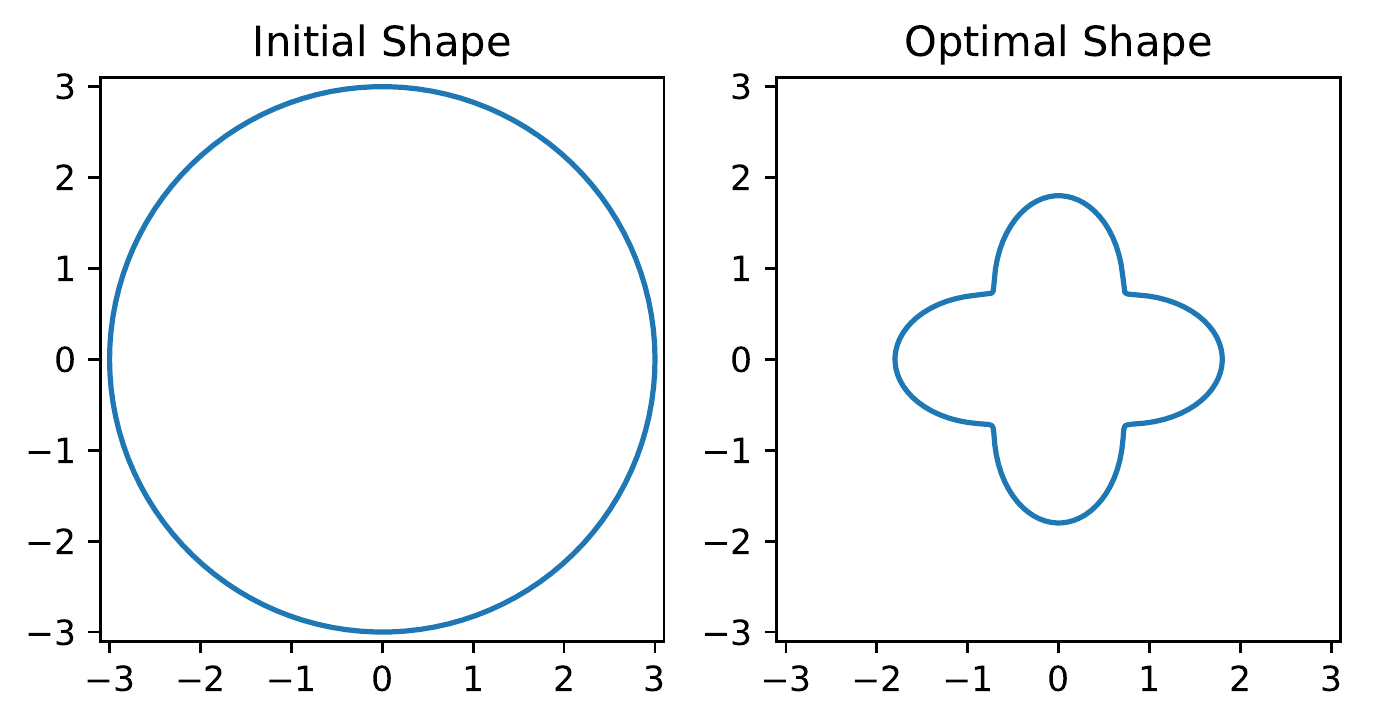}
	\caption{The initial shape (left) is a circle of radius 3; the optimal shape (right) is given by a clover-like shape.}
    \label{fig:circle-clover-domain}
\end{figure}

We compare four different inner products to calculate the shape gradients:
\begin{equation} \label{eq:four-inner-products-fem}
    \begin{aligned}
        (\Vu, \Vv)_{\mathring{H}^1} &= (\partial \Vu,\partial\Vv)_{[L^2]^{2\times2}},\\
        (\Vu, \Vv)_{\mathring{H}(\sym)} &= (\sym(\partial \Vu),\sym(\partial\Vv))_{[L^2]^{2\times2}},\\
        (\Vu, \Vv)_{\CR(\alpha)+\mathring{H}^1} &=\frac{1}{\alpha} (\Cb\Vu,\Cb\Vv)_{[L^2]^2} +(\partial \Vu,\partial\Vv)_{[L^2]^{2\times2}}, \\
        (\Vu, \Vv)_{\CR(\alpha)+\mathring{H}(\sym)} &=\frac{1}{\alpha} (\Cb\Vu,\Cb\Vv)_{[L^2]^2} + (\sym(\partial \Vu),\sym(\partial\Vv))_{[L^2]^{2\times2}}.
    \end{aligned}
\end{equation}

\subsubsection{Comparison of mesh quality for small $\alpha$}
We begin by choosing $\alpha=10^{-2}$; this leads to deformations that are close to perfectly conformal.
In Figure~\ref{fig:levelset-fem-optim-shapes} we show the optimal shape obtained from each of the four inner products in \eqref{eq:four-inner-products-fem}.
Using the first inner product, which corresponds to solving Laplace's equations with Neumann data given by the shape derivative, the optimisation algorithm fails after a few iterations as the mesh has degenerated (cf. the first row in Figure~\ref{fig:levelset-fem-optim-shapes}).
The inner-product that uses the symmetric part of the gradient performs significantly better and using it the L-BFGS algorithm converges to the expected optimal shape.
However, looking closer at the area of high curvature we can see that some of the elements have been significantly stretched (cf. the second row in Figure~\ref{fig:levelset-fem-optim-shapes}).
Adding the Cauchy-Riemann regularisation to each of the two inner products we can see that all triangles remain close to unilateral (cf. the third and fourth row in Figure~\ref{fig:levelset-fem-optim-shapes}).
The nearly conformal mappings achieve this by changing the size of the elements: elements are shrunk where large deformations are necessary and magnified elsewhere.

We can quantify these findings by considering a histogram plot of the element quality $\eta(K)=h(k)/(2\rho(K))$.
Figure \ref{fig:mesh-quality-levelset-histogram-fem} shows the element quality of the initial mesh as well as of the final meshes obtained from using the inner-products $(\cdot, \cdot)_{\mathring{H}(\sym)}$, $(\cdot, \cdot)_{\CR(\alpha)+\mathring{H}^1}$ and $(\cdot, \cdot)_{\CR(\alpha)+\mathring{H}(\sym)}$.
We omit the inner-product $(\cdot, \cdot)_{H^1}$ as in that case the mesh degenerates entirely.

When using $(\cdot, \cdot)_{\mathring{H}(\sym)}$, the mesh quality decreases overall and a significant number of elements has a ratio $\eta(K)$ of greater than $2$.
Contrary to that, when considering the nearly conformal mappings obtained from our proposed inner-products, the quality for each element stays essentially constant and no degradation is visible.
\begin{figure}[H]
    \centering
    \includegraphics[width=5cm]{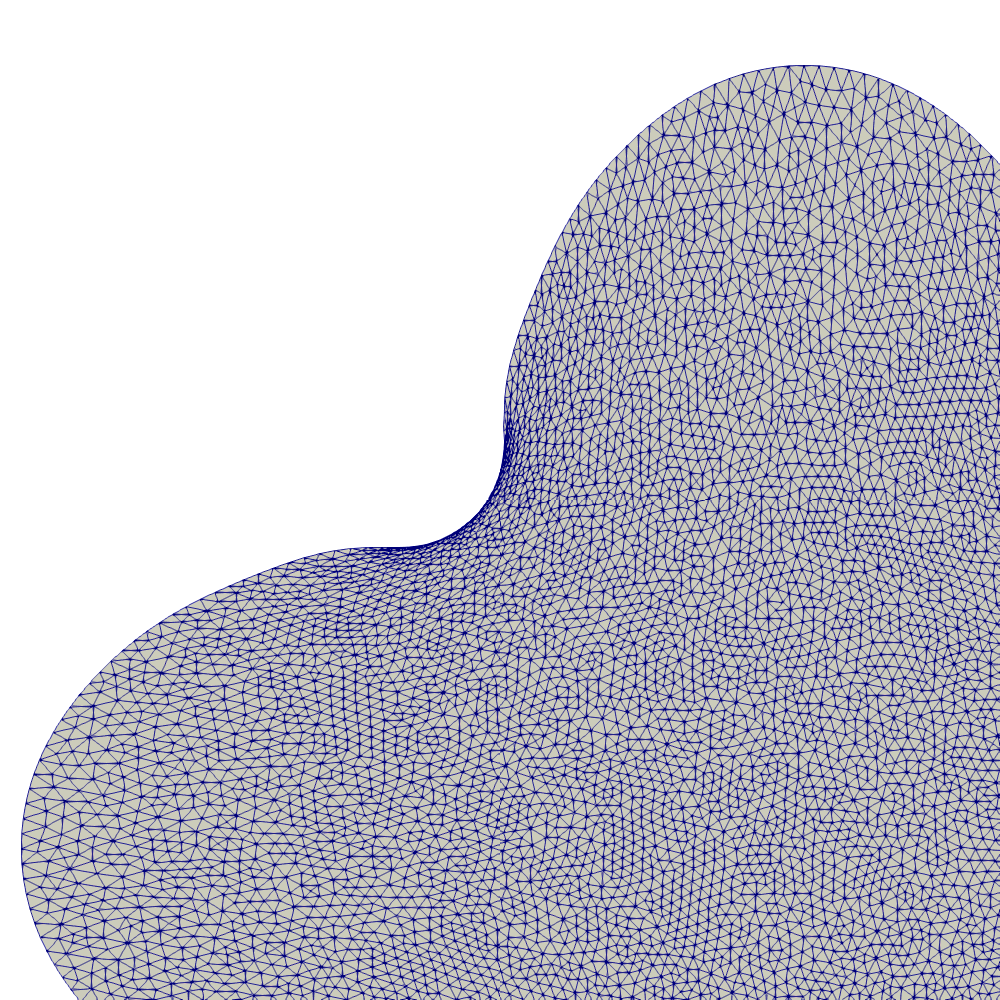}
    \includegraphics[width=5cm]{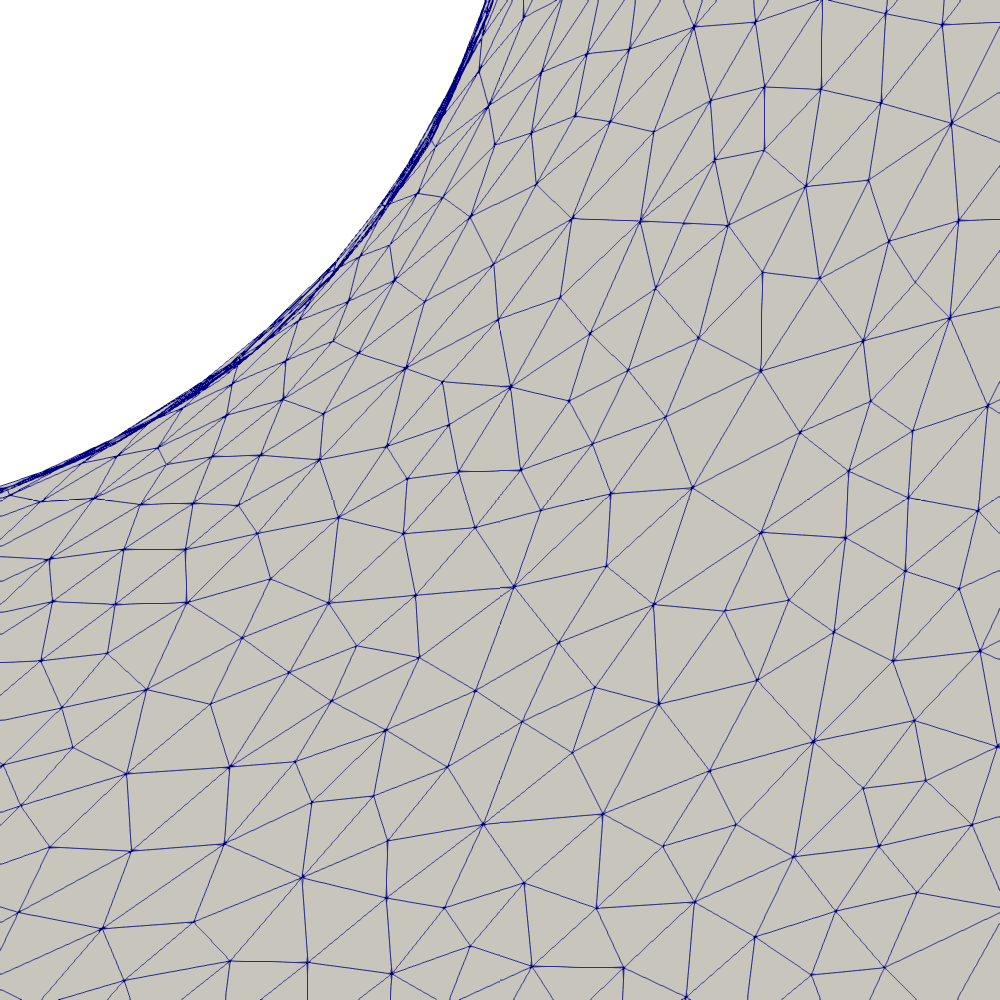}
    \includegraphics[width=5cm]{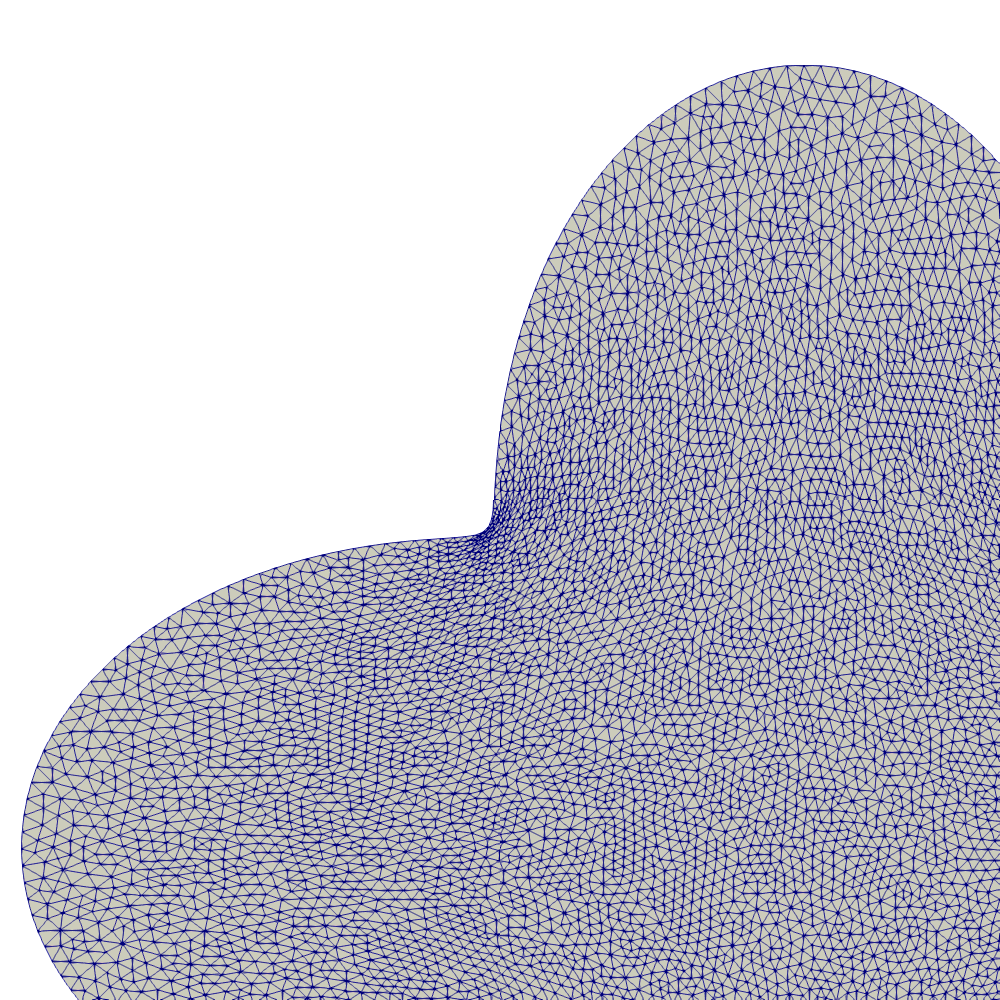}
    \includegraphics[width=5cm]{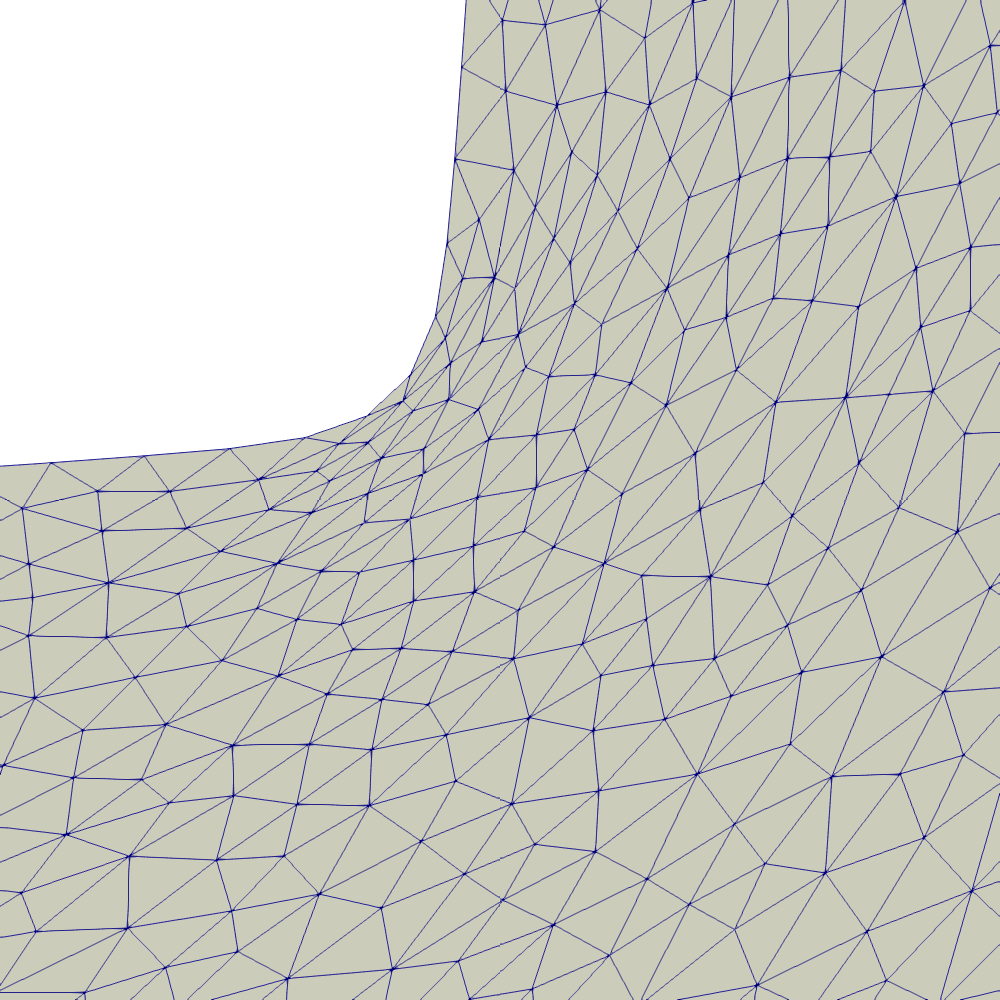}
    \includegraphics[width=5cm]{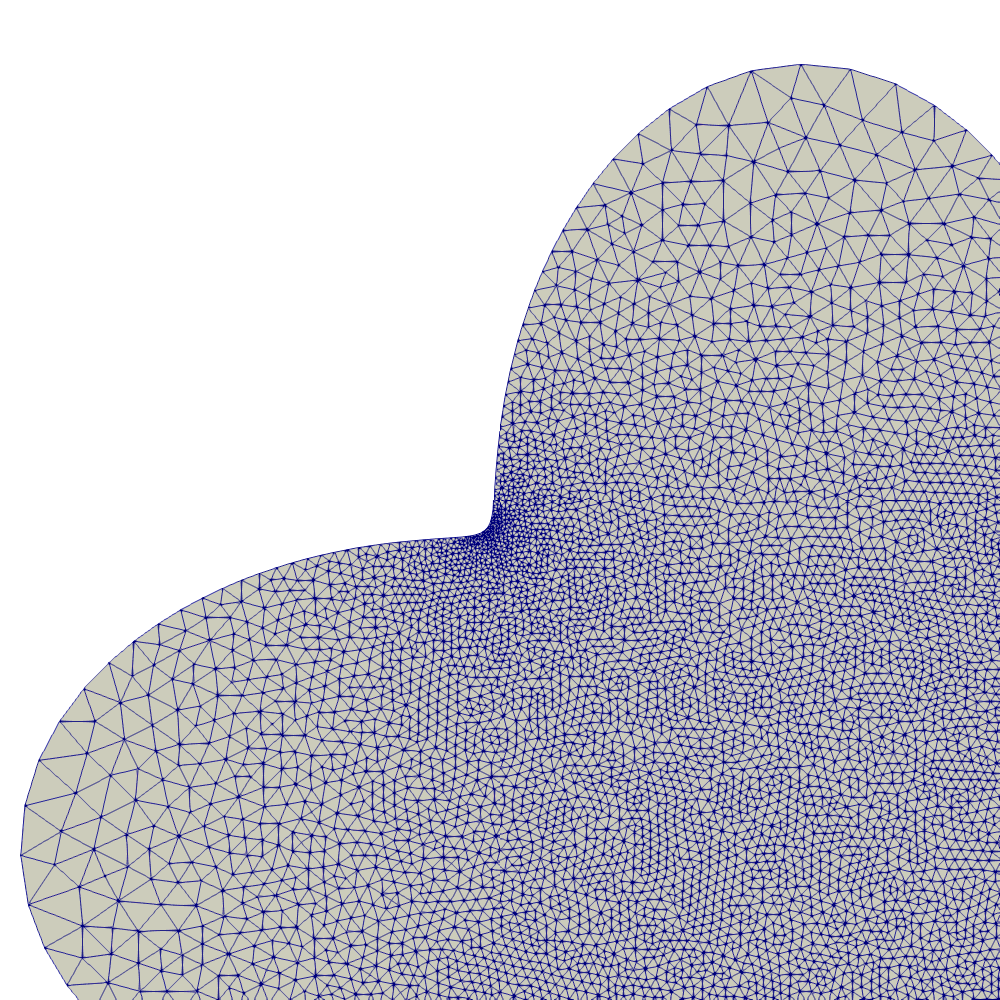}
    \includegraphics[width=5cm]{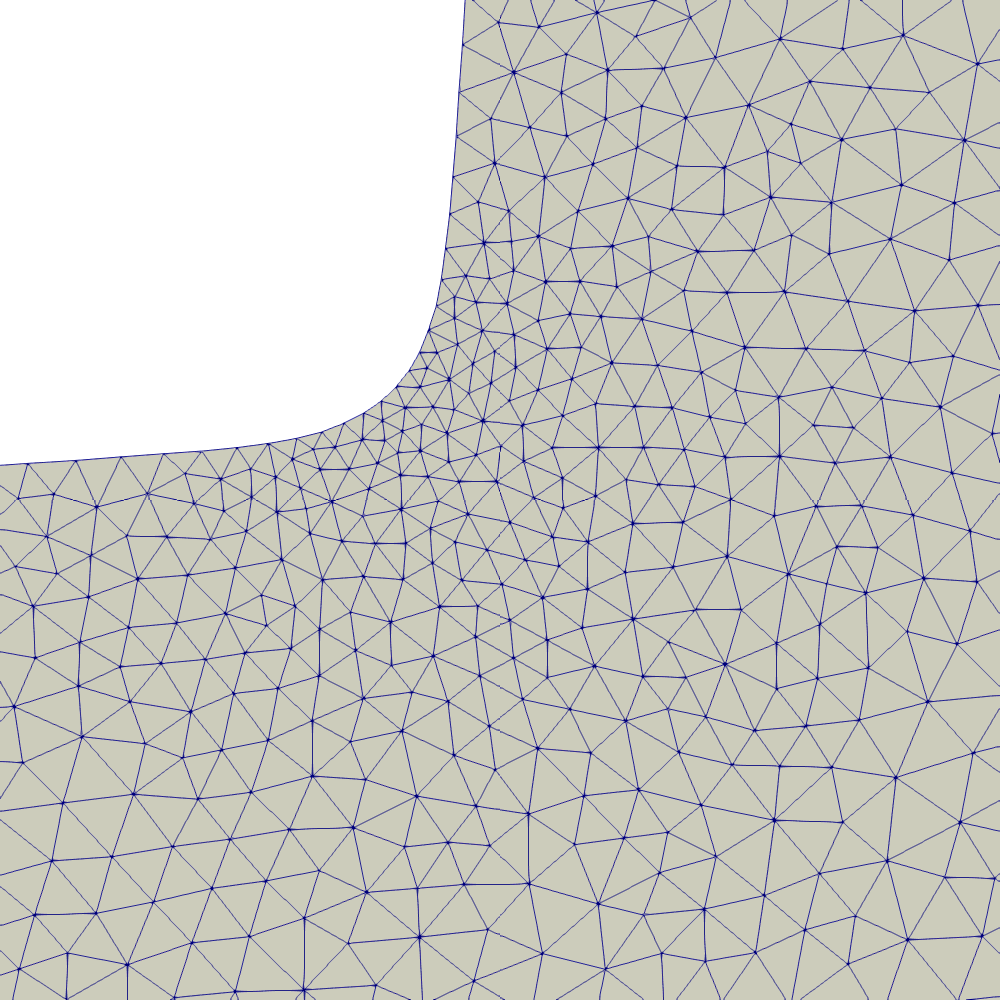}
    \includegraphics[width=5cm]{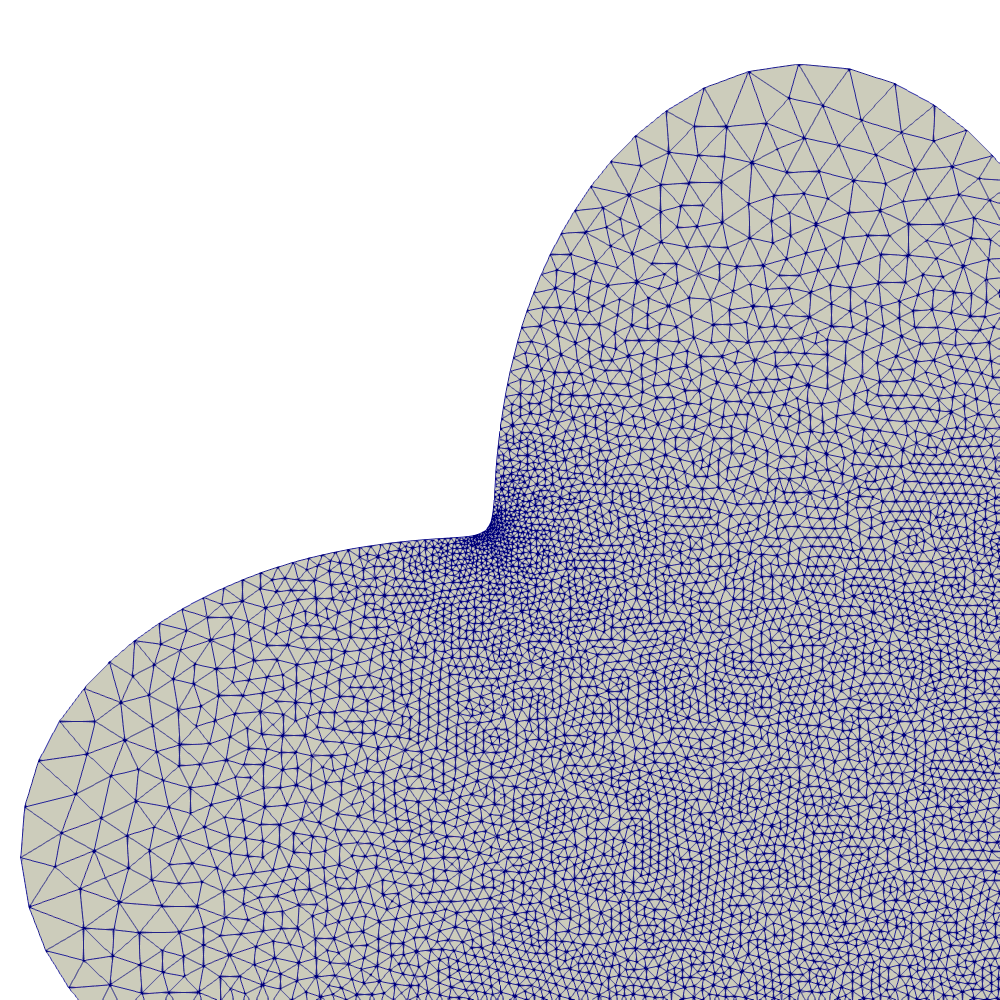}
    \includegraphics[width=5cm]{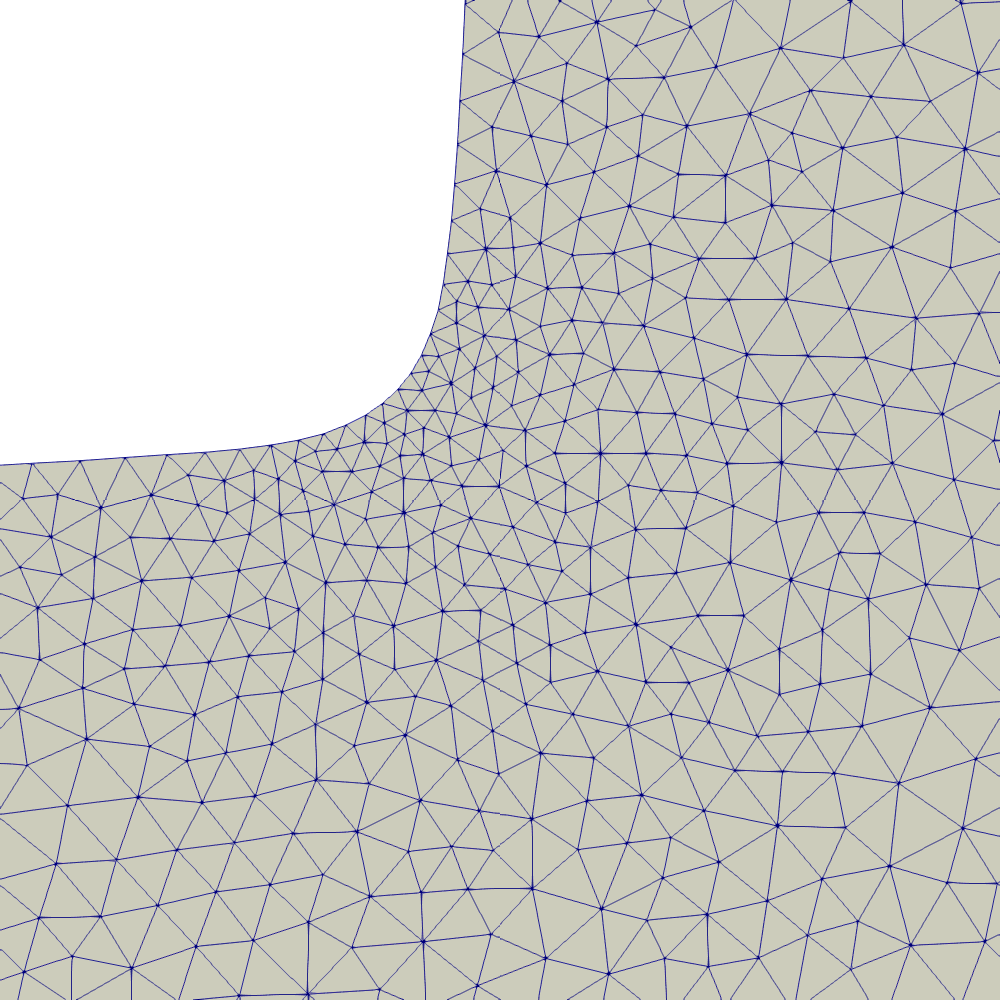}
    \caption{Optimal shapes obtained using the different inner-products. 
        First row: $(\cdot, \cdot)_{\mathring{H}^1}$ (last shape before the mesh degenerates shown).
        Second row: $(\cdot, \cdot)_{\mathring{H}(\sym)}$.
        Third row: $(\cdot, \cdot)_{\CR(10^{-2})+ \mathring{H}^1}$.
        Fourth row: $(\cdot, \cdot)_{\CR(10^{-2})+ \mathring{H}(\sym)}$.
    }
    \label{fig:levelset-fem-optim-shapes}
\end{figure}
\begin{figure}[H]
    \centering
    \includegraphics[width=6.2cm]{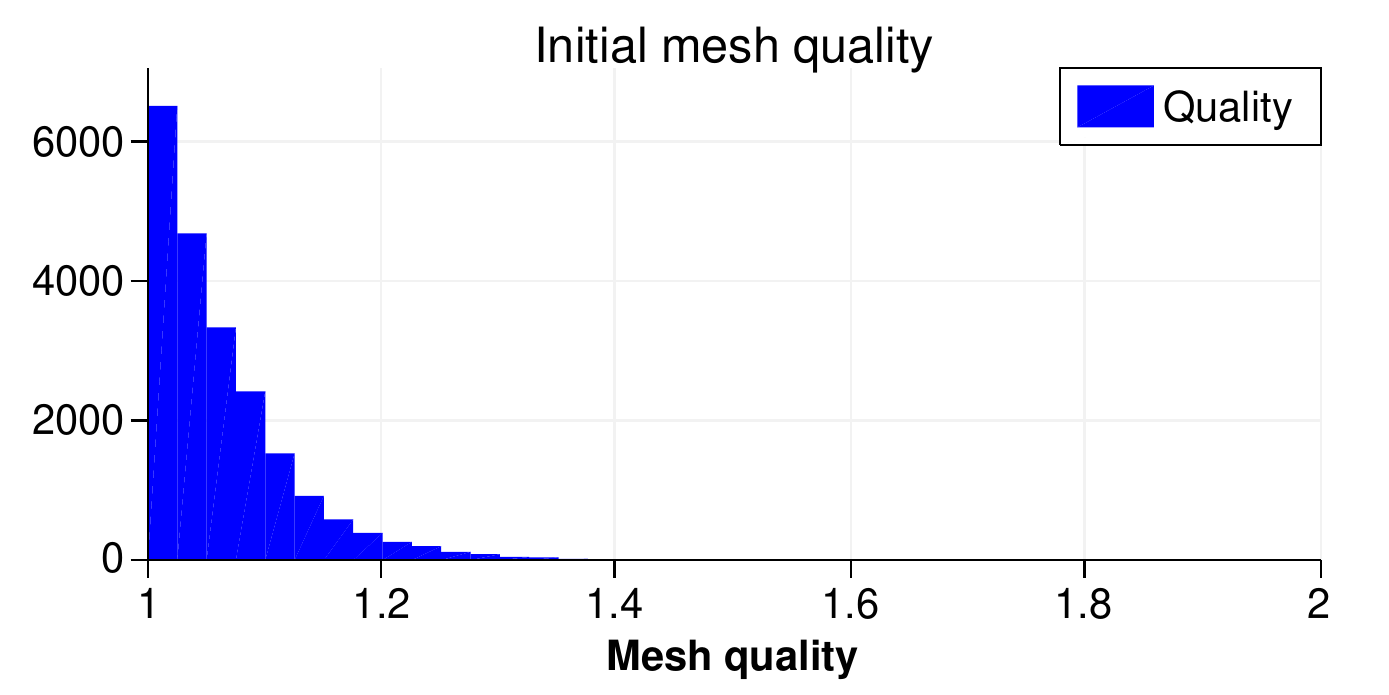}
    \includegraphics[width=6.2cm]{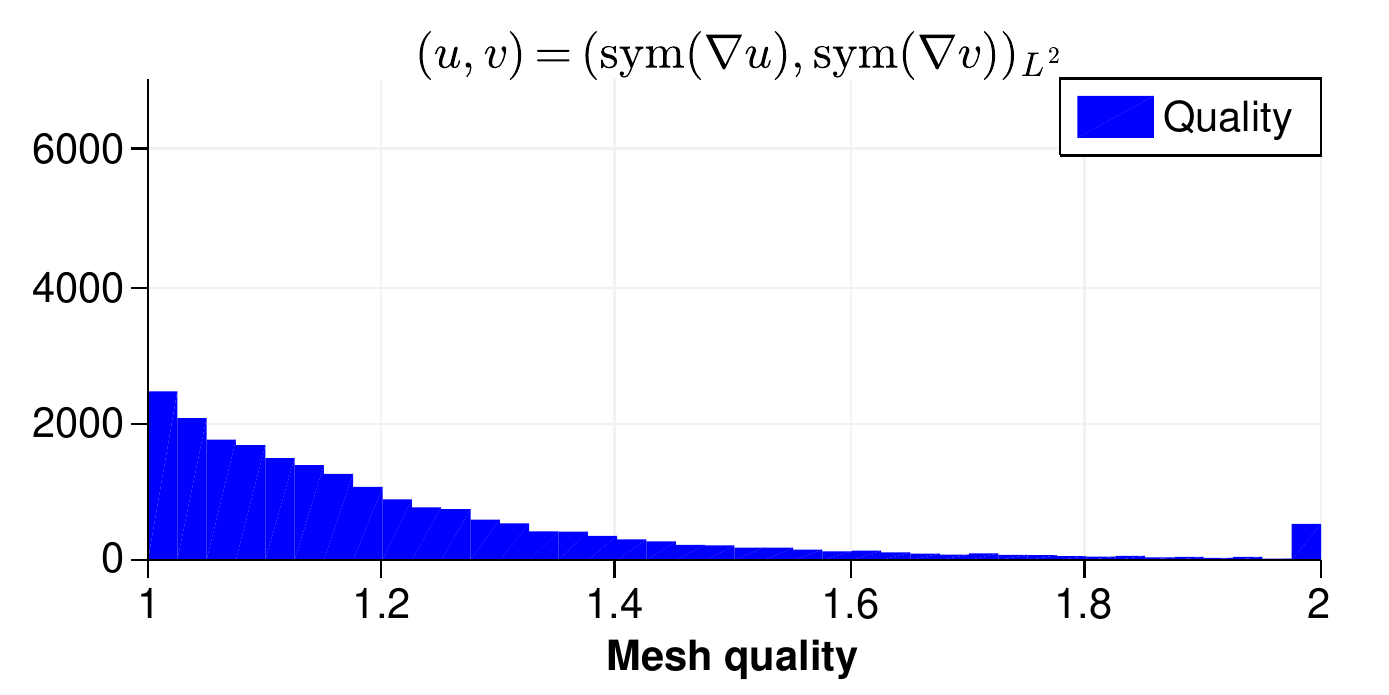}
    \includegraphics[width=6.2cm]{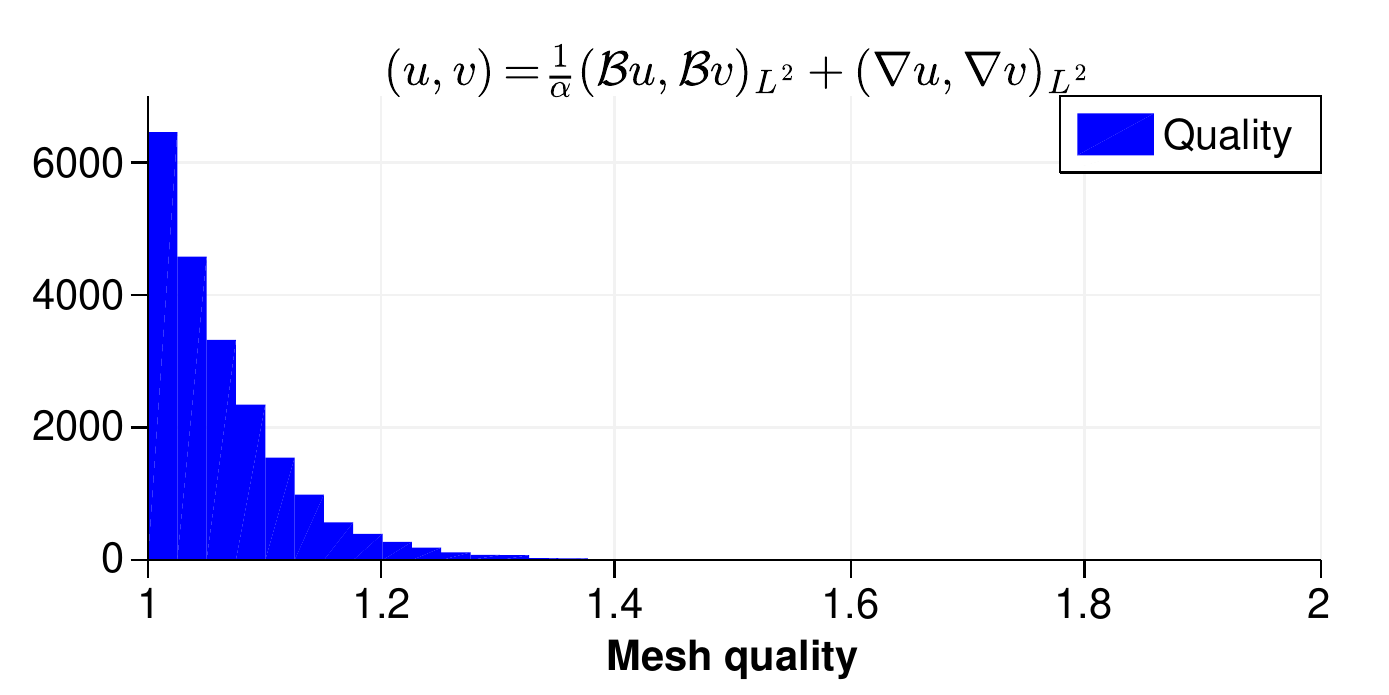}
    \includegraphics[width=6.2cm]{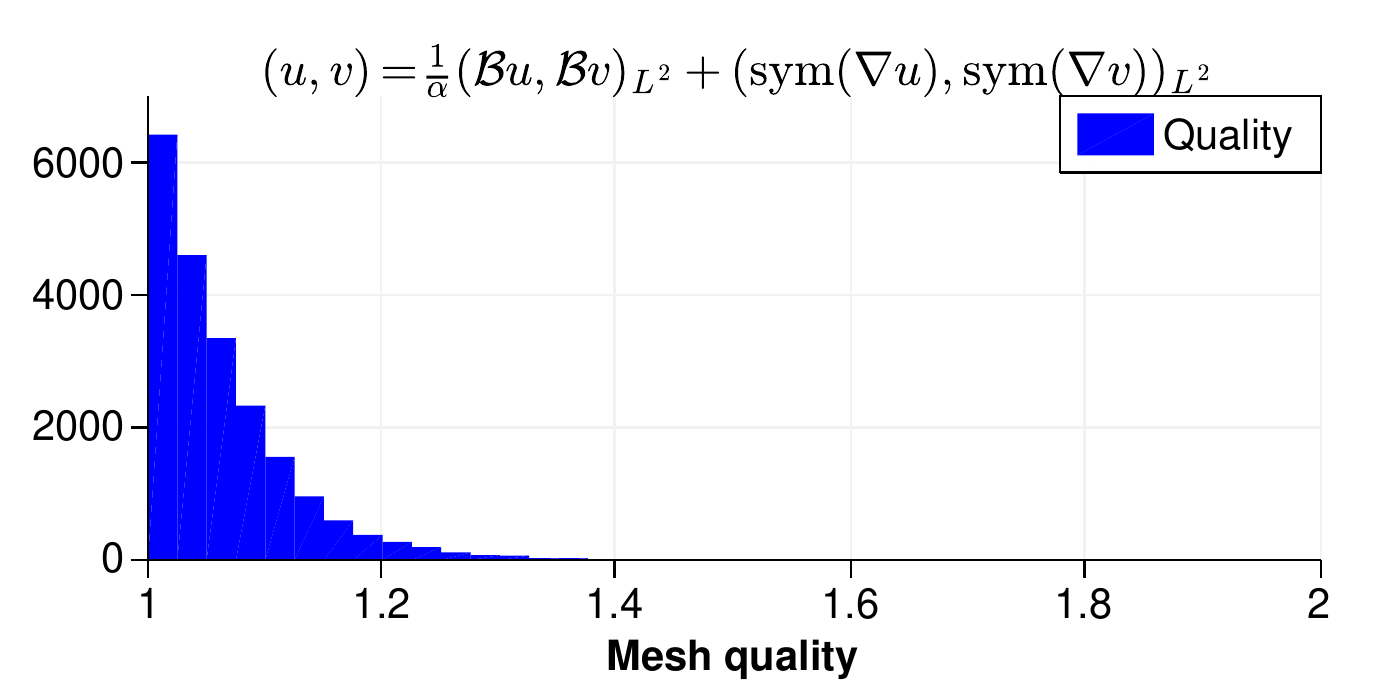}
    \caption{Mesh quality as measured via \eqref{eqn:mesh-quality}. 
        Top left: initial mesh.
        Top right: $(\cdot, \cdot)_{\mathring{H}(\sym)}$.
        Bottom left: $(\cdot, \cdot)_{\CR(10^{-2})+ \mathring{H}^1}$.
        Bottom right: $(\cdot, \cdot)_{\CR(10^{-2})+ \mathring{H}(\sym)}$.
    }
    \label{fig:mesh-quality-levelset-histogram-fem}

\end{figure}
\subsubsection{Behaviour for larger $\alpha$}
In order to achieve deformations that are nearly conformal, Proposition~\ref{prop:solution-decomposition} says that we should choose $\alpha$ very small. 
However, while this does indeed lead to very good meshes, there are two disadvantages to picking $\alpha$ too small.
First, the linear mapping associated with the Riesz map that is used becomes poorly conditioned and calculating the gradient becomes more computationally expensive.
However, this is usually not a problem as the bottleneck for most applications lies in the calculation of the state and the adjoint equation and the calculation of the gradient is comparatively cheap.
\begin{figure}[H]
    \centering
    \includegraphics[width=12cm]{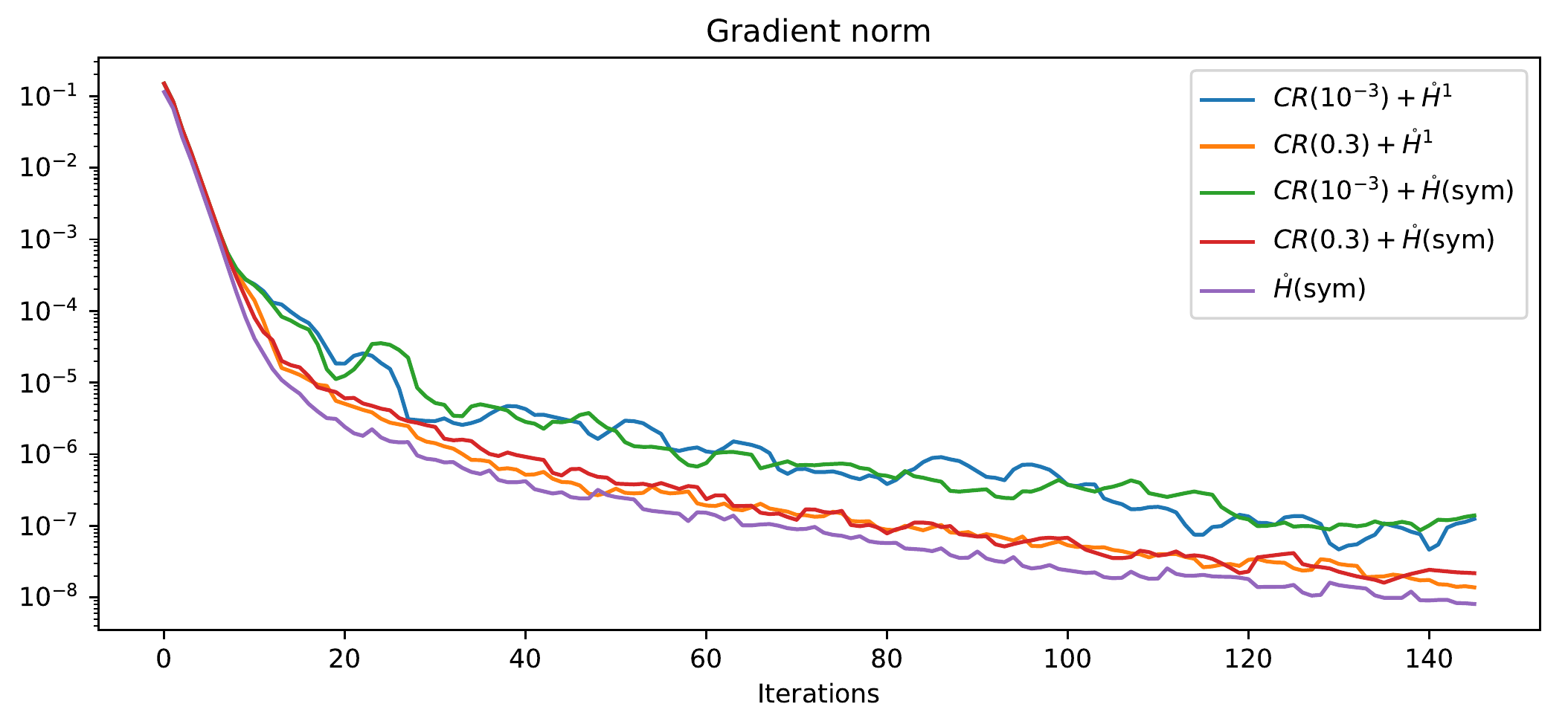}
    \caption{Convergence history of the L-BFGS algorithm with memory $m=5$. To improve readability we plot the moving average of the gradient over five iterations.}
    \label{fig:levelset-fem-convergence-speed}
\end{figure}

Second, restricting the optimisation to mappings that are almost entirely conformal means that the algorithm will need more steps to converge; this can be seen in Figure~\ref{fig:levelset-fem-convergence-speed}.
For $\alpha=10^{-3}$, the gradient is approximately  two orders of magnitude larger than without the added Cauchy-Riemann terms.
As $\alpha$ increases, this difference becomes smaller.
While for small $\alpha$ the two inner products $(\cdot, \cdot)_{\CR(\alpha)+\mathring{H}^1}$ and $(\cdot, \cdot)_{\CR(\alpha)+\mathring{H}(\sym)}$ performed almost identically in terms of mesh quality, we expect to see a difference as $\alpha$ becomes bigger.
In Figure~\ref{ref:levelset-fem-large-alpha} we show a close-up of the high curvature area for comparatively large $\alpha=0.3$.
The nearly conformal mappings using the inner-product $(\cdot, \cdot)_{\CR(\alpha)+\mathring{H}(\sym)}$ show an excellent mesh quality, while the ones based on $(\cdot, \cdot)_{\CR(\alpha)+\mathring{H}^1}$ lead to stretched elements.
As Figure~\ref{fig:levelset-fem-convergence-speed} shows no difference in speed for the two variants, we will use the inner-product based on the symmetric gradient from now on.
\begin{figure}[h]
    \centering
    \includegraphics[width=6cm]{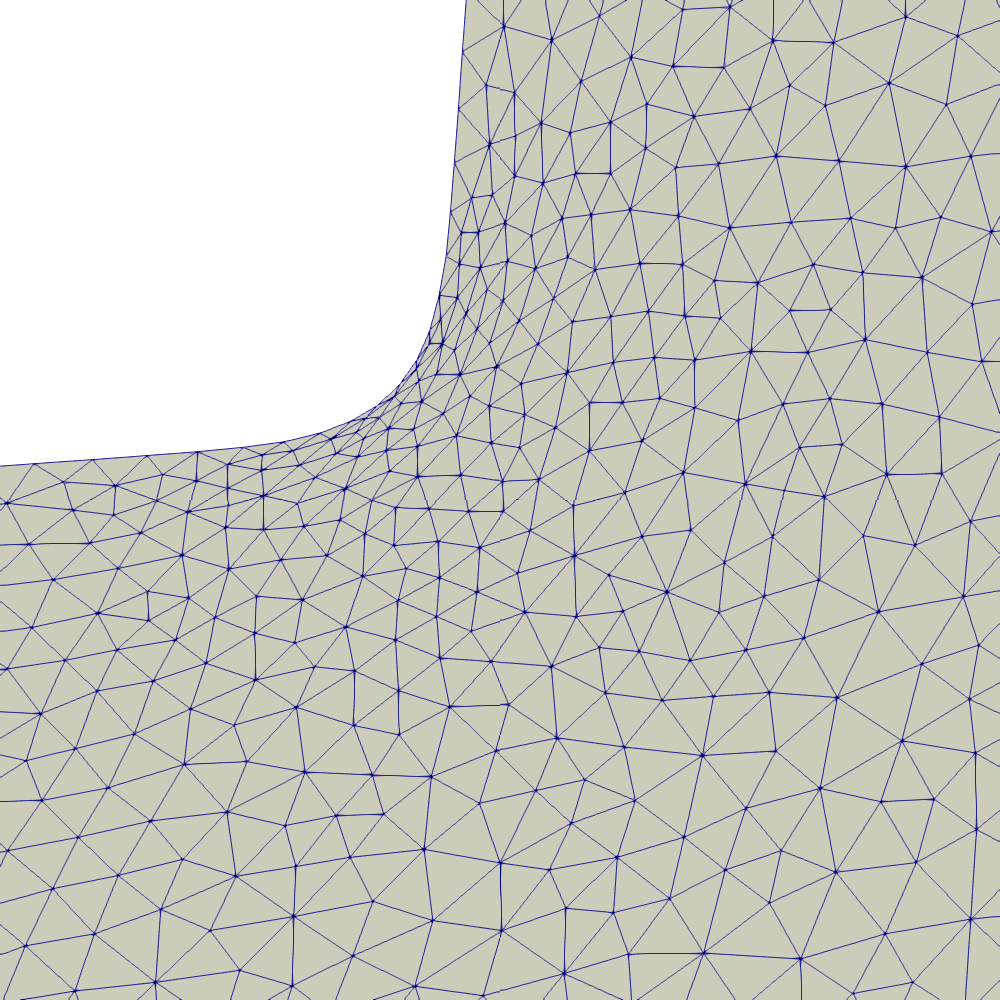}
    \includegraphics[width=6cm]{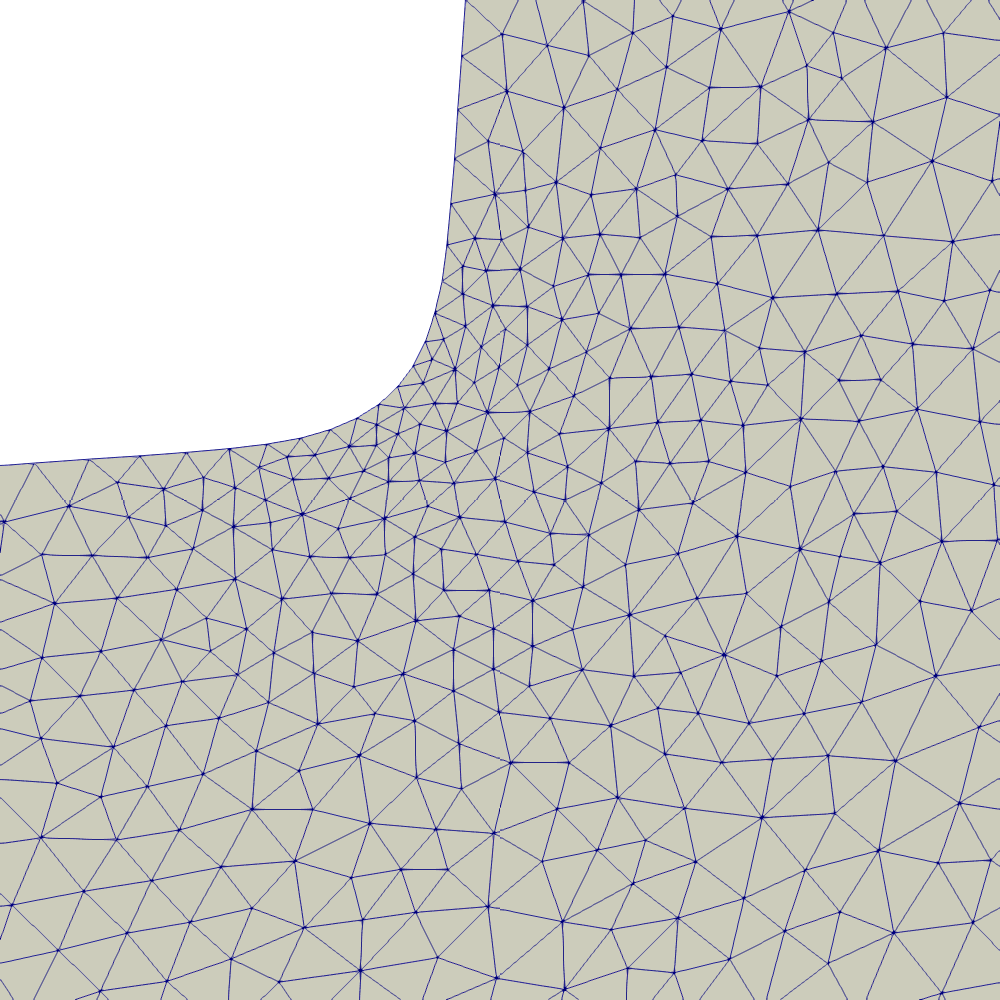}
    \caption{
        Close up of the obtained shape using $(\cdot, \cdot)_{\CR(\alpha)+\mathring{H}^1}$ (left) and $(\cdot, \cdot)_{\CR(\alpha)+\mathring{H}(\sym)}$ (right) for $\alpha=0.3$.
    }
    \label{ref:levelset-fem-large-alpha}
\end{figure}
\FloatBarrier
\subsection{Levelset example in RKHS}
We now consider the same shape optimisation problem but we discretise the deformations using the Wendland kernel in \eqref{eqn:wendland-kernel-definition}.
The levelset function is again discretised using finite element functions on a triangular mesh.
The vertices of this mesh are used as the points $p_\ell$ at which we enforce conformality in \eqref{eq:min_CR_relaxed_pw}.

Associating the vector $\Vu_h$ and $\Vv_h$ with functions $\Vu$ and $\Vv$ in $[\Ch(\Omega)]^2$, the inner-products that we compare are given by 
\begin{equation}
    \begin{aligned}
        (\Vu, \Vv)_{[\Ch]^2} &= \Vu_h^\top K \Vv_h,\\
        (\Vu, \Vv)_{\CR(\alpha)+ [\Ch]^2} &= \frac{1}{\alpha}\frac{1}{n}\sum_{k=1}^n \Cb \Vu(\Vp_k)\cdot \Cb\Vv(\Vp_k)  + (\Vu, \Vv)_{[\Ch]^2},
    \end{aligned}
\end{equation}
where $K$ is the matrix defined in \eqref{eq:matrices_pw}.

In Figure~\ref{fig:levelset-kernel-optim-shapes} we can see the optimal shapes obtained; as before we can see that adding the Cauchy-Riemann term to the inner-product leads to a significant improvement.
\begin{figure}
    \centering
    \includegraphics[width=5cm]{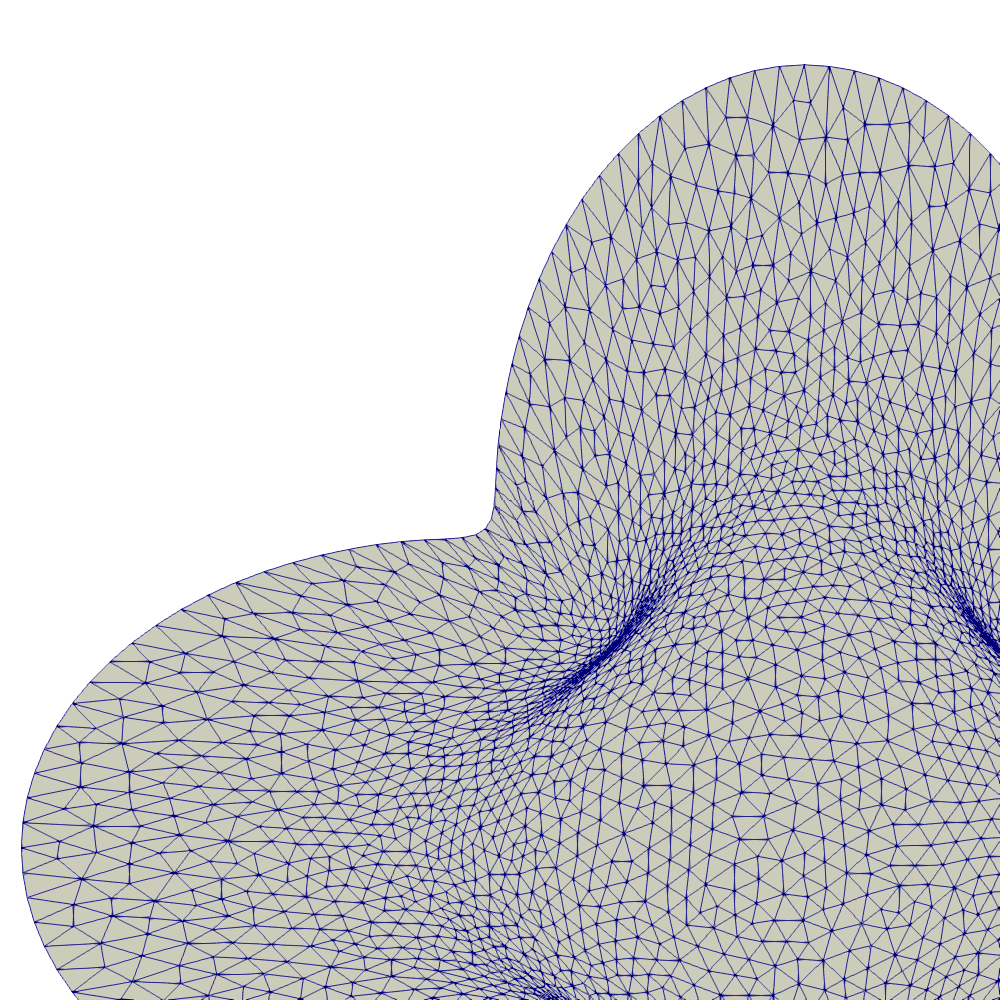}
    \includegraphics[width=5cm]{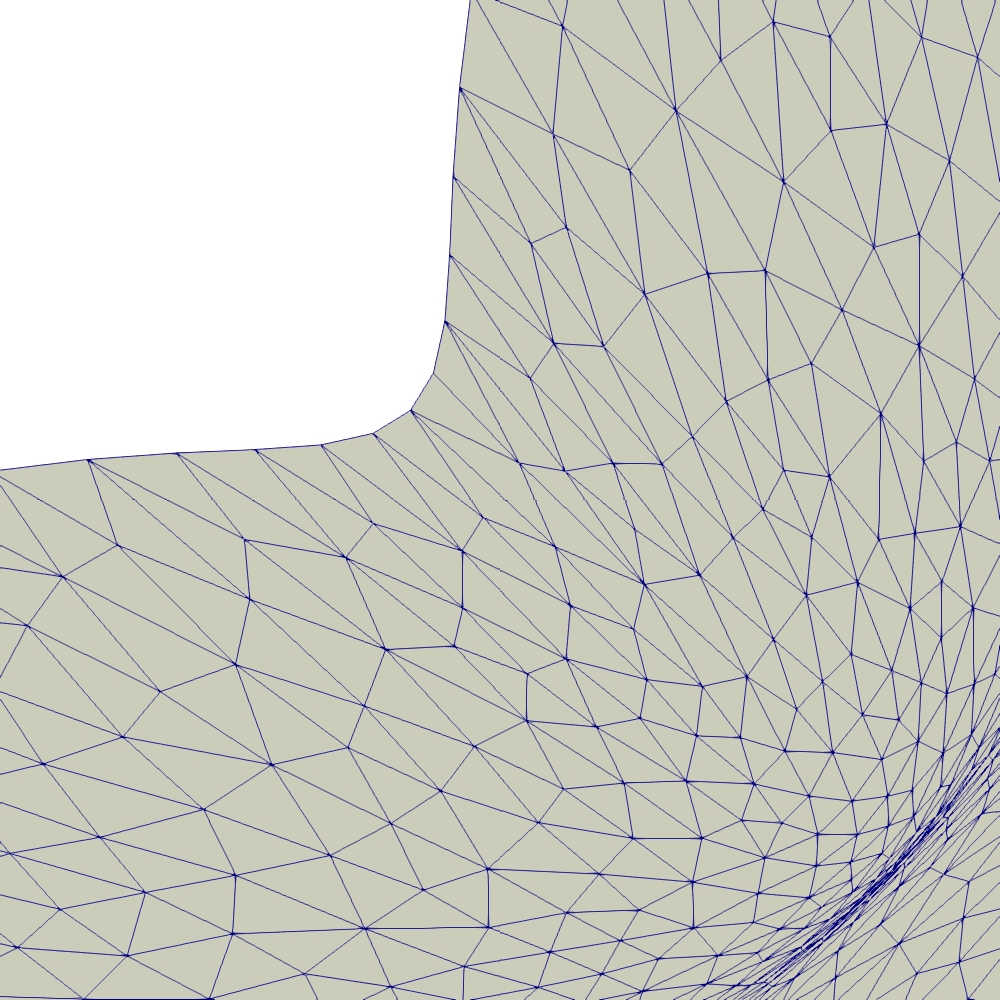}
    \includegraphics[width=5cm]{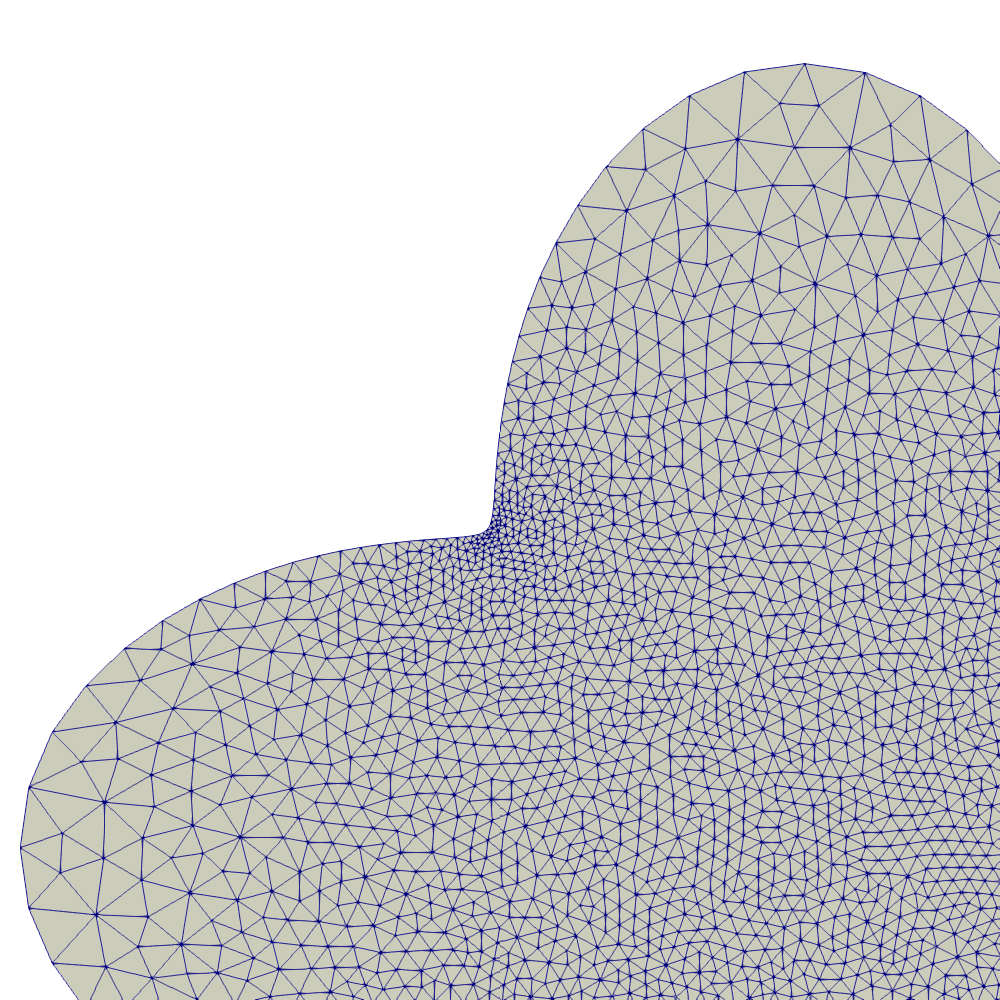}
    \includegraphics[width=5cm]{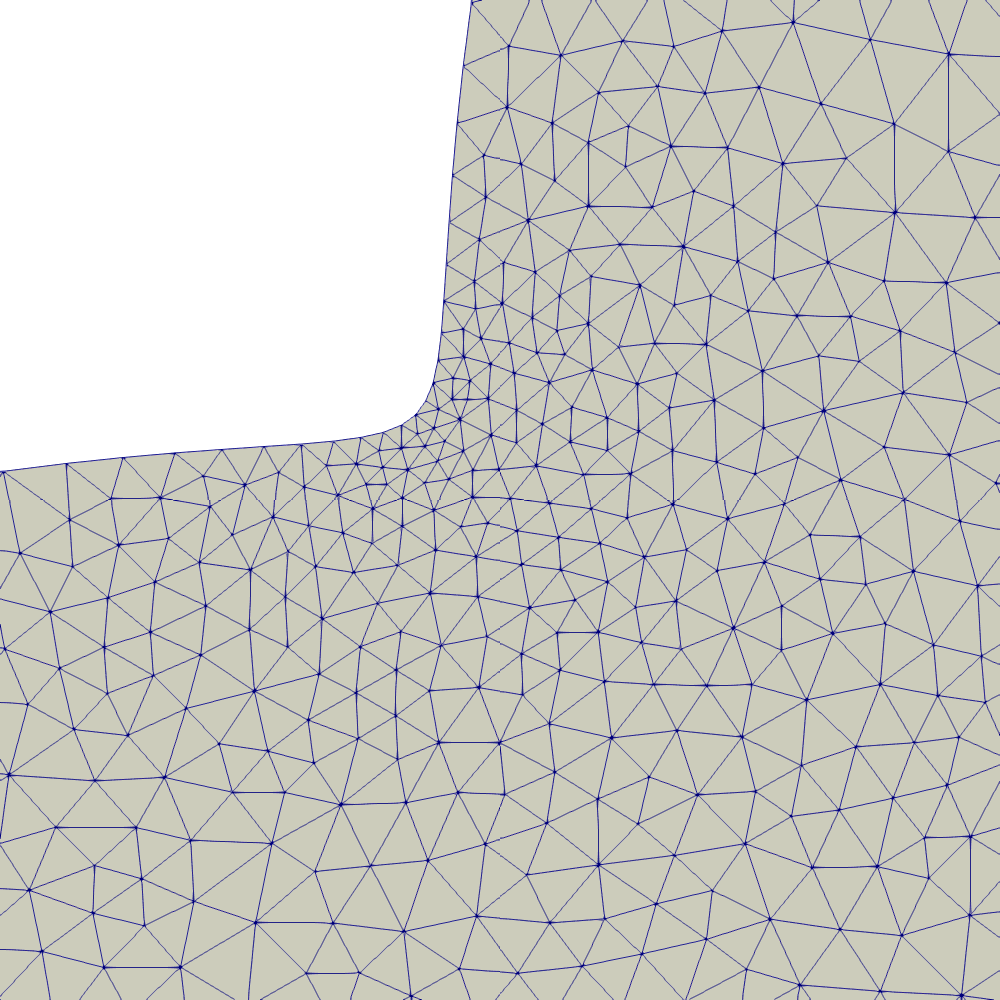}
    \caption{Optimal shapes obtained using the inner product  $(\cdot, \cdot)_{[\Ch]^2}$ (top row) and $(\cdot, \cdot)_{\CR(\alpha)+[\Ch]^2}$ for $\alpha=10^{-4}$ (bottom row).}
    \label{fig:levelset-kernel-optim-shapes}
\end{figure}
To quantify this we look at the distribution of element quality and we can see that the quality remains essentially unchanged for the new inner product, whereas it significantly degrades when using the inner product induced by the kernels alone, see Figure~\ref{fig:levelset-kernel-mesh-quality}.
\begin{figure}[H]
    \centering
    \includegraphics[width=6.2cm]{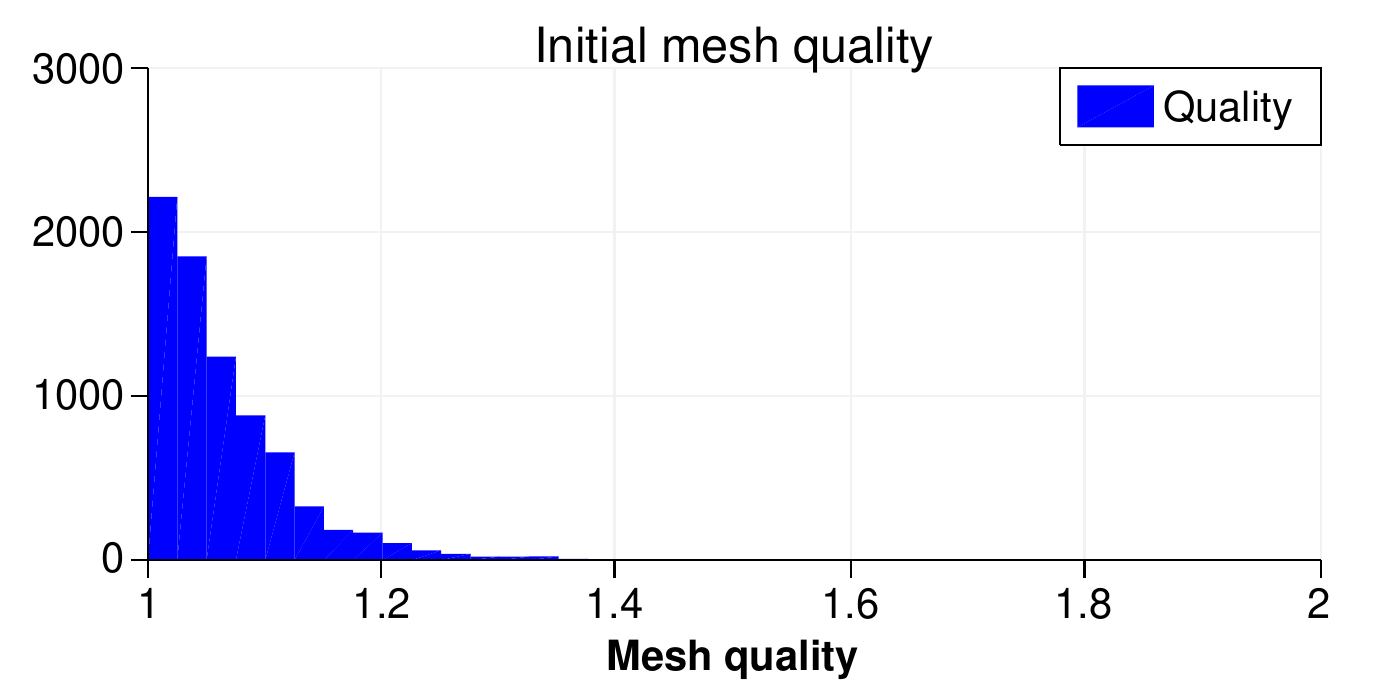}
    \includegraphics[width=6.2cm]{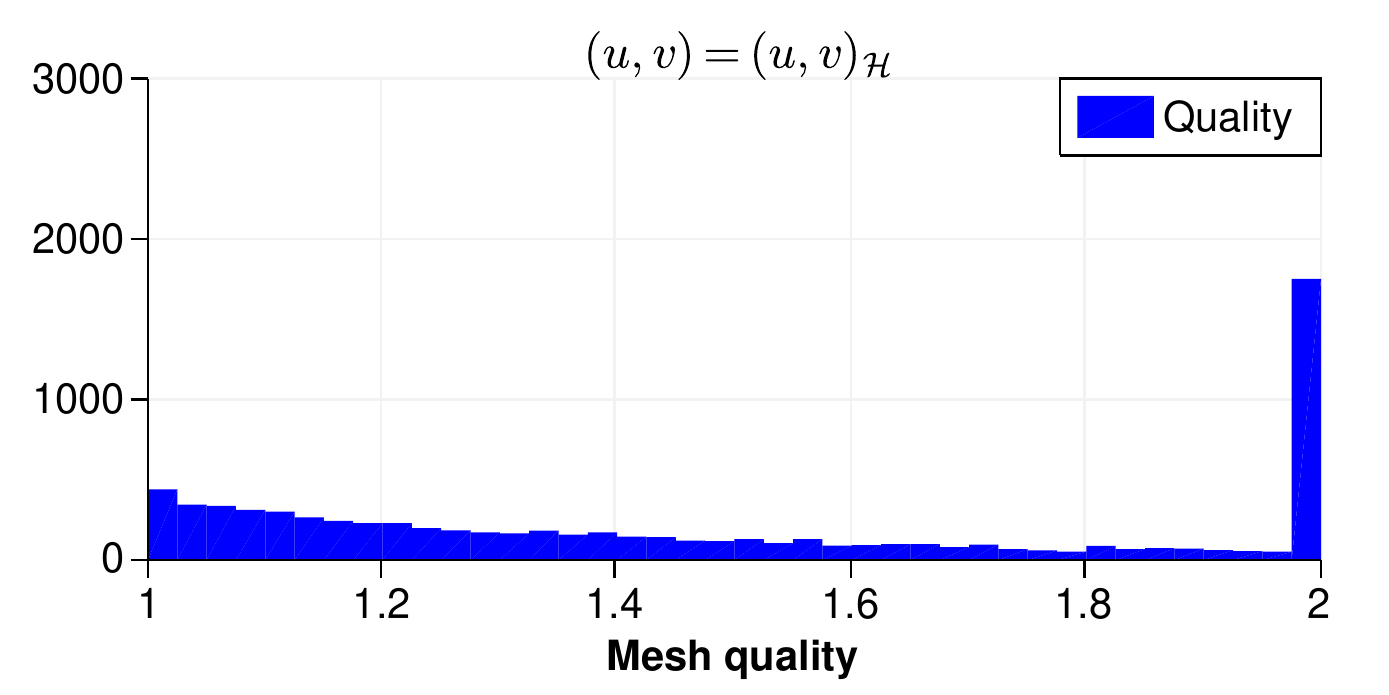}
    \includegraphics[width=6.2cm]{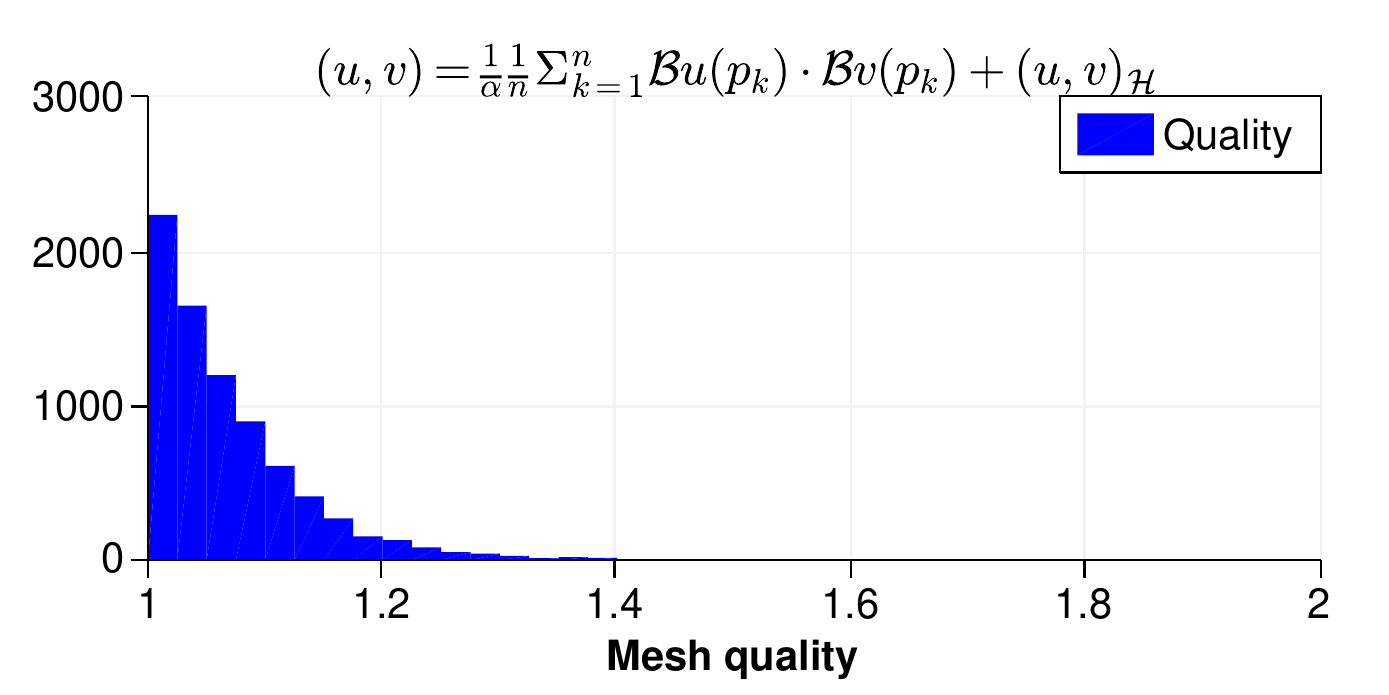}
    \caption{Mesh quality as measured via \eqref{eqn:mesh-quality}. 
        Top row, left: initial mesh.
        Top row, right: final mesh obtained using $(\cdot, \cdot)_{[\Ch]^2}$.
        Bottom row: final mesh obtained using $(\cdot, \cdot)_{\CR(\alpha)+[\Ch]^2}$ for $\alpha=10^{-4}$.
    }
    \label{fig:levelset-kernel-mesh-quality}
\end{figure}
\subsection{A negative example: annulus deformation}
We saw in the previous section that if a conformal mapping between the initial and the optimal shape exists, then our proposed inner-products yield excellent mesh quality.
However, the Riemannian mapping does not hold for domains that are not simply connected.
The simplest example is that of two annuli:
let $A(r,R) := \{ \Vz:\; r<|\Vz|<R\}$ denote an annulus of outer radius $R$ and inner radius  $r$.
Schottky \cite{schottky1877konforme} proved that the annulus $A(r,R)$ can be conformally mapped to the annulus $A(r', R')$ if and only if $\frac{R}{r} = \frac{R'}{r'}$.

In regard of mesh deformation, this is a somewhat sobering result as it implies, that two annuli that do not have the same radii-ratio can never be mapped to each other without loss of mesh quality. 
To illustrate this, we pick $R=R'=1$ and $r=\frac{1}{2}$.
Now for $r'\in(0,1)$ we define $f_{r'}(\Vx) = |\Vx-1||\Vx-r'|$ and the corresponding levelset function $J_{r'}(\Omega) \defeq \int_{\Omega} f_{r'}(\Vx)\dr \Vx$.
The annulus $A(r', R')$ is the global minimiser of this function.

We choose the inner-product given by 
\begin{equation}
    (\Vu, \Vv)_{\CR(10^{-2},\mu)+\mathring{H}(\sym)} = 10^2(\mu \Cb\Vu,\mu\Cb\Vv)_{[L^2]^2} + (\sym(\partial\Vu),\sym(\partial\Vv))_{[L^2]^{2\times 2}}
\end{equation}
for this example.
Figure~\ref{fig:annulus-optimal-shapes} show the initial shape as well as the optimal shapes for values $r'=0.7$ and $r'=0.85$.
We can see that as the inner radius is increased, the mesh elements are compressed more and more.

Computations using the standard $[H^1(\Omega)]^2$ or $\mathring{H}^1(\sym, \Omega)$ inner-products lead to the same observation, and in fact, in \cite{IWKOON11} it is proved that there exists a harmonic homeomorphism $A(r,R) \to A(r_*, R_*)$ if and only if $\frac{R_*}{r_*} \ge \frac12\left( \frac{R}{r}+ \frac{r}{R} \right)$. This shows that while for simply-connected domains large deformations can be performed without loss of mesh quality, once the topology of the domain changes the problem of mesh deformation becomes significantly harder.
\begin{figure}[H]
    \centering
    \includegraphics[width=0.32\textwidth]{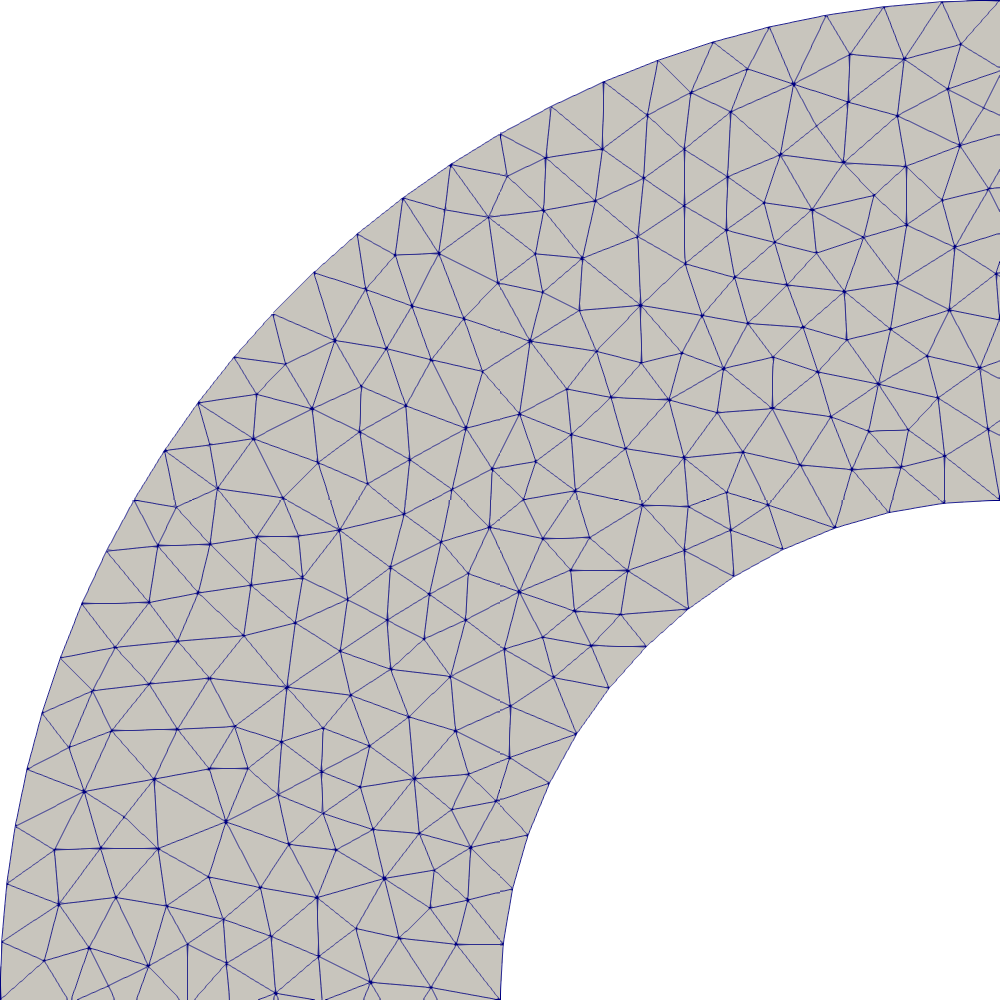}
    \includegraphics[width=0.32\textwidth]{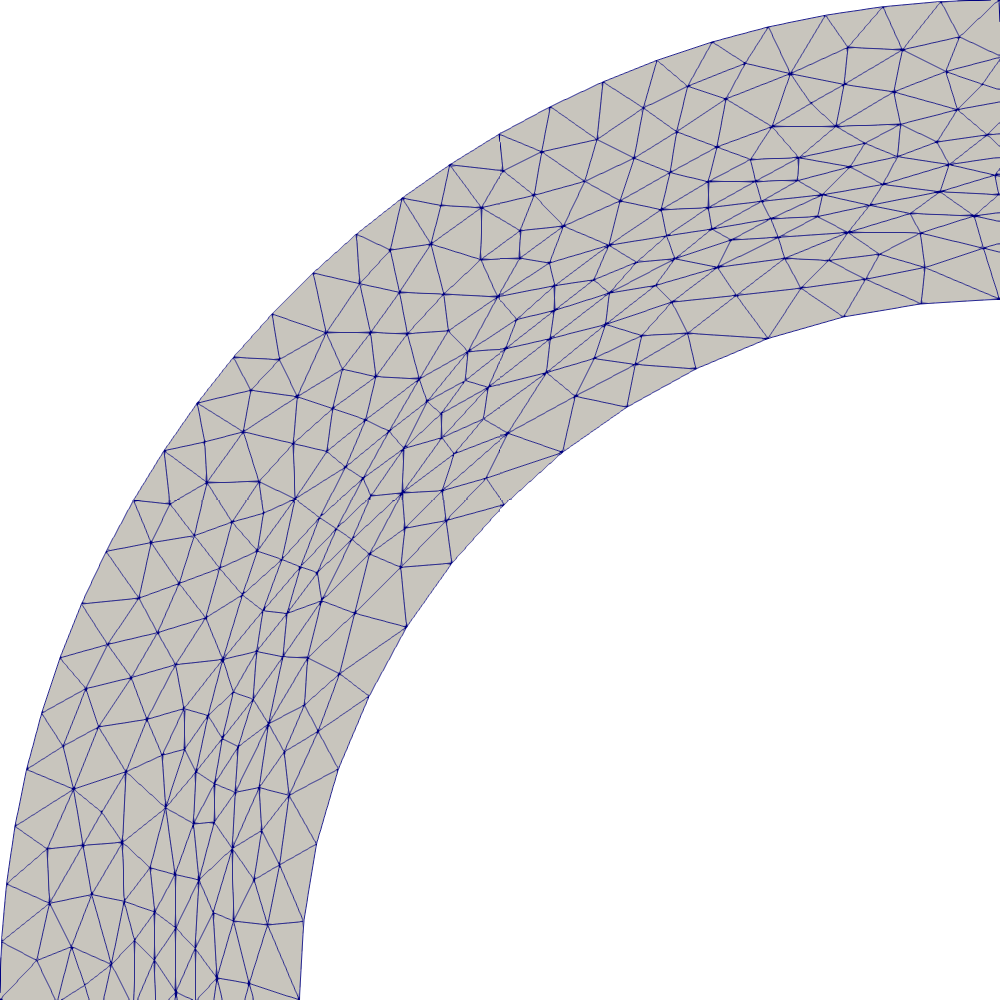}
    \includegraphics[width=0.32\textwidth]{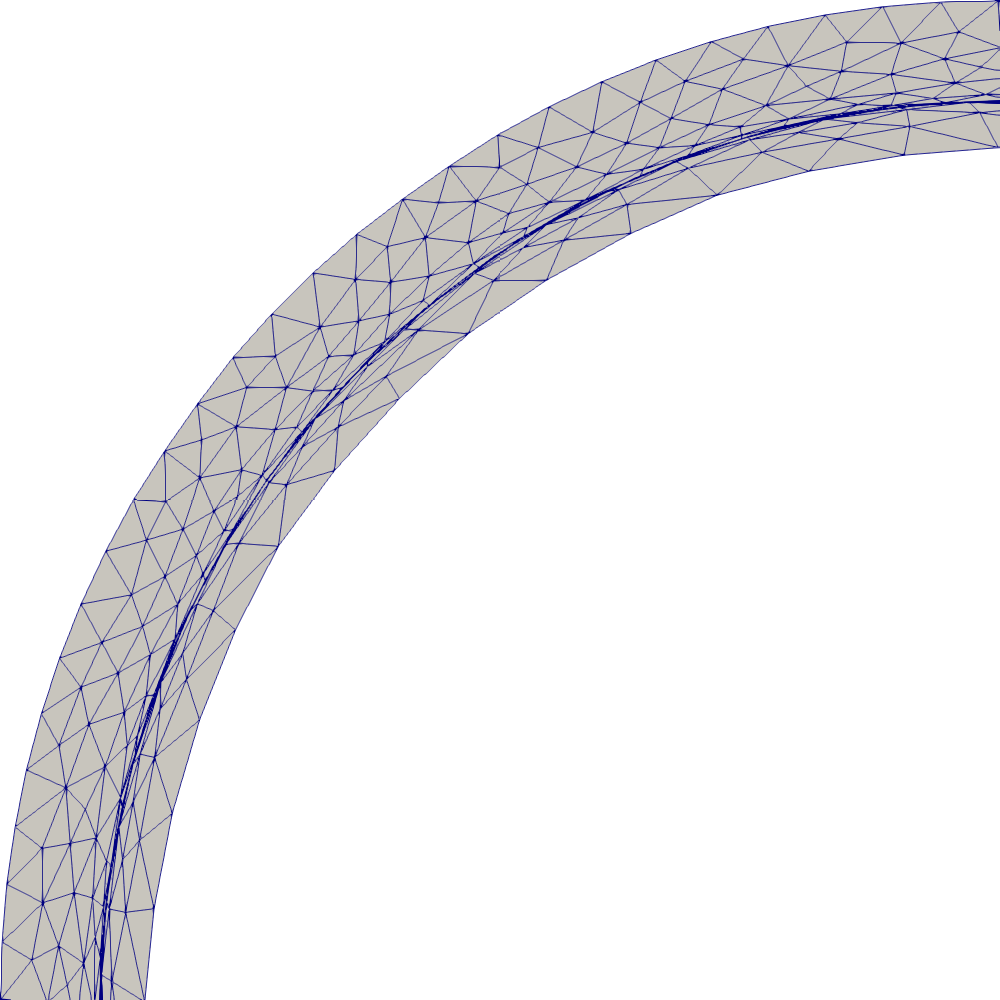}
    \caption{Left: initial shape $A(1, 1/2)$. Middle: $A(1, 0.7)$. Right: $A(1, 0.85)$. The two deformed annuli were obtained using the inner-product $(\cdot, \cdot)_{\CR(\alpha, \mu) + \mathring{H}(\sym)}$.}
    \label{fig:annulus-optimal-shapes}
\end{figure}

\subsection{Energy minimisation in slow flow}
We now consider a classical example in shape optimisation: the minimisation of dissipated energy in a fluid governed by Stokes' equations.
Let $D\subset\VR^2$ be a rectangular channel $(-3,+3)\times(-2,+2)$ with a circular obstacle at $(0,0)$ of radius $1/2$ as shown in Figure~\ref{fig:stokes-geometry}.
We denote the part of the channel without the obstacle by $\Omega := (-3,+3)\times(-2,+2) \setminus \overline{B_{1/2}(0)}$.
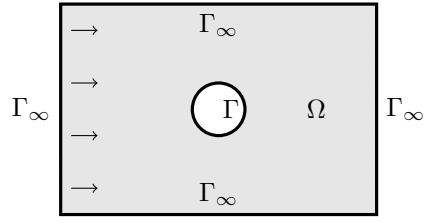
\begin{figure}[H]
    \centering
    \begin{tikzpicture}[scale=0.7]
\draw[very thick, fill=black!10] (-3, -2) rectangle (3,2);
\coordinate (bl) at (-3,-2);
\coordinate (br) at (+3,-2);
\coordinate (tl) at (-3,+2);
\coordinate (tr) at (+3,+2);
\draw[thick] (bl) --node[above]{$\Gamma_\infty$} (br);
\draw[thick, dashed] (br) --node[right]{$\Gamma_\infty$} (tr);
\draw[thick] (tr) --node[below]{$\Gamma_\infty$} (tl);
\draw[thick, dashed] (tl) --node[left]{$\Gamma_\infty$} (bl);
\coordinate (cntr) at (0,0);
\draw[very thick, fill=white] (cntr) circle (0.5);
\node[right] at (-0.1, 0) {$\Gamma$};
\node[right] at (1.5, 0) {$\Omega$};
\draw[->] (-2.8, 1.5) -- (-2.3, 1.5);
\draw[->] (-2.8, .5) -- (-2.3, .5);
\draw[->] (-2.8, -.5) -- (-2.3,-.5);
\draw[->] (-2.8, -1.5) -- (-2.3,-1.5);
\end{tikzpicture}
    \caption{Computational domain for the BVP \eqref{eqn:bvp-stokes}.}
    \label{fig:stokes-geometry}
\end{figure}
Given a far-field velocity $\Vu_\infty$, the velocity field $\Vu$ and the pressure $p$ are then governed by the following boundary value problem:
\begin{equation}\label{eqn:bvp-stokes}
    \begin{aligned}
        - \Delta \Vu + \nabla p &=0 && \text{in } \Omega,\\
        \mathrm{div}(\Vu) &= 0 && \text{in } \Omega,\\
        \Vu &= \Vu_{\infty} && \text{on } \Gamma_\infty,\\
        \Vu &= 0 && \text{on } \Gamma.
    \end{aligned}
\end{equation}
We choose these boundary conditions over a more natural outflow condition for comparability with existing literature on energy minimisation in Stokes flow \cite{pironneau1974optimum, schulz2016computational}.
The corresponding weak form is given by: find $p\in L^2(\Omega)$ with $\int_{\Omega} p\dr x=0$ and $\Vu \in [H^1(\Omega)]^2$ with $\Vu\vert_{\Gamma_\infty} = \Vu_\infty$ and $\Vu\vert_{\Gamma} = 0$ such that 
\begin{equation}\label{eq:stokes_weak}
    \int_{\Omega} \partial \Vu : \partial \Vv - p\,\mathrm{div}\Vv + q\,\mathrm{div}\Vu \dr x = 0 
\end{equation}
for all $q\in L^2(\Omega)$ and $\Vv\in [H^1(\Omega)]^2$ with $\Vv\vert_{\Gamma_\infty\cup\Gamma}=0$.
The dissipated energy of \eqref{eq:stokes_weak} can then be calculated by the integral 
\begin{equation}\label{eq:cost_stokes}
    J(\Omega) =  \frac12 \int_{\Omega} \partial \Vu: \partial \Vu \dr x.
\end{equation}

\begin{proposition}
	The shape derivative of $J$ at $\Omega$ in direction $\VX\in [C^{0,1}(\overbar \Omega)]^2$, $\VX=0$ on $\Gamma_\infty$ is given by 
	\begin{equation}\label{eq:shape_derivative_stokes}
	\Dr J(\Omega)(\VX) = \int_{\Omega} \VS_1:\partial \VX \dr x, 
	\end{equation}
	where $(\Vu,p)$ solve \eqref{eq:stokes_weak} and $\VS_1$ is given by
	\begin{equation}\label{eq:formula_S1}
	\VS_1 = (\frac12\partial\Vu:\partial\Vu -p \Div(\Vu) ) \VI_2 + \partial\Vu^\top p - \partial \Vu^\top \partial \Vu.
	\end{equation}
\end{proposition}
\begin{proof}
Let $\VX\in [C^{0,1}(\overbar \Omega)]^2$ with $\VX=0$ on $\Gamma_{\infty}$ be a given vector field . We set $\VT_t := \id + t\VX$ and $\Omega_t := \VT_t(\Omega)$ and denote by $(\Vu_t, p_t)$ the solution of \eqref{eq:stokes_weak} where $\Omega$ is replaced by $\Omega_t$:
 $p_t\in L^2(\Omega_t)$, $\int_{\Omega_t} p_t \dr x=0$ and $\Vu_t \in [H^1(\Omega_t)]^2$ with $\Vu_t\vert_{\Gamma_\infty} = \Vu_\infty$ and $\Vu_t\vert_{\Gamma} = 0$ such that 
 \begin{equation}\label{eq:stokes_weak_per}
 \int_{\Omega_t} \partial \Vu_t : \partial \Vv - p_t\,\mathrm{div}\Vv + q\,\mathrm{div}\Vu_t \dr x = 0
 \end{equation}
for all $q\in L^2(\Omega_t)$ and $\Vv\in [H^1(\Omega_t)]^2$. Invoking a change of variables in \eqref{eq:stokes_weak_per} shows that $(\Vu^t,p^t) := (\Vu_t\circ \VT_t, p_t\circ \VT_t)$ satisfy
 \begin{equation}\label{eq:stokes_weak_per_ut}
 \int_{\Omega_t} \Det(\partial \VT_t)  \left(\partial \VT_t^{-1}\partial \Vu^t : \partial \VT_t^{-1}\partial \Vv - p^t\,\tr(\partial\Vv \partial \VT_t^{-1}) + q\,\tr(\partial \Vu^t \partial \VT_t^{-1})\right) \dr x = 0
 \end{equation}
 for all $q\in L^2(\Omega)$ and $\Vv\in [H^1(\Omega)]^2$. Here we used that $(\partial \Vv)\circ \VT_t = \partial (\Vv\circ \VT_t)(\partial \VT_t)^{-1}$ and 
 $\Div(\Vv)\circ \VT_t = \tr(\partial(\Vv\circ \VT_t)(\partial \VT_t)^{-1})$.
 Hence the parametrised Lagrangian associated with \eqref{eq:stokes_weak_per_ut} reads
 \begin{equation}
 \begin{split}
 G(t, \Vv, q) = & \frac12\int_{\Omega} \Det(\partial \VT_t) \partial \Vv(\partial \VT_t)^{-1}:\partial \Vv (\partial \VT_t)^{-1} \; \dr x\\
 & - \int_{\Omega} \Det(\partial \VT_t)  q \tr(\partial\Vv (\partial \VT_t)^{-1})\; \dr x.
 \end{split}
 \end{equation} 
Observe that taking the partial derivative of $G(t,\cdot,\cdot)$ gives the system \eqref{eq:stokes_weak_per_ut}.   Now using one of the many available methods to differentiate parametrised Lagrangians, see e.g. \cite{sturm2, ITKUGU08}, and \cite{sturm4} for an overview, we can show that
 \begin{equation}
 \Dr J (\Omega)(\VX) = \frac{d}{dt}G(t,\Vu^t,0)\Bigr|_{t=0} = \partial_t G(0, \Vu,p). 
 \end{equation}
 It can readily be checked that the derivative 
 $\partial_t G(0,\Vu, p)$ coincides indeed with the right hand side of \eqref{eq:shape_derivative_stokes}. 
 
\end{proof}

\begin{remark}
    As shown in \cite{LaSt16} assuming that $\Vu$ and $p$ are smooth enough, 
    we can retrieve the boundary form of the shape derivative from \eqref{eq:shape_derivative_stokes} by
    \ben
        \Dr J(\Omega)(\VX) = \int_{\Gamma} \VS_1\nubf \cdot \nubf (\VX\cdot \nubf) \dr s. 
    \een
Since $\Vu=0$ on $\Gamma$ the tangential gradient of $u_i$ vanishes on $\Gamma$ and hence 
we have $\nabla u_i\cdot \nabla u_i = \partial_{\nubf} u_i \partial_{\nubf} u_i$. Moreover 
we have $\Div^\tau(\Vu) := \Div(\Vu) - \partial \Vu\nubf \cdot \nubf=0$ on $\Gamma$. 
Hence recalling $\partial \Vu^\top \partial \Vu = \nabla u_1\otimes \nabla u_1 + \nabla u_2 \otimes \nabla u_2$, we obtain 
using \eqref{eq:formula_S1},
\ben
\begin{split}
    \VS_1\nubf \cdot \nubf & = \frac12 (\nabla u_1 \cdot \nabla u_1 + \nabla u_2 \cdot \nabla u_2)  - p \underbrace{\Div^\tau(\Vu)}_{=0} - (\nabla u_1\otimes \nabla u_1 + \nabla u_2 \otimes \nabla u_2)\nubf\cdot\nubf\\
& =-\frac12 ( (\partial_{\nubf} u_1)^2 + (\partial_{\nubf} u_2)^2).
\end{split}
\een
This boundary form coincides with the one stated in \cite[Prop.~2.13]{mohammadi2010applied}.
\end{remark}

We require the volume and the barycentre of the obstacle $\Omega_o := D\setminus\overbar\Omega$ to remain constant; these constraints are enforced using an augmented Lagrangian method.
For a more detailed treatment of the Stokes problem with these constraints we refer to \cite{schulz2016computational, paganini2017higher}.

The optimal shape for this problem is the well known ogive \cite{pironneau1974optimum} and depicted in Figure~\ref{fig:optimal-shape-stokes}.
Mesh deformation methods often struggle with this problem at the leading and trailing edge, as creating the sharp tips represents a significant deformation.
Due to this sharp corner we do not consider RKHS for this example, as these functions are differentiable and thus not appropriate here.
We define $\mu$ as before in \eqref{eqn:weighting-function-distance} and consider the inner-product
\begin{equation}
    (\Vu,\Vv)_{\CR(\alpha,\mu)+H(\sym)} \defeq \frac1\alpha(\mu\Cb\Vu,\mu\Cb\Vv)_{[L^2]^2} + (\sym(\partial\Vu), \sym(\partial\Vv))_{[L^2]^{2\times2}}.
\end{equation}
The deformations are discretised using P1 finite elements;
furthermore, we discretise the state equation using a P1 element for the pressure field and a mini-element for the velocity field \cite[p.~322]{girault2012finite}.

The obtained flow-fields for the initial and the optimal shape are presented in Figure~\ref{fig:optimal-shape-stokes} and the final mesh is shown in Figure~\ref{fig:optimal-shape-stokes-mesh}.
Due to Lemma~\ref{lem:non-existence-clamped-cr} we know that there is no conformal mapping between the initial and the optimal shape.
However, using the weighting function as proposed in \eqref{eqn:weighting-function-distance} we are able to distribute the non-conformality over the entire volume of the mesh and hence still obtain a very high mesh quality.
As in the levelset example, the nearly conformal mappings achieve this by shrinking the elements where the deformation is largest; this can be seen especially well near the two tips.
\begin{figure}[htbp]
    \centering
    \includegraphics[width=0.49\textwidth]{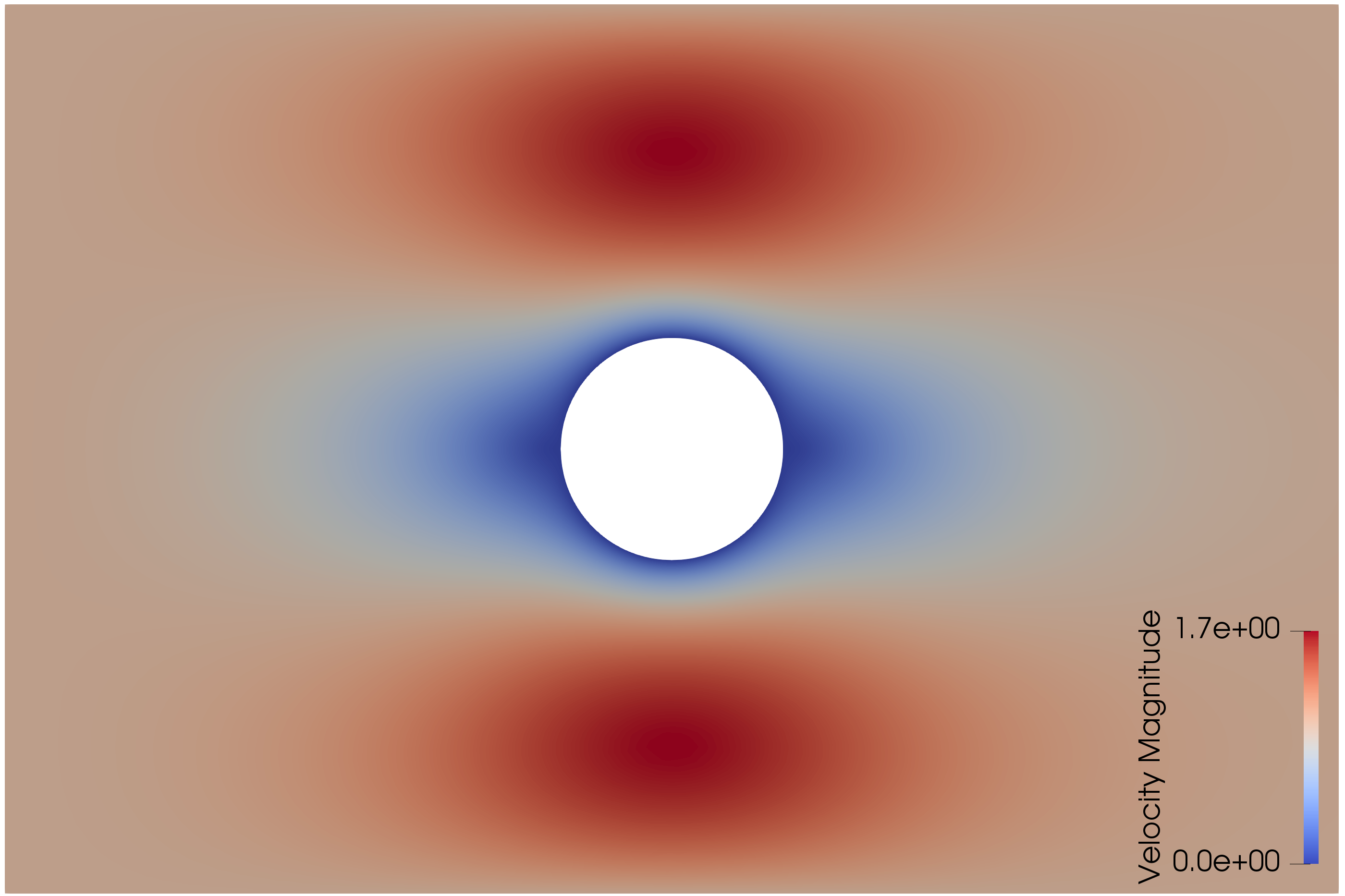}
    \includegraphics[width=0.49\textwidth]{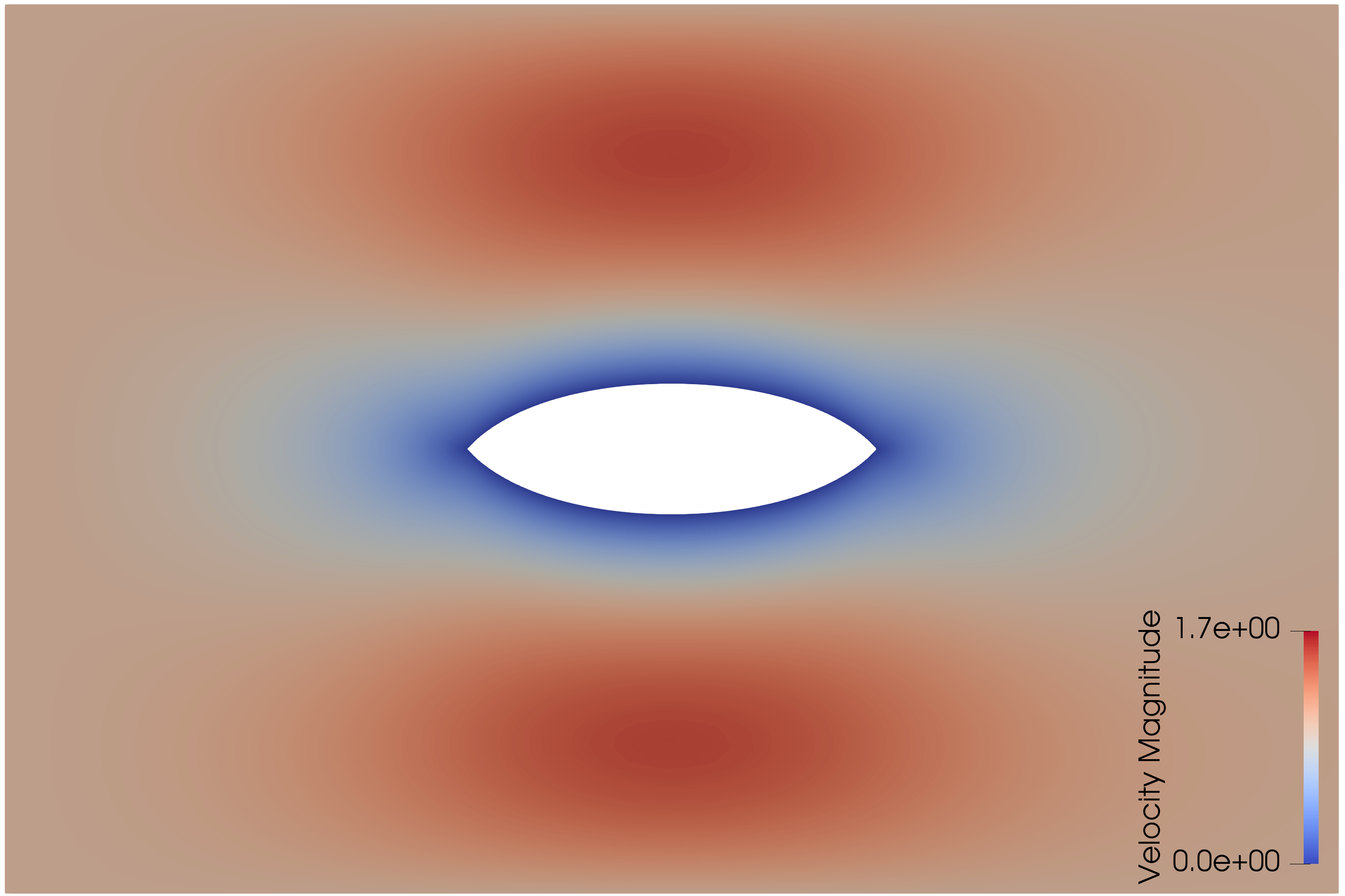}
    \caption{Initial shape and optimal shape for the Stokes energy minimisation problem subject to volume and barycentre constraints.}
    \label{fig:optimal-shape-stokes}
\end{figure}

\begin{figure}[htbp]
    \centering
    \includegraphics[width=0.49\textwidth]{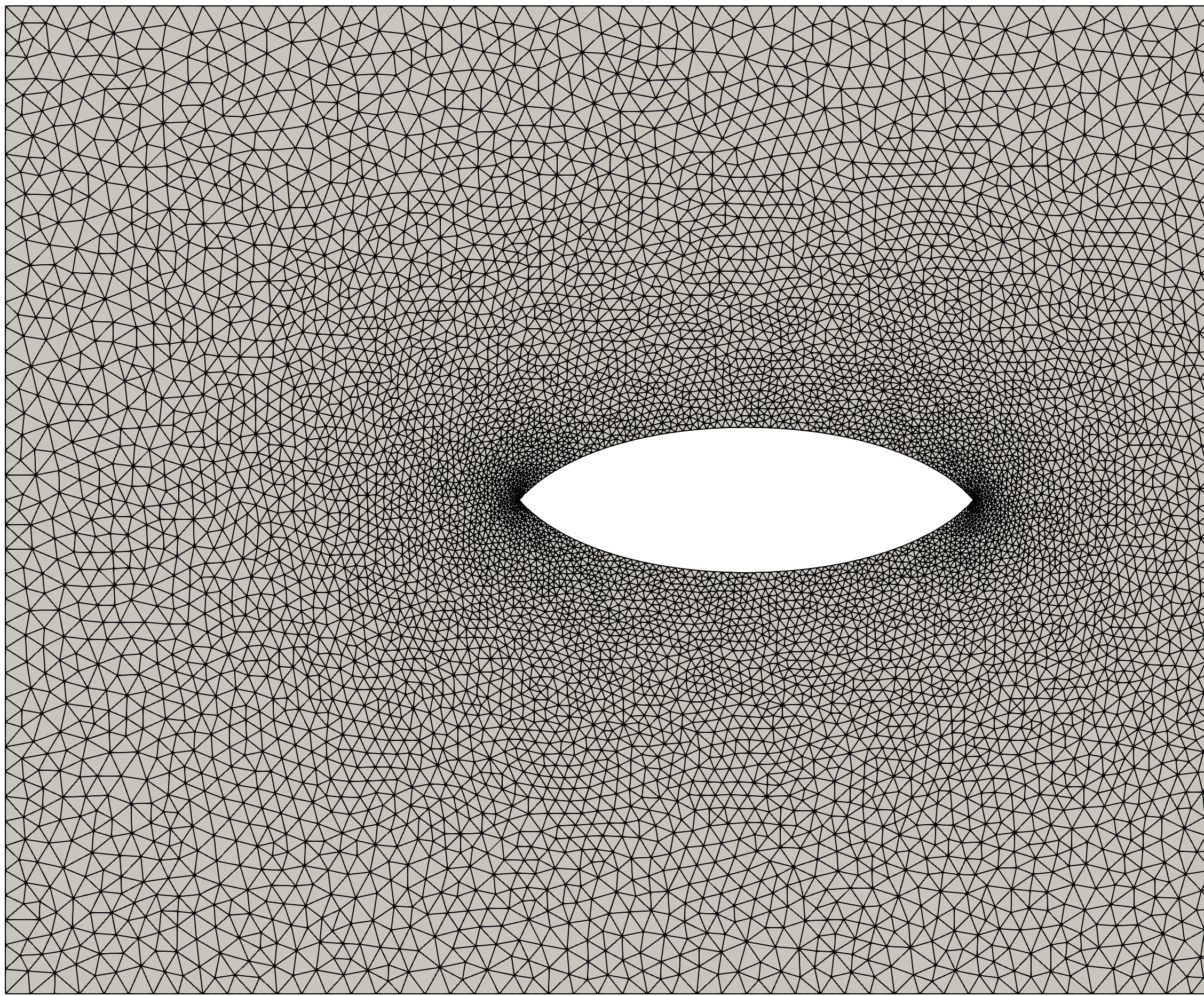}
    \includegraphics[width=0.49\textwidth]{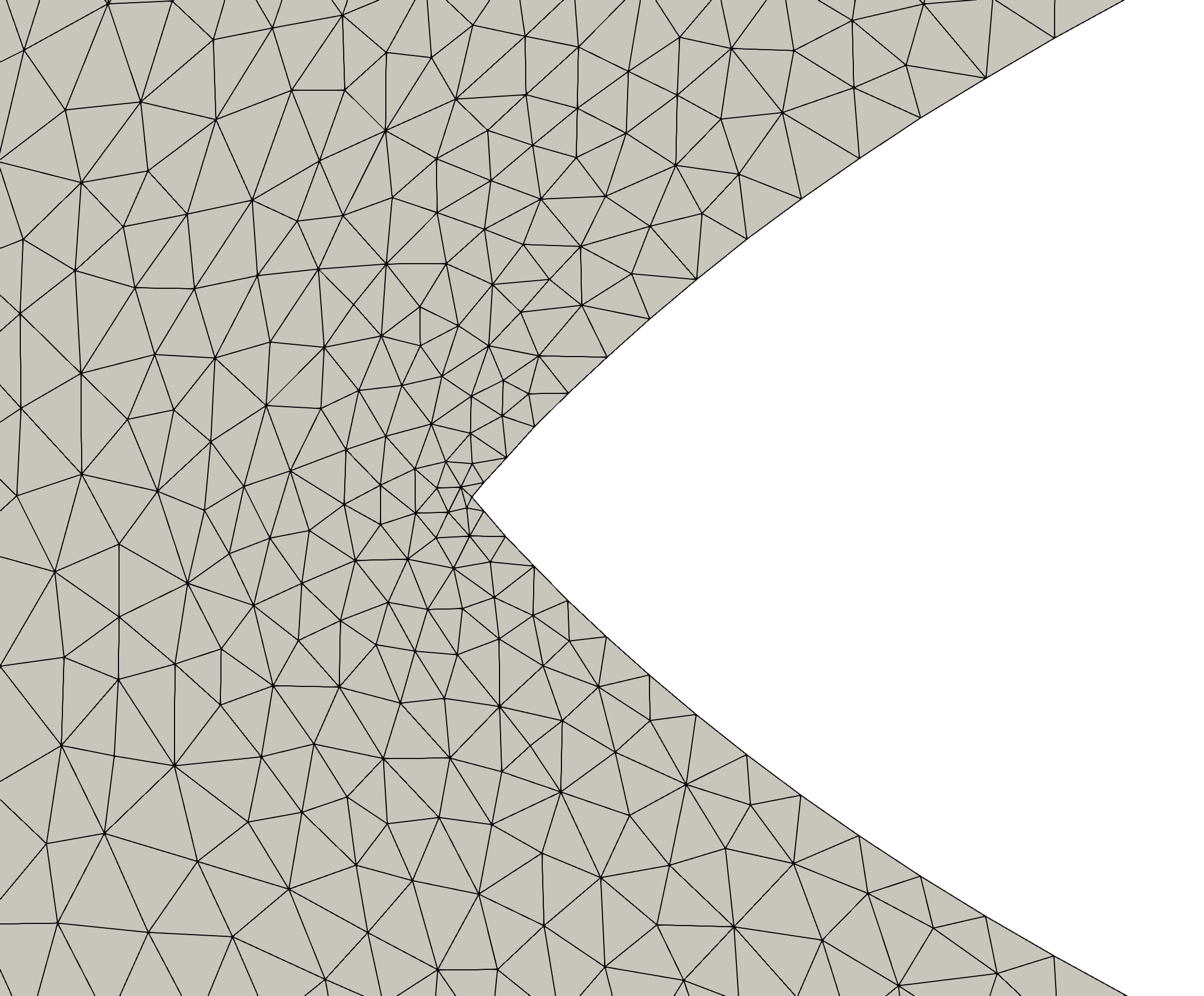}
    \caption{Mesh obtained using the inner-product $(\cdot,\cdot)_{\CR(10^{-3},\mu)+\mathring{H}(\sym)}$; Left: zoom out of the mesh; Right: close-up of the tip of the rugby shape.}
    \label{fig:optimal-shape-stokes-mesh}
\end{figure}
\FloatBarrier

\section{Summary and future work}

In this work we illustrate how to use conformal mappings for shape optimisation. 
Our approach is based on modifying the existing inner-product on a Hilbert space which yields shape gradients that are nearly holomorphic.
We do this for both the classical $[H^1(\Omega)]^2$ Sobolev space, as well as for reproducing kernel Hilbert spaces $[\Ch(\Omega)]^2$.

Our numerical examples suggest that if a conformal mapping between the initial and the optimal shape exists, then the gradients based on the new inner-products lead to deformations that have essentially no influence on the mesh quality.
We observe that in that case the proposed method significantly outperforms gradients purely based on the standard inner-product in $[H^1(\Omega)]^2$, elasticity based inner-products or the inner-product induced by a reproducing kernel.

In the case when there is no such mapping, we propose a weighting function which leads to deformations that emphasize mesh quality near the boundary.
We show this for the classical Stokes energy minimisation example.

We consider only two dimensional problems in this paper as the Cauchy-Riemann equations are an inherently two dimensional concept.
However, we briefly outline how one could attempt to extend the work presented here to higher dimensions.
In general, a mapping $\VT:\VR^d\to\VR^d$ is angle preserving if it satisfies  the system of  nonlinear equations \cite[Section 2.1]{iwaniec2001geometric}
\begin{equation}\label{eqn:conformal-higher-dim}
    \Dr\VT^\top\Dr\VT = (\mathrm{det}(\Dr \VT))^{2/d} \VI_d.
\end{equation}
One can show (see, e.g., \cite{iwaniec2001geometric})  that this equation reduces to the Cauchy-Riemann equations when $d=2$. 
For $d\ge 3$ however, this equation is highly non-linear and furthermore overdetermined, which in turn implies that the exact solutions for \eqref{eqn:conformal-higher-dim} are given by M\"obius transformations.
The powerful statement of the Riemannian mapping theorem hence does not extend to higher dimensions. 
We can however still hope to obtain nearly conformal mappings. 
This could be done by either adding a term of the form $\|\Dr\VT^\top\Dr\VT - (\mathrm{det}(\Dr \VT))^{2/d} \VI_d\|$ to the shape function or by linearizing \eqref{eqn:conformal-higher-dim} around $\VT=\id$ and then including it into the inner-product.
We leave this as work for future research.
\section*{Acknowledgement}
The authors thank Patrick Farrell for helpful discussions and his detailed comments on the manuscript.
\newpage

\section*{Appendix}
\subsection{Traces of functions in vvRKHS}	
Let $\Omega_1\subset \Omega_2\subset \VR^2$ be two sets. Suppose that $\ksf(\Vx,\Vy)=\phi(|\Vx-\Vy|)$ is a positive definite kernel, where $\phi\in C(\Omega_2\times \Omega_2)$. In \cite[Thm.~10.46, p.~169]{Wendlandbook} it is shown that there is an extension operator $E: \Ch(\Omega_1) \to \Ch(\Omega_2)$ having the property $\|Ef\|_{\Ch(\Omega_2)} = \|f\|_{\Ch(\Omega_1)}$. 

Conversely, in \cite[Thm.~10.47, p.~170]{Wendlandbook} it is shown that the restriction of $f\in \Ch(\Omega_2)$ to $\Omega_2$ belongs to $\Ch(\Omega_1)$ and that $\|f|_{\Omega_1}\|_{\Ch(\Omega_1)}\le \|f\|_{\Ch(\Omega_2)}$. 

With these two ingredient we define the trace operator as follows. 
\begin{definition}
The trace operator $T:\Ch(\Omega) \to \Ch(\partial\Omega)$ is defined by 
 $T:= R\circ E$, where $E:\Ch(\Omega) \to \Ch(\overbar\Omega)$ denotes the extension operator and $R:\Ch(\overbar\Omega) \to \Ch(\partial\Omega)$ denotes
 the restriction operator. 
\end{definition}

\begin{remark}
    Since the trace operator is continuous and linear it is also weakly continuous; see \cite[Thm.~2.10, p.~61]{BR2010}
\end{remark}

\subsection{Proof of Proposition~\ref{prop:solution-decomposition-RKHS}}

\begin{proof} 
    (i)    First observe that the first order optimality condition is equivalent to: find $\Vu_\alpha^{(1)}\in V$ and $\Vu_\alpha^{(2)}\in V^\bot$, such that
    \ben\label{eq:opt_cond_rkhs}
\begin{split}
  \frac1n \sum_{\ell=1}^n\Cb \Vu_\alpha^{(2)}(\Vp_\ell)\cdot \Cb\Vv^{(2)}(\Vp_\ell) & + \alpha(\Vu_\alpha^{(1)},\Vv^{(1)})_{[\Ch]^2} + \alpha(\Vu_\alpha^{(2)},\Vv^{(2)})_{[\Ch]^2}\\
  &= \alpha (\varphibf_{\Vg}^{(1)},\Vv^{(1)})_{[\Ch]^2} + \alpha (\varphibf_{\Vg}^{(2)},\Vv^{(2)})_{[\Ch]^2}
\end{split}
\een
for all $\Vv^{(1)}\in V$ and all  $\Vv^{(2)} \in V^\perp$.  Testing  \eqref{eq:opt_cond_rkhs} with $\Vv^{(2)}=0$, we get
\ben
(\Vu_\alpha^{(1)}-\varphibf_{\Vg}^{(1)},\Vv^{(1)})_{[\Ch]^2}=0 \quad \text{ for all } \Vv^{(1)}\in V.
\een
As a result $\Vu_\alpha^{(1)}-\varphibf_{\Vg}^{(1)}\in V^\bot\cap V = \{0\}$ and it follows $\Vu_\alpha^{(1)}=\varphibf_{\Vg}^{(1)}$. 

(ii)-(iii)
Using  $\Vv^{(1)}=0$ as test function in \eqref{eq:opt_cond_rkhs} and taking into account \eqref{eq:grad_RKHS}, we obtain \eqref{eq:u2-characterisation_H_RKHS}. Moreover, using $\Vv^{(2)}=\Vu_\alpha^{(2)}$ as test function in \eqref{eq:u2-characterisation_H_RKHS} and estimating the right hand side yields \eqref{eq:estimate_kernel}.

(iv) By (ii), for every null-sequence $(\alpha_n)$, there exists $\Vu_0^{(2)}\in V^\bot$ and a subsequence $(\alpha_{n_k})$ of $(\alpha_n)$ such that $\Vu_{\alpha_{n_k}}^{(2)} \rightharpoonup \Vu_0^{(2)}$ weakly in $[\Ch(\Omega)]^2$. Therefore using \eqref{eq:u2-characterisation_H_RKHS} we find 
\ben
\begin{split}
\frac1n\sum_{\ell=1}^n |\Cb\Vu_0^{(2)}(\Vp_\ell)|^2 & =  \lim_{k\to \infty} \frac1n\sum_{\ell=1}^n \Cb\Vu_\alpha^{(2)}(\Vp_\ell)\cdot \Cb\Vu^{(2)}_0(\Vp_\ell)\\
 &= \lim_{k\to \infty} - \alpha_{n_k} ( \Vu_{\alpha_{n_k}}^{(2)}, \Vu^{(2)}_0)_{[\Ch]^2} + \alpha_{n_k} \int_{\partial \Omega} \Vg\cdot \Vu^{(2)}_0 \dr s = 0.
\end{split}
\een
This shows that $\Cb\Vu_0^{(2)}(\Vp_\ell)=0$ for $\ell=1,\ldots, n$, which means $\Vu_0^{(2)}\in V$, but since $\Vu_0^{(2)}\in V^\bot$ we conclude $\Vu_0^{(2)}=0$. Since the null-sequence $(\alpha_n)$ was arbitrary we get $\Vu_\alpha^{(2)} \rightharpoonup 0$ in $[\Ch(\Omega)]^2$.  

Testing \eqref{eq:u2-characterisation_H_RKHS} with $\Vv^{(2)} = \Vu_\alpha^{(2)}$ we obtain 
\ben\label{eq:estimate_B_kernel}
\lim_{\alpha \to 0}\frac1n\sum_{\ell=1}^n |\Cb\Vu_\alpha^{(2)}(\Vp_\ell)|^2 =  \lim_{\alpha\to 0} - \alpha ( \Vu_{\alpha}^{(2)}, \Vu^{(2)}_\alpha)_{[\Ch]^2} + \alpha \int_{\partial \Omega} \Vg\cdot \Vu^{(2)}_\alpha \dr s = 0,
\een
since $\Vu_\alpha^{(2)}$ is bounded in $[\Ch(\Omega)]^2$. This shows that $ |\Cb\Vu_\alpha^{(2)}(\Vp_\ell)| \to 0$ for $\ell=1,\ldots, n$ as $\alpha\to0$. 

Testing again \eqref{eq:u2-characterisation_H_RKHS} with $\Vv^{(2)} = \Vu_\alpha^{(2)}$ and estimating the left hand side yields
\ben\label{eq:estimate_B_kernel-2}
\begin{split}
\alpha \frac1n\sum_{\ell=1}^n |\Cb\Vu_\alpha^{(2)}(\Vp_\ell)|^2 + \alpha\|\Vu_\alpha^{(2)}\|_{[\Ch]^2}^2 & \le \frac1n\sum_{\ell=1}^n |\Cb\Vu_\alpha^{(2)}(\Vp_\ell)|^2 + \alpha\|\Vu_\alpha^{(2)}\|_{[\Ch]^2}^2\\
& = \alpha  \int_{\partial \Omega} \Vg\cdot \Vu_\alpha^{(2)}\dr s
\
\end{split}
\een
which in view of \eqref{eq:estimate_B_kernel} shows that
\ben
 \lim_{\alpha \to 0}\|\Vu_\alpha^{(2)}\|_{[\Ch]^2}^2   \le  - \lim_{\alpha \to 0} \frac1n\sum_{\ell=1}^n |\Cb\Vu_\alpha^{(2)}(\Vp_\ell)|^2  +  \lim_{\alpha \to 0} \int_{\partial \Omega} \Vg\cdot \Vu_\alpha^{(2)}\dr s = 0.
\een
This shows the strong convergence of  $(\Vu_\alpha^{(2)})$ to zero and hence $\Vu_\alpha$ converges strongly to $\varphibf_\Vg^{(1)}$ as $\alpha \to0$.
It also readily follows from \eqref{eq:estimate_B_kernel-2} that 
$\lim_{\alpha \to0}\max_{\Vp\in \Cp_n} \frac{1}{\sqrt\alpha} |\Cb\Vu_\alpha^{(2)}(\Vp)| =0$. 

(v) 

Recall that $\Vu_\alpha^{(2)}$ is bounded in $V^\bot$ uniformly in $\alpha$ and hence for every sequence $(\alpha_n)$ with $\alpha_n\to \infty$,  we find $\overbar{\Vu}^{(2)}\in V^\bot$ and a subsequence $(\alpha_{n_k})$ of $(\alpha_n)$ such that $\Vu_{\alpha_{n_k}}^{(2)}\to \overbar{\Vu}$. 
Hence choosing $\alpha=\alpha_{n_k}$ in \eqref{eq:u2-characterisation_H_RKHS} and then passing to the limit $k\to\infty$ yields
\ben\label{eq:u_infty}
(\bar\Vu, \Vv^{(2)})_{[\Ch]^2} = \int_{\partial \Omega} \Vg\cdot \Vv^{(2)}\dr s
\een
for all $\Vv^{(2)} \in V^\bot$. Since \eqref{eq:u_infty} admits a unique solution we conclude
that $\bar \Vu = \varphibf_\Vg^{(2)}$. 
\end{proof}

\bibliographystyle{siamplain}
\bibliography{refs}

\begin{thebibliography}{10}

\bibitem{BE05}
{\sc H.~Begehr}, {\em {Boundary value problems in complex analysis I}},
  Bolet\'in de la Asociaci\'on Matem\'etica Venezolana, XII (2005).

\bibitem{BR2010}
{\sc H.~Brezis}, {\em Functional analysis, Sobolev spaces and partial
  differential equations}, Universitext, Springer New York, 2010.

\bibitem{CO78}
{\sc J.~Conway}, {\em Functions of one complex variable I}, Functions of one
  complex variable / John B. Conway, Springer, 1978.

\bibitem{DEZO11}
{\sc M.~C. Delfour and J.-P. Zol\'esio}, {\em Shapes and geometries: metrics,
  analysis, differential calculus, and optimization}, Society for Industrial
  and Applied Mathematics (SIAM), second~ed., 2011.

\bibitem{DemDem12}
{\sc F.~Demengel and G.~Demengel}, {\em Functional spaces for the theory of
  elliptic partial differential equations}, Universitext, Springer, London; EDP
  Sciences, Les Ulis, 2012.

\bibitem{dwight2009robust}
{\sc R.~P. Dwight}, {\em {Robust mesh deformation using the linear elasticity
  equations}}, Computational Fluid Dynamics 2006,  (2009), pp.~401--406.

\bibitem{EigelSturm17}
{\sc M.~Eigel and K.~Sturm}, {\em {Reproducing kernel Hilbert spaces and
  variable metric algorithms in PDE-constrained shape optimization}},
  Optimization Methods and Software,  (2017), pp.~1--29.

\bibitem{ESSI2007}
{\sc K.~Eppler, S.~Schmidt, V.~Schulz, and C.~Ilic}, {\em Preconditioning the
  pressure tracking in fluid dynamics by shape {H}essian information}, Journal
  of Optimization Theory and Applications, 141 (2009), pp.~513--531.

\bibitem{gamelin2003complex}
{\sc T.~Gamelin}, {\em Complex analysis}, Springer Science \& Business Media,
  2003.

\bibitem{girault2012finite}
{\sc V.~Girault and P.-A. Raviart}, {\em Finite element methods for
  Navier-Stokes equations: theory and algorithms}, vol.~5, Springer Science \&
  Business Media, 2012.

\bibitem{GraMor78}
{\sc J.~D. Gray and S.~A. Morris}, {\em When is a function that satisfies the
  {C}auchy-{R}iemann equations analytic?}, Amer. Math. Monthly, 85 (1978),
  pp.~246--256.

\bibitem{HE03}
{\sc B.~T. Helenbrook}, {\em Mesh deformation using the biharmonic operator},
  International Journal for Numerical Methods in Engineering, 56 (2003),
  pp.~1007--1021.

\bibitem{ITKUGU08}
{\sc {Ito, Kazufumi}, {Kunisch, Karl}, and {Peichl, Gunther H.}}, {\em
  Variational approach to shape derivatives}, ESAIM: Control, Optimisation and
  Calculus of Variations, 14 (2008), pp.~517--539.

\bibitem{IWKOON11}
{\sc T.~Iwaniec, L.~V. Kovalev, and J.~Onninen}, {\em The {N}itsche
  conjecture}, J. Amer. Math. Soc., 24 (2011), pp.~345--373.

\bibitem{iwaniec2001geometric}
{\sc T.~Iwaniec and G.~Martin}, {\em Geometric function theory and non-linear
  analysis}, Clarendon Press, 2001.

\bibitem{KAetal2016}
{\sc K.~G. K., P.-K.~E. M., S.~T., V.~E. de, G.~K. C., and O.~C.}, {\em Adjoint
  optimization for vehicle external aerodynamics}, International Journal of
  Automotive Engineering, 7 (2016), pp.~1--7.

\bibitem{KonOle88}
{\sc V.~A. Kondrat'ev and O.~A. Ole{\u i}nik}, {\em Boundary value problems for
  a system in elasticity theory in unbounded domains. {K}orn inequalities},
  Russian Math. Surveys, 43 (1988), pp.~65--119,
  \url{https://doi.org/10.1070/RM1988v043n05ABEH001945}.

\bibitem{LaSt16}
{\sc A.~Laurain and K.~Sturm}, {\em Distributed shape derivative via averaged
  adjoint method and applications}, ESAIM: M2AN, 50 (2016), pp.~1241--1267.

\bibitem{mohammadi2010applied}
{\sc B.~Mohammadi and O.~Pironneau}, {\em Applied shape optimization for
  fluids}, Oxford university press, 2010.

\bibitem{paganini2017higher}
{\sc A.~Paganini, F.~Wechsung, and P.~E. Farrell}, {\em {Higher-order moving
  mesh methods for PDE-constrained shape optimization}}, arXiv preprint
  arXiv:1706.03117,  (2017).

\bibitem{pironneau1974optimum}
{\sc O.~Pironneau}, {\em On optimum design in fluid mechanics}, Journal of
  Fluid Mechanics, 64 (1974), pp.~97--110.

\bibitem{Rathgeber2016}
{\sc F.~Rathgeber, D.~A. Ham, L.~Mitchell, M.~Lange, F.~Luporini, A.~T.~T.
  Mcrae, G.-T. Bercea, G.~R. Markall, and P.~H.~J. Kelly}, {\em Firedrake:
  Automating the finite element method by composing abstractions}, ACM
  Transactions on Mathematical Software, 43 (2016), pp.~24:1--24:27,
  \url{https://arxiv.org/abs/1501.01809}.

\bibitem{ridzal2014rapid}
{\sc D.~Ridzal and D.~P. Kouri}, {\em Rapid optimization library.}, tech.
  report, Sandia National Laboratories (SNL-NM), Albuquerque, NM (United
  States), 2014.

\bibitem{PhDSchmidtShape}
{\sc S.~Schmidt}, {\em Efficient large scale aerodynamic design based on shape
  calculus}, PhD thesis, University of Trier, Germany, 2010,
  \url{http://ubt.opus.hbz-nrw.de/frontdoor.php?source_opus=569&la=en}.

\bibitem{schmidt2014twostage}
{\sc S.~Schmidt}, {\em {A two stage CVT / Eikonal convection cesh deformation
  approach for large nodal deformations}},  (2014),
  \url{https://arxiv.org/abs/1411.7663}.

\bibitem{SCSCILGA2011}
{\sc S.~Schmidt, V.~Schulz, C.~Ilic, and N.~Gauger}, {\em Large scale
  aerodynamic shape optimization based on shape calculus}, in
  ECCOMAS-CFD\&Optimization, I.~H. Tuncer, ed., 2011-054, ISBN
  978-605-61427-4-1, 2011.

\bibitem{schottky1877konforme}
{\sc F.~Schottky}, {\em {Ueber die conforme Abbildung mehrfach
  zusammenhängender ebener Flächen}}, Journal für die reine und angewandte
  Mathematik, 83 (1877).

\bibitem{schulz2016computational}
{\sc V.~Schulz and M.~Siebenborn}, {\em {Computational comparison of surface
  metrics for PDE constrained shape optimization}}, {Computational Methods in
  Applied Mathematics}, 16 (2016), pp.~485--496.

\bibitem{SMER87}
{\sc R.~E. Smith and L.-E. Eriksson}, {\em Algebraic grid generation}, Computer
  Methods in Applied Mechanics and Engineering, 64 (1987), pp.~285 -- 300.

\bibitem{SOZO92}
{\sc J.~Soko{\l}owski and J.-P. Zol{\'e}sio}, {\em Introduction to shape
  optimization. Shape sensitivity analysis}, Springer, Berlin, 1992.

\bibitem{ST93}
{\sc S.~Steinberg}, {\em Fundamentals of grid generation}, Taylor \& Francis,
  1993.

\bibitem{sturm2}
{\sc K.~Sturm}, {\em Minimax {L}agrangian approach to the differentiability of
  nonlinear {PDE} constrained shape functions without saddle point assumption},
  SIAM Journal on Control and Optimization, 53 (2015), pp.~2017--2039.

\bibitem{sturm4}
{\sc K.~Sturm}, {\em Shape differentiability under non-linear {PDE}
  constraints}, in New trends in shape optimization, vol.~166 of International
  Series of Numerical Mathematics, Birkh\"auser/Springer, Cham, 2015,
  pp.~271--300.

\bibitem{Wendlandbook}
{\sc H.~Wendland}, {\em Scattered data approximation}, Cambridge University
  Press, Cambridge, 2005.

\bibitem{ZOHO89}
{\sc A.~Zochowski and P.~Holnicki}, {\em Interactive method of variational grid
  generation}, Journal of Computational and Applied Mathematics, 26 (1989),
  pp.~281 -- 287.

\end{thebibliography}

\end{document}